\documentclass[12pt]{article}
\usepackage{amsthm,amssymb,amsfonts,amsmath}
\usepackage{appendix}
\usepackage[english]{babel}
\usepackage{nicefrac}  
\usepackage{fancyhdr} 

\usepackage{setspace}
\usepackage{enumerate} 

\usepackage{tocloft}
\usepackage{dsfont}        
\usepackage[colorlinks=true]{hyperref}

\DeclareMathAlphabet{\mathpzc}{OT1}{pzc}{m}{it}

\newcommand{\bi}{{\boldsymbol i}}

\newcommand{\cA}{{\mathcal A}}
\newcommand{\cB}{{\mathcal B}}
\newcommand{\cC}{{\mathcal C}}

\newcommand{\cE}{{\mathcal E}}
\newcommand{\cF}{{\mathcal F}}
\newcommand{\cG}{{\mathcal G}}

\newcommand{\cI}{{\mathcal I}}
\newcommand{\cL}{\mathcal L}
\newcommand{\cO}{{\mathcal O}}

\newcommand{\cS}{{\mathcal S}}
\newcommand{\cT}{{\mathcal T}}
\newcommand{\cV}{{\mathcal V}}

\newcommand\cW{{\mathcal W}}

\newcommand{\bC}{{\mathbb C}}

\newcommand{\bF}{{\mathbb F}}

\newcommand{\bH}{{\mathbb H}}
\newcommand{\bN}{{\mathbb N}}
\newcommand{\bM}{{\mathbb M}}

\newcommand{\bR}{{\mathbb R}}

\newcommand{\bT}{{\mathbb T}}
\newcommand{\bZ}{{\mathbb Z}}

\newcommand{\ve}{\varepsilon}

\newcommand{\vf}{\varphi}
\newcommand{\olambda}{\overline{\lambda}}

\newcommand{\bRp}{\bR_+}  
\newcommand{\Cs}{C}
\newcommand{\Id}{{\mathds{1}}}

\newcommand{\tphi}{{\phi}}

\newcommand{\supp}{\operatorname{supp}}
\newcommand{\tr}{\operatorname{tr}}
\newcommand{\trf}{\operatorname{tr}^\flat}
\newcommand{\vol}{\operatorname{vol}}
\newcommand\vuo{{t_0}}
\newcommand{\wz}{w}
\newcommand{\whU}{\widehat {U}}
\newcommand\htop{h_{{\textrm{top}}}(\phi_1)}
\newcommand\Cnz{C_0}
\newcommand\Cnu{C_5}
\newcommand\Cnd{C_1}
\newcommand\Cnt{C_4}

\newcommand\Cnc{C_6}
\newcommand\Cns{C_2}
\newcommand\Cnse{C_3}
\newcommand\Ruelle{{\textrm{Ruelle}}}

\newcommand\volform{\omega_{\text{\rm vol}}}

\newcommand{\spqnorm}[2][p,q,\ell]{\ensuremath{\left| \hspace{-1.5pt} \left| \hspace{-1.5pt} \left| {#2} \right| \hspace{-1.5pt} \right| \hspace{-1.5pt} \right|_{#1}}}
\newcommand{\pqnorm}[2][p,q,\ell]{\left\|#2 \right\|_{#1}}
\newcommand{\spqnormbis}[2][p,q,\ell]{\ensuremath{\left| \hspace{-1.5pt} \left| \hspace{-1.5pt} \left| {#2} \right| \hspace{-1.5pt} \right| \hspace{-1.5pt} \right|_{#1,2}}}
\newcommand{\pqnormbis}[2][p,q,\ell]{\left\|#2 \right\|_{#1,2}}

\newlength{\strikelength}

\numberwithin{equation}{section}
\theoremstyle{plain}
\newtheorem{thm}{Theorem}[section]
\newtheorem{lem}[thm]{Lemma}
\newtheorem{sublem}[thm]{Sub-Lemma}
\newtheorem{prop}[thm]{Proposition}
\newtheorem{cor}[thm]{Corollary}

\newtheorem*{lem*}{Lemma}
\newtheorem{defn}[thm]{Definition}
\newtheorem{rmk}[thm]{Remark}

\begin{document}
\author{P. Giulietti\thanks{Dipartimento di Matematica,
Universit\`{a} di Roma {\em Tor Vergata},
Via della Ricerca Scientifica, 00133 Roma, Italy. \newline e-mail: giuliett@mat.uniroma2.it}, C. Liverani\thanks{Dipartimento di Matematica,
Universit\`{a} di Roma {\em Tor Vergata},
Via della Ricerca Scientifica, 00133 Roma, Italy.\newline e-mail: liverani@mat.uniroma2.it} and M. Pollicott\thanks{Mathematics Institute,
University of Warwick,
Coventry, CV4 7AL, UK.\newline e-mail: mpollic@maths.warwick.ac.uk}}
\title{Anosov Flows and Dynamical Zeta Functions}

\maketitle
{\let\thefootnote\relax\footnotetext{\textsl{Date:} \date{\today}}
\let\thefootnote\relax
\let\thefootnote\relax
\let\thefootnote\relax\footnotetext{We thank  V. Baladi, L. Caporaso, M. Carfora, J. de Simoi, S. Gouezel, F. Ledrappier,  A. Sambusetti, R. Schoof and M. Tsujii  for helpful discussions and comments. Also we thank the anonymous referee, S.Gouezel and M. Tsujii for pointing out several imprecisions in an earlier version. Giulietti and Liverani acknowledge the support of the  ERC Advanced Grant MALADY (246953) and they also thank the Fields Institute, Toronto, where part of this work was written.} }

\begin{abstract}
We study the Ruelle and Selberg zeta functions for $\Cs^r$ Anosov flows, $r > 2$, on a compact smooth manifold.
We prove several results, the most remarkable being: (a) for $\Cs^\infty$ flows the zeta function is meromorphic on the entire complex plane;
(b) for contact flows satisfying a bunching condition (e.g. geodesic flows on manifolds of negative curvature better than  $\frac 19$-pinched) the zeta function has a pole at the topological entropy and is analytic in a strip to its left; (c) under the same hypotheses as in (b) we obtain sharp results on the number of periodic orbits. Our arguments are based on the study of the spectral properties of a transfer operator acting on suitable Banach spaces of
anisotropic currents. 
\end{abstract}

\section{Introduction} \label{introduction}

In  1956,   Selberg  introduced a  zeta function for a surface of constant curvature $\kappa = -1$ formally defined to be   the  complex function 
\begin{equation} \label{eq:SelbergOriginal}
\zeta_{\textrm{Selberg}}(z) =
\prod_{\gamma}\prod_{n=0}^\infty  
\left(  1- e^{-\left(z+n\right)\lambda\left(\gamma\right)} \right), z \in \bC
\end{equation}
where $\lambda(\gamma)$ denotes the length of a closed geodesic  $\gamma$. 
This converges to a non-zero analytic function on the
half-plane $\textup{Re}(z) > 1$  and Selberg showed that $\zeta_{\textrm{Selberg}}$ 
has an analytic extension to the entire complex plane,  by
using the trace formula which now bears his name
\cite{Selberg56}. Moreover, he also showed that the zeros
of  $\zeta_{\textrm{Selberg}}$ 
correspond to the eigenvalues of the Laplacian. In fact, the trace formula connects the eigenvalues of $ - \Delta$  with the information provided by the geodesics, their lengths and their distribution.\footnote{ See \cite{Marklof08}, and references therein,
for a precise, yet friendly, introduction to the Selberg Trace formula and its relationship with the Selberg zeta function.} 
The definition \eqref{eq:SelbergOriginal}  was subsequently adapted to more general settings, 
including surfaces of variable curvature, thus giving birth to a new class of zeta functions which we refer to as dynamical zeta functions. However,  due to the lack of a suitable generalized trace formula,  few  results are known on their  meromorphic extension, the location of their zeros or  their relationships with appropriate operators. 

In 1976, Ruelle \cite{Ruelle76}  proposed  generalizing the definition by replacing 
the closed geodesics in $\zeta_{\textrm{Selberg}}$  by the closed orbits of an Anosov flow 
\mbox{$\phi_t: M \to M$}, where $M$ is a  $\Cs^\infty$, $d$-dimensional Riemannian compact manifold.
We recall that an Anosov flow is a flow such that there exists a $D\phi_t$-invariant 
continuous splitting $TM =  E^s \oplus E^u \oplus E^c $, where $E^c$ is the one-dimensional subspace tangent to the flow, and  constants $\Cnz\geq 1$ and  $\bar{\lambda}  >0$, such that for all $x\in M$
\begin{equation}\label{eq:asplit} 
\begin{array}{lll}
&\|D\phi_t(v)\| \leq  \|v\| \Cnz e^{- \bar{\lambda} t} \quad&\text{ if } t\geq 0,
v \in E^s , \\ 
&\|D\phi_{-t}(v)\| \leq  \|v\| \Cnz e^{-\bar{\lambda} t}  &\text{ if } t\geq 0, v
\in E^u ,  \\
&\Cnz^{-1} \|v\| \leq\|D\phi_{t}(v)\| \leq \Cnz \|v\| &  \text{ if } t\in\bR, v \in E^c.
\end{array}
\end{equation} 
We will denote by $d_u \doteq \dim(E^u) $ and $d_s \doteq \dim(E^s)$. The geodesic flow on manifolds with negative  sectional curvatures
are very special examples of mixing Anosov flows (see \cite{Sinai61}, \cite{Ratner87} and references therein).

In this context Ruelle defined a zeta function by 
\begin{equation}
\zeta_{\Ruelle}(z) =
\prod_{\tau \in \cT_p}    \left( 1- e^{-z \lambda(\tau)} \right)^{-1}, z \in \bC
\end{equation}
where $\cT_p$ denotes the set of prime orbits and $\lambda(\tau)$ denotes the period  of the  closed orbit $\tau$.\footnote{ A periodic orbit $\tau$, of period $\lambda(\tau)$, is a closed curve parametrized by the flow, i.e. $\tau:[0,\lambda(\tau))\to M$ such that $\tau(t)=\phi_t(\tau(0))$. A periodic orbit $\tau$ is prime if it is one-to-one with
its image. The range of $\tau$ is indicated again by $\tau$. If $\tau_p$ is the prime orbit related to $\tau$, then $\mu(\tau)$ is the unique integer such that $\lambda(\tau) = \mu(\tau) \lambda(\tau_p)$.}
Since it is known that the number of periodic orbits grows at most exponentially in the length  \cite{Margulis04}, it is easy to see that the above zeta functions are well defined for $\Re(z)$ large enough.
Also we can relate the Ruelle and Selberg zeta functions by 
\[
\zeta_{\textrm{Ruelle}}(z) =\frac{\zeta_{\textrm{Selberg}}(z+1)}{\zeta_{\textrm{Selberg}}(z)} \;\;, \qquad 
\zeta_{\textrm{Selberg}}(z) = \prod_{i = 0 }^{\infty}
\zeta_{\textrm{Ruelle}}(z + i)^{-1}
\]
when they are both defined. Here we study $\zeta_{\textrm{Ruelle}}$.

It is known, for weak mixing Anosov flows, that $\zeta_{\textrm{Ruelle}}(z)$ is analytic and 
non zero for $\Re(z)\geq \htop $ apart for a single pole at $z=\htop$, where by $\htop$ we mean the topological entropy of the flow
(see \cite{Bowen73}, or \cite[Page 143]{ParryPollicott90} for more details). Also it is long known that on the left of $\htop$ there exists a strip in which $\zeta_{\textrm{Ruelle}}(z)$ is meromorphic (\cite{ParryPollicott90} and references therein). It is interesting to notice that the poles of dynamical zeta functions (called also {\em resonances}) are often computationally accessible and of physical interest (see \cite[Chapter 17]{chaosbook} for a detailed discussion).

In the very special case of analytic Anosov flows with real analytic stable and unstable foliations, 
Ruelle already showed that his zeta function has a meromorphic extension to $\mathbb{C}$.
This result was generalized first by Rugh \cite{rugh96}, for three dimensional manifolds, 
and then by Fried \cite{Fried95}, in arbitrary dimensions, but still assuming an analyticity of the flow.
Here we extend such results to the $\Cs^\infty$ setting. More precisely, for $C^r$ flows we obtain a strip in which  $\zeta_{\textrm{Ruelle}}(z)$ is meromorphic of width unboundedly increasing with $r$. Note that this is consistent  with a previous example of Gallavotti \cite{Gallavotti76}, where  $\zeta_{\textrm{Ruelle}}$  is not meromorphic in the entire complex plane.

An additional knowledge of the location of the zeroes of $\zeta_{\textrm{Ruelle}}$ allows one to gain information on the distribution of the periodic orbits. For example it is known that if a negatively curved Riemannian manifold has dimension 2 \cite{Dolgopyat98a} or the sectional curvatures are  $\frac 14$-pinched \cite{Stoyanov11a}, then the number $N(T)$ of closed orbits of period less than $T$ satisfies $N(T)=\operatorname{li}(e^{\htop T})+O(e^{cT})$ where $ c<\htop$, and $\operatorname{li}(x)= \int_2^x\frac 1{\ln (s)}ds$. Note that either the assumption that $M$ is two dimensional  or that the sectional curvatures are $\frac 14$-pinched imply that the horocycle foliation is $\Cs^1$.
 One might then conjecture that such a foliation regularity is necessary in order to obtain the above estimate of the error term.  However, we show that this is not the case  and, although we cannot prove it in full generality, we conjecture that the above bound holds for all contact Anosov flows. 
As is usual for number theoretic zeta functions, a key ingredient in this analysis is showing that $\zeta_{\textrm{Ruelle}}$ is analytic in a strip to the left of its first pole. This result was stated in (\cite{dolgopyatpollicott98}) for geodesic  flows for which the sectional curvatures are $\frac 14$-pinched, although the proof  there was only sketched and is incomplete.  An earlier paper  (\cite{Pollicott89}) addressing the same question contains a more serious  flaw in the approach.

The general plan of the proof is explained in the next section, here we anticipate a few ideas and give a bit of their history. The basic idea goes back to Ruelle \cite{Ruelle76} and consists of relating the properties of the zeta function to the {\em dynamical determinant} of some appropriate operator. Classically such a program has been implemented by first coding the dynamics via Markov partitions and then defining appropriate operators acting on functions of the resulting subshift. As such a coding is, in general, only H\"older, all the properties connected to the smoothness of the dynamics are annihilated in the process. Yet, the work of Ruelle in the case of expanding maps and the seminal work of Kitaev \cite{Kitaev99} for Anosov diffeomorphisms showed that the smoothness properties of the map are exactly the ones responsible for larger domain of analyticity of the zeta function. The above cited works of Ruelle, Rugh and Fried for analytic flows made this even clearer.

The turning point in  overcoming this problem was \cite{Blank02} where it is shown how to construct Banach spaces in which the action of an Anosov map gives rise to a quasi compact operator (often called {\em Ruelle Transfer operator}) allowing to completely bypass the need of a Markov coding. This  opened the door to the possibility of dispensing with the analyticity hypothesis while still retaining the relevant smoothness properties of the dynamical systems. Since then a considerable amount of work has been done to implement such a goal. In particular, \cite{GouezelLiverani06, Liverani05,  BaladiTsujii00, GouezelLiverani08} and especially \cite{BaladiTsujii08} have clarified the relation between the smoothness of the map and the essential spectral radii of transfer operators as well as connected dynamical determinants to appropriate ``flat traces''\footnote{ Such terminology has also been adopted by the dynamical system community, following the work of Atiyah and Bott (\cite{AB64}, \cite{AB67}), since, as the 
reader will see, it is morally a regularization, or a flattening, of the trace.} of transfer operators (see \cite{Baillif04, BaladiBaillif05, BaladiTsujii08}). On the other hand, \cite{Liverani04, ButterleyLiverani07} illustrated how this approach can be applied also to flows by showing that the resolvent of the generator of the flow is quasi compact on such spaces. More recently some beautiful work, although limited  
to the case of contact flows, allows one to study directly the transfer operator associated to the time one map of the flow \cite{Tsujii10,Tsujii12, TsujiiFaure, FaureSjostrand}.  Note that, for such types of result, some condition on the flow is necessary. Indeed, the quasi compactness of the operator associated to the time one map plus mixing implies exponential mixing,  while there are known examples of Axiom A flows, constructed as piecewise constant ceiling suspensions, that are mixing but enjoy arbitrarily slow rates of decay correlations \cite{Ruelle83,Pollicott84}).

Given such a state of the art, it looked like all the ingredients needed to tackle the present problems were present even though quite a lot work remained to be done. It remained to choose a precise line of attack. Given the complexity of the task, we chose to follow the path of least resistance at the cost of obtaining possibly suboptimal results.
In particular, here we extend the results of \cite{GouezelLiverani06, ButterleyLiverani07} to allow the study of the action of the flow on the space of exterior forms (similarly to the original Ruelle approach for expanding maps \cite{Ruelle76}).   Next, to study the relation between the spectra of such operators and the dynamical determinants we extend to the case of flows a suboptimal, but very efficient, trick introduced in \cite{LiveraniTsujii06} for the case of maps and inspired by the work of Margulis (as explained to Liverani by Dolgopyat). 
It is likely that, at the price of some more work, one could adapt alternative approaches to the present case. For example defining the operator on different Sobolev-like spaces or replacing the tensor trick by a strategy based on kneading determinants, as in Baladi-Tsujii \cite{BaladiTsujii08}. Such alternatives might allow one to obtain sharper estimates in the case of finite differentiability.

Next,  we show that the zeta function has a pole free strip at the left of the topological entropy and we obtain bounds for the growth of the zeta function on such a strip  provided a Dolgopyat type estimate for the action of the flow on $d_s$-external forms holds true, $d_s$ being the dimension of the strong stable manifold. To do so we use the simple strategy put forward in \cite{Liverani05}. Finally, to establish the Dolgopyat estimate in the present context we follow the strategy developed in \cite{Liverani95, Liverani04, BaladiLiverani11} for the 
action on functions. Note that to look at $d_s$-forms corresponds, in the old Markov based strategy, to studying the statistical properties of the flow with respect to the measure of maximal entropy. The extension of the geometric part of the original Dolgopyat argument to this situation presents two difficulties. One, well known, related to the lack of regularity of the foliations that is solved, following \cite{Liverani04}, by restricting the study to the contact flows. The second (here treated in Lemma \ref{sublem:doest-2}), was unexpected (at least to us) and is related to having (morally) the measure of maximal entropy, rather than Lebesgue, as a reference measure. This we can solve only partially, hence the presence of a bunching (or pinching) condition in our results.

The structure of the paper is as follows.  In section \ref{sec:main} we present  the  statements of  the main results in this paper. We also  explain the strategy of  the proofs assuming several Lemmata and constructions detailed in later sections. 
In section \ref{banachspace} we construct the spaces on which our operators will act.   
Sections \ref{secresolvent},   \ref{sec:flat}  and \ref{sec:splitting}  contain estimates for transfer operators and their ``flat traces''.  
In sections \ref{sec:contact} and  \ref{sec:zeta-growth} we restrict ourself to the case of contact flows.
In particular, in  section \ref{sec:contact} we exclude the existence of zeroes in a vertical strip to  the left of the topological entropy by the means of a Dolgopyat like estimate. In section \ref{sec:zeta-growth} we obtain a bound on the growth of $\zeta_{\textrm{Ruelle}}$ in this  strip.

In Appendix  \ref{appendiceLie} we collect together, for the reader's convenience, several facts from differential geometry,  while in Appendix \ref{app:orientable} we discuss the orientability of the stable distribution. In Appendix \ref{app:topent} we relate the topological entropy and the volume growth of manifolds. In Appendix \ref{subsec:average} we detail some facts about mollificators acting on the Banach spaces of interest. Finally, in Appendix \ref{app:holo-est},  we recall some necessary facts concerning holonomies. 

\footnotesize
{
\tableofcontents
}
\normalsize

\section{Statement of Results} \label{sec:main}
\subsection{Some notation}\label{secsec:notation}
 We use $B_{d}\! \left( x, r \right)$ to designate the open $d$-dimensional ball with  center $x$ and radius $r$. 

 We will use $C_{\#}$ to represent a generic constant (depending only on the manifold, the flow, the choice of the charts and the partition of unity made in Section \ref{banachspace}) which could change from one occurrence to the next, even within the same equation.  We will write $C_{a_1 , \ldots , a_k }$ for generic constants depending on the parameters $a_1 ,\ldots , a_k$,  which could still change at every occurrence. Finally, numbered constants $\Cnz, \Cnd , \Cns, \dots$ are constants with a fixed value thruought the paper.
 
\subsection{Theorems and Proofs}
Our first result applies to all Anosov flows $\phi_t$ on a connected, compact and orientable $\Cs^\infty$ Riemannian manifold $M$.\footnote{ One could easily extend the result to $\Cs^r$ manifolds, but we avoid it to ease notation.}

\begin{thm}\label{main1}
For any $\Cs^r$ Anosov flow $\phi_t$ with $r > 2$,
$\zeta_{\textrm{Ruelle}}(z)$  is meromorphic in the region 
\[
\Re(z) >  h_{\mathrm{top}}(\phi_1) - \frac{\olambda}{2} 
\left\lfloor\frac {r-1} 2 \right\rfloor 
\]
where $\olambda$, is determined by the Anosov splitting, and $ \lfloor x\rfloor $ denotes the integer part of  $x$.
Moreover,  $\zeta_{\textrm{Ruelle}}(z)$ is analytic  for $\Re(z)> h_{\textrm{top}}(\phi_1)$ and non zero for $\Re(z)> \max\{ h_{\mathrm{top}}(\phi_1) - \frac{\olambda}{2}  \left\lfloor\frac {r-1} 2 \right\rfloor,  h_{\mathrm{top}}(\phi_1) -\olambda\}$. If the flow is topologically mixing then $\zeta_{\textrm{Ruelle}}(z)$ has no  poles on the line $\{h_{\textrm{top}}(\phi_1)+i b\}_{b\in\bR}$ apart from a single simple pole at $z=h_{\textrm{top}}(\phi_1)$. 
\end{thm} 

\begin{cor}
For any $\Cs^\infty$ Anosov flow the zeta function $\zeta_{\textrm{Ruelle}}(z)$  is
meromorphic in the entire complex plane.\footnote{ This provides an answer to an old question of Smale: ``Does $\zeta_{\textrm{Selberg}}(z)$  have nice properties for any general class of flows?''  cf. pages 802-803, of \cite{Smale67}.  Smale  also specifically  asked if for suspension flows over Anosov diffeomorphisms close to constant height suspensions  the zeta function $\zeta_{\textrm{Selberg}}(z)$ has a meromorphic extension to all of $\mathbb  C$.  The above corollary answers  these questions in the affirmative for $C^\infty$ Anosov flows, despite Smale's  comment ``I must admit that a positive answer would be a little shocking''.}
\end{cor}

Note that if the flow is not topologically  mixing then the flow can be reduced to a constant ceiling suspension and  hence there exists $b>0$ such that $\zeta_{\Ruelle}(z+ib) = \zeta_{\Ruelle}(z)$  (for more details see \cite{ParryPollicott83}).

\begin{cor} $\zeta_{\textrm{Ruelle}}(z)$ and $\zeta_{\textrm{Selberg}}(z)$ are meromorphic
in the entire complex plane for smooth geodesic flows on any connected compact orientable Riemannian manifold  with 
variable strictly negative sectional curvatures.
Moreover, the zeta functions $\zeta_{\Ruelle}(z)$ and $\zeta_{Selberg}(z)$ have no zeroes or poles on the line $\{\htop +ib\}_{b\in \mathbb R}$, except at $z=\htop $ where both $\zeta_{\Ruelle}(z)^{-1}$  and $\zeta_{Selberg}(z)$ 
have  a simple zero.
\end{cor}

Next, we specialize to {\em contact} Anosov flows. Let $\lambda_+\geq 0$ such that $\|D\phi_{-t}\|_\infty \leq \Cnz e^{\lambda_+ t}$ for all $t\geq 0$.
\begin{thm}\label{thm:main2} For any $\Cs^r, r > 2$, contact flow with $\frac {\olambda}{\lambda_+}> \frac 13$ there exists $\tau_*>0$ such that the Ruelle zeta function is analytic in $\{z\in\bC\;:\;\Re(z)\geq \htop-\tau_*\}$ apart from a simple pole at $z=\htop$.
\end{thm} 
\begin{rmk} Note that the bunching condition $\frac {2\olambda}{\lambda_+}> \frac 23$ implies that the invariant foliations are at least $\frac 23$-H\"older continuous. Also remember that an $a$-pinched geodesic flow is $(2\sqrt{a} + \ve)$-bunched, that is $ a \leq  \left( \frac{\olambda}{\lambda_+} \right)^2 + \ve $, see \cite{Has1} for more details.
\end{rmk}
The above facts have several important implications. We begin with some low-hanging fruits. The proof will be obvious once the reader goes over the construction explained later in this section and remembers that the dynamical determinants for $0$-forms and $d$-forms have their first zeros at $\Re(z) = 0$ (since they are exactly the dynamical determinants of the usual Ruelle transfer operator). 

\begin{cor} For a volume preserving three dimensional Anosov flow we have that the zeta function $\zeta_{\Ruelle}(z)$ is meromorphic in $\bC$ and, moreover,
\begin{itemize} 
\item $\zeta_{\Ruelle}(z)$  is analytic for $ \Re(z) \geq \htop - \tau_*$, except for a pole at $\htop$; 
\item  $\zeta_{\Ruelle}(z) $ is non-zero for $\Re(z) > 0$.
\end{itemize}
\end{cor}
Theorem  \ref{thm:main2} has the following simple  consequence for  the rate of mixing with respect to the measure of maximal entropy.

\begin{cor} The  geodesic flow $\phi_t: T_1M \to T_1M$ for  a compact manifold  $M$ with better than $\frac{1}{9}$-pinched negative sectional curvatures (or more generally any  contact Anosov flow satisfying the hypothesis of the  Theorem  \ref{thm:main2}) is  exponentially mixing with respect to the Bowen-Margulis measure $\mu$, i.e., there exists $\alpha > 0$ such that 
for $f,g  \in C^\infty (T_1M)$ there exists $C > 0$ for which the correlation function
$$\rho(t) = \int f\circ \phi_t  g d\mu  - \int  f d\mu  \int  g d\mu$$
 satisfies
$|\rho(t)| \leq C_\#e^{-\alpha |t|}$, for all $t \in \mathbb R$.
\end{cor}

\begin{proof}
Consider the Fourier transform $\widehat \rho(s) = \int_{-\infty}^\infty e^{ist} \rho(t) dt$ of 
the correlation function $\rho(t)$.  By (\cite{Pollicott85}, Theorem 2)  and (\cite{Ruelle87}, Theorem 4.1), the analytic extension of $\zeta_{\Ruelle}(z)$ 
in Theorem  \ref{thm:main2} 
 implies that there exists $0 < \eta  \leq \tau_* $ such that  $\widehat \rho(s)$ has an analytic extension to a strip 
$|\Im(s)| < \eta$.  Moreover, adapting  the argument in 
(\cite{Liverani04}, proof of Theorem 2.4)  we can use the smoothness of the test functions to  
allow us to assume without loss of generality that for each  fixed value $- \eta < t <  \eta$ we have that the function 
$\sigma \mapsto \widehat \rho(\sigma + it)$ is in $L^1(\Bbb R)$.
Finally, we apply the Paley-Wiener theorem (\cite{ReedSimon}, Theorem IX.14) to deduce the result. 
\end{proof}

Moreover, Theorem \ref{thm:main2} allows us
to extend results of Huber-Selberg (for constant sectional curvatures), Pollicott-Sharp (for surfaces of negative curvature) and Stoyanov (for $\frac 14$-pinched geodesic flows). By a prime closed geodesics  we mean an oriented closed geodesic which is a closed curve that traces out its image exactly once. 

\begin{thm}[Prime Geodesic Theorem with exponential error]\label{thm:main3}
Let $M$ be a manifold better than $\frac 19$-pinched,  with strictly negative sectional curvature. Let $\pi(T)$ denote the number of prime closed geodesics  on $M$ with length at most $T$, then there exists $\delta > 0$ such that
\[
\pi(T) = \hbox{\rm li}(e^{\htop T}) + O(e^{(\htop-\delta)T}) \hbox{ as } T \to +\infty.
\]
\end{thm}

\begin{rmk} The above Theorems are most likely not optimal. The $\frac 19$-pinching might conceivably be improved with some extra work (one would need to improve, or circumvent the use of, Lemma \ref{sublem:doest-2}) but we do not see how to remove such conditions completely, even though we believe it to be possible.
\end{rmk}

Let us start the discussion of the proofs. The basic objects we will study are the {\sl dynamical determinants}, following the approach  introduced 
by Ruelle \cite{Ruelle76}, which arise naturally in the dynamical context and are formally of the general form
\begin{equation}\label{eq:det-dif}
 \mathfrak{D}_\ell(z) = \exp \left( - \sum_{\tau \in \cT } \frac{ \mathrm{tr} \left( \wedge^\ell (D_{\mathrm{hyp}}\phi_{-\lambda(\tau)}) \right)
  e^{-z \lambda(\tau)} }{\mu(\tau)\epsilon(\tau)\left|    \det\left(\mathds{1}-D_{\mathrm{hyp}}\phi_{-\lambda(\tau)} \right) \right|} \right),
\end{equation}
where $\epsilon (\tau) $ is $1$ if the flow preserves the orientation of $E^s$ along $\tau$ and $-1$ otherwise.  
The symbol $D_{\textrm{hyp}}\phi_{-t}$ indicates the derivative of the map induced  by two local  transverse sections to the orbit (one at $x$, the other at $\phi_{-t}(x)$) and can be represented as a $(d-1)\times (d-1) $ dimensional matrix.   
By $\wedge^\ell A$ we mean the matrix associated to the standard $\ell$-th exterior product of $A$.
Note that, given any $\tau\in\cT$, $\mathrm{tr} \left( \wedge^\ell (D_{\mathrm{hyp}}\phi_{-\lambda(\tau)}) \right)$, $\det\left(\mathds{1}-D_{\mathrm{hyp}}\phi_{-\lambda(\tau)} \right)$, as well as $ \det(\left. D\phi_{-\lambda(\tau)} \right|_{E^s})$ that will be used shortly, when computed at a point $x\in\tau$, depend only on $\tau$ (see comments before equation \ref{eq:trace-ok} for more details).  

The sum in \eqref{eq:det-dif} is well defined provided  $\Re(z)$ is large enough.

Note that $ \epsilon (\tau) =  \mathrm{sign}\left(\det(\left. D\phi_{-\lambda(\tau)} \right|_{E^s})\right) $. 
Hence, from the hyperbolicity conditions \eqref{eq:asplit}, we have
\begin{equation}\label{eq:sign-orientation}
\begin{split}
\mathrm{sign}\left(\det\left(\mathds{1}-D_{\mathrm{hyp}}\phi_{-\lambda(\tau)} \right)\right) =& (-1)^{d_s} \mathrm{sign}\left(\det(\left. D\phi_{-\lambda(\tau)} \right|_{E^s})\right)\\
= & (-1)^{d_s} \epsilon(\tau) .
\end{split}
\end{equation}
Recall the linear algebra identity for a $n \times n$ matrix $A$ (see, for example, \cite[Section 3.9]{winitzki09} for more details),
\begin{equation} \label{eq:algebra}
\det(\Id-A)=\sum_{\ell=0}^n (-1)^\ell\operatorname{tr}(\wedge^\ell A ).
\end{equation}
Note that 
\begin{equation} \label{eq:ruelle-product}
\begin{split} 
\zeta_{\textrm{Ruelle}}(z) &= \prod_{\tau \in \cT_p}    \big( 1- e^{-z \lambda(\tau)} \big)^{-1} = 
 \exp \left( \sum_{\tau \in \cT_p } \sum_{m=1}^\infty \frac{1}{m} e^{-zm\lambda(\tau)} \right)\\
 &=\exp \left( \sum_{\tau \in \cT} \frac{1}{\mu(\tau)} e^{-z\lambda(\tau)} \right),
\end{split} 
\end{equation}
where $\mu(\tau)$ is the multiplicity of the associated orbit
$\tau$ and $\cT$ is the whole set of periodic orbits on $M$.
From the equations \eqref{eq:det-dif}, \eqref{eq:sign-orientation},  \eqref{eq:algebra} and \eqref{eq:ruelle-product} it follows a product formula analogous to that of Atiyah-Bott for elliptic differential operators \cite{AtiyahBott66}:
\begin{equation} \label{eq:prod}
\begin{split} & \prod_{\ell = 0}^{d-1}
 \mathfrak{D}_\ell(z)^{(-1)^{\ell + d_s + 1}} \hskip-.4cm = 
 \exp \left( \sum_{\ell = 0}^{d-1} \sum_{\tau \in \cT } \frac{ (-1)^{\ell + d_s }\mathrm{tr} \left( \wedge^\ell (D_{\mathrm{hyp}}\phi_{-\lambda(\tau)}) \right)
  e^{-z \lambda(\tau)} }{\mu(\tau)\epsilon(\tau)\left|    \det\left(\mathds{1}-D_{\mathrm{hyp}}\phi_{-\lambda(\tau)} \right) \right|} \right) \\
& =  \exp \left( \sum_{\tau \in \cT } \frac{ (-1)^{d_s }  \det\left(\mathds{1}-D_{\mathrm{hyp}}\phi_{-\lambda(\tau)} \right)
  e^{-z \lambda(\tau)} }{\mu(\tau)\epsilon(\tau)\left|    \det\left(\mathds{1}-D_{\mathrm{hyp}}\phi_{-\lambda(\tau)} \right) \right|} \right)  \\
 &=  \exp \left( \sum_{\tau \in \cT } \frac{    e^{-z \lambda(\tau)} }{\mu(\tau)} \right)  =  \zeta_{\textrm{Ruelle}}(z).
  \end{split} 
\end{equation}

Thus Theorem \ref{main1} follows by the analogous statement for   the
dynamical determinants $\mathfrak{D}_\ell(z)$. 
To study the region in which the $\mathfrak{D}_\ell(z)$ are meromorphic we will proceed in the following roundabout manner.
First we define the following objects.
\begin{defn} Given $ 0 \leq \ell \leq d - 1  $, $\tau \in \cT $ let 
\begin{equation} \label{eq:tracechar}
\chi_{\ell}(\tau) \doteq  \frac{ \mathrm{tr} \left( \wedge^\ell (D_{\mathrm{hyp}} \phi_{-\lambda(\tau)})  \right)} { \epsilon(\tau)
\left|  \det\left(\mathds{1}-D_{\mathrm{hyp}}\phi_{-\lambda(\tau)} \right) \right|}. 
\end{equation}
Moreover, for  $\xi  , z \in \bC $, we let  
\begin{equation} \label{mockdet}
 \widetilde{\mathfrak{D}}_\ell(\xi,z) \doteq   \exp \left( - \sum_{n=1}^\infty \frac{\xi^n}{n!}
\sum_{\tau \in \cT} \frac{ \chi_{\ell}(\tau) }{\mu(\tau)}  \lambda(\tau)^n e^{-z \lambda(\tau)} \right).
\end{equation} 
\end{defn}
\noindent Note that the series in \eqref{mockdet} trivially converges  for $| \xi | $ sufficiently small and $\Re(z)$ sufficiently large. 
\begin{lem} \label{damnedremark}
Let $0 \leq \ell \leq d - 1, \xi,z \in \bC $, $\Re(z)$ sufficiently large and $ | \xi - z |$ sufficiently small. 
Then we can write 
\begin{equation} \label{eq:damnedremark}
 \widetilde{\mathfrak{D}}_\ell(\xi - z,  \xi ) =  \frac{\mathfrak{D}_\ell(z)}{\mathfrak{D}_\ell(\xi)}.
\end{equation}
 \end{lem}
 \begin{proof}
The proof is by a direct calculation
\[ 
\begin{split}
\widetilde{\mathfrak{D}}_\ell(\xi - z, \xi ) & = \exp \left( - \sum_{n=1}^\infty \frac{(\xi - z)^n}{n!} \sum_{\tau \in \cT}\frac{ \chi_{\ell}(\tau) }{\mu(\tau)}
\lambda(\tau)^n e^{-\xi \lambda(\tau)}  \right) \\
& =\exp \left( - \sum_{\tau \in \cT}   \frac{ \chi_{\ell}(\tau)}{\mu(\tau)}
 \left( e^{-z \lambda(\tau)} - e^{-\xi \lambda(\tau)}  \right) \right) 
 =  \frac{\mathfrak{D}_\ell(z)}{\mathfrak{D}_\ell(\xi)}.
\end{split} 
\]
\end{proof}
\noindent  Hence Theorem \ref{main1} is implied by the following. 
\begin{prop} \label{main2}
For any $\Cs^r$ Anosov flow, with $r  > 2$, $\xi,z\in\bC$,  $\mathfrak{D}_\ell(\xi)$ is analytic and non zero in the region $\Re(\xi)>\htop-\bar\lambda|d_s-\ell|$ and, for $\xi$ in such a region, the function $\widetilde{\mathfrak{D}}_\ell(\xi - z,\xi)$ is analytic and non zero for $z$ in the region 
\[
|\xi-z| < \Re(\xi)-h_{\textrm{top}}(\phi_1) + |d_s-\ell|\olambda   
\] 
and analytic in $z$ in the region  
\[
|\xi-z| < \Re(\xi)-h_{\textrm{top}}(\phi_1) + |d_s-\ell|\olambda + \frac{\olambda}{2}
\left\lfloor\frac {r-1} 2 \right\rfloor \,.
\] 
\end{prop}
We can then freely move $\xi$ along the line $\{a+ib\}_{b\in\bR}$ and we obtain that $\widetilde{\mathfrak{D}}_\ell(\xi-z,\xi)$ is analytic for $\Re(z)>  h_{\textrm{top}}(\phi_1) - |d_s-\ell|\olambda - \frac{\olambda}{2}\left\lfloor\frac {r-1} 2 \right\rfloor \,$.

\begin{proof}[{\bf Proof of Proposition \ref{main2}}]
For all $v, \ell\in\bN$, let $\Omega^{\ell}_{v}(M)$ be the space of $\ell$-forms on $M$, i.e. the $\Cs^v$ sections of $\wedge^\ell({T^*M})$. Let $\Omega^\ell_{0,v}(M) \subset \Omega^{\ell}_{v}(M) $ be the subspace of forms null in the flow direction (see \eqref{eq:ooo}). 
In section \ref{banachspace} we construct a family of Banach spaces $\cB^{p,q,\ell}$, $p\in\bN, q\in\bRp$, as the closure of  $\Omega^\ell_{0,v}(M)$
with respect to a suitable anisotropic norm so that the spaces $\cB^{p,q,\ell}$ are an extension of the spaces in 
\cite{GouezelLiverani06}.\footnote{ The indexes $p,q$ measure, respectively, the regularity in the unstable and stable direction.}  Such spaces are canonically embedded in the space of currents (see Lemma \ref{lem:currents}). 

In Section \ref{secresolvent} we define a family of operators indexed by
positive real numbers $t \in \bRp$ as
\begin{equation} \label{def:reducedoperator}
\cL^{(\ell)}_t(h)  \doteq \tphi_{-t}^*h,
\end{equation}
for $h \in \Omega^{\ell}_{0,v}(M)$. Here we adopt the standard notation where  $f^*$ denotes the pullback and $f_*$ indicates the push-forward.
\begin{rmk} Note that by restricting the transfer operator $\cL^{(\ell)}_t$ to  the space $\Omega^{\ell}_{0,v}(M)$ we mimic the action of the standard transfer operators on sections transverse to the flow, in fact we morally project our forms on a Poincar\'e section. 
\end{rmk}
\begin{rmk}\label{rem:orientability-s}
In order to simplify a rather involved argument we chose to give full details only for the case in which the invariant foliations are orientable (i.e., $\epsilon(\tau)=1$ for all $\tau$). Notably, this includes some of the most interesting  examples, such as geodesic flows on orientable manifolds with  strictly negative sectional curvature (see Lemma \ref{lem:geor}). To treat the non orientable case it is often sufficient to slightly modify the definition of the operator \eqref{def:reducedoperator} by introducing an appropriate weight, see equation \eqref{eq:transferop0}, and then repeating almost verbatim the following arguments. Unfortunately, as far as we can see, to treat the fully general case one has to consider more general Banach spaces than the ones used here. This changes very little in the arguments but makes the notation much more cumbersome. The reader can find the essential details in Appendix \ref{app:orientable}.
\end{rmk}

The operators \eqref{def:reducedoperator} generalize the action of the transfer operator $\cL_t$ on the spaces
$\cB^{p,q}$ of \cite{GouezelLiverani06}.  
We prove in Lemma \eqref{lem:LY}  that  the operators ${\cL}^{(\ell)}_t$ satisfy a Lasota-Yorke type estimate for sufficiently large times. To take care of short times we restrict ourself to a new space $\widetilde{\cB}^{p,q,\ell}$ which is the closure of $\Omega^{\ell}_{0,s}(M)$ with respect to a slightly stronger norm $\spqnorm{\cdot}$ (see \eqref{eq:strong-norms-sigh}), in this we follow \cite{BaladiLiverani11}.  We easily obtain Lasota-Yorke type estimate for all times in the new norm (Lemma \ref{lem:LYstrong}).

On $\widetilde{\cB}^{p,q,\ell} $ the operators $\cL_t^{(\ell)}$ form a strongly continuous semigroup  with generators $X^{(\ell)}$ (see Lemma \ref{lem:LYstrong}). We can then consider the resolvent ${R}^{(\ell)} (z)\doteq ( z \Id-X^{(\ell)})^{-1}$. The cornerstone of our analysis is that, although the operator $X^{(\ell)}$ is an unbounded closed operator  on $\widetilde{\cB}^{p,q,\ell} $, we can access its spectrum thanks  to the fact that its resolvent $R^{(\ell)}(z)$ is a quasi compact operator on the same space. More precisely, in Proposition \ref{lem:Nussbaum}  we show  that for $z \in \bC$, $\Re(z) > \htop-|d_s-\ell|\bar  \lambda$, the operator ${R}^{(\ell)}(z)\in L(\widetilde{\cB}^{p,q,\ell}, \widetilde{\cB}^{p,q,\ell})$, $p+q< r-1$, has spectral radius 
\[
\rho ({R}^{(\ell)}(z) ) \leq  \big( \Re(z) -  \htop + |d_s-\ell | \olambda   \big)^{-1}  
\]
and essential spectral radius 
\[
\rho_{\textrm{ess}} ({R}^{(\ell)}(z) ) \leq  \big( \Re(z) -  \htop + |d_s-\ell | \olambda + \olambda\min\{p,q\}  \big)^{-1} .  
\]
We can then write the spectral decomposition $R^{(\ell)}(z)=P^{(\ell)}(z)+U^{(\ell)}(z)$ where $P^{(\ell)}(z)$ is a finite rank operator and $U^{(\ell)}(z)$ has spectral radius arbitrarily close to  $\rho_{\textrm{ess}} ({R}^{(\ell)}(z) )$.
In section \ref{sec:flat},   we define a ``flat trace" denoted by $\trf$. In Lemma \ref{lem:traceorbits} we show that for $\Re(z) > \htop-|d_s-\ell|\bar  \lambda$ and $n \in \bN$, 
we have that
 $ {\trf}\left({R}^{(\ell)}(z)^n\right) < \infty $ and 
\begin{equation}\label{eq:addendum}
 {\trf}\left({R}^{(\ell)}(z)^n\right) = \frac{1}{(n-1)!} 
\sum_{\tau \in \cT} \frac{ \chi_{\ell}(\tau)}{\mu(\tau)}
\lambda(\tau)^n e^{-z \lambda(\tau)} .
\end{equation}
Furthermore, in Lemma \ref{lem:tensor-trace}, we prove that, for each $\lambda<\olambda$ and $n\in\bN$,
\begin{equation}\label{eq:hoihoihoi}
\begin{split}
\left|\trf({R}^{(\ell)}(z)^n) - \mathrm{tr} \left({P}^{(\ell)}(z)^n\right)\right|\leq& C_{z,\lambda}
\bigg( \Re(z) -  h_{\textrm{top}}(\phi_1) + |d_s-\ell| \lambda \\
&\quad\quad\quad+ \frac{\lambda}{2}\min\{p,q\}   \bigg)^{-n}. 
\end{split}
\end{equation}
where ``$\textrm{tr}$'' is the standard trace.
\begin{rmk}
The loss of a factor $2$ in the formula \eqref{eq:hoihoihoi} is due to an artifact of the method of proof. We live with it since it does not change substantially the result and to obtain a sharper result might entail considerably more work.
\end{rmk}
\begin{rmk}
Given the two previous formulae one can have a useful heuristic explanation of the machinery we are using. Indeed, \eqref{eq:hoihoihoi} shows that $\trf(R^{(\ell)}(z))$ is essentially a real trace, then substituting formula \eqref{eq:addendum} in \eqref{eq:det-dif} and performing obvious formal manipulations we have that $\widetilde{\mathfrak{D}}_\ell(\xi,z)$ can be interpreted as the ``determinant" of $(\Id-\xi R^{(\ell)}(z))^{-1}$, while $\mathfrak{D}_\ell(z)$  can be interpreted as the (appropriately renormalized)``determinant" of $z\Id-X^{(\ell)}$.
\end{rmk}
Note that if $\nu\in \sigma({P}^{(\ell)}(z))\setminus \{0\}$, then $z-\nu^{-1}\in\sigma(X^{(\ell)})$.\footnote{ As usual, we denote by $\sigma$  the spectrum.} Thus, $\nu=(z-\mu)^{-1}$ where $\mu\in\sigma(X^{(\ell)})$.
Let  $\xi \in \bC$ such that $ a = \Re(\xi)$ is sufficiently large so that $\mathfrak{D}_\ell(\xi)$ is well defined. Let $\rho_{p,q,\ell}< a -  h_{\textrm{top}}(\phi_1) + |d_s-\ell|  \lambda + \frac{\lambda}{2}\min\{p,q\}  $. 
Let  $\lambda_{i,\ell}$ be the eigenvalues of $X^{(\ell)}$. For each $z\in  B_2(\xi,\rho_{p,q,\ell}) $,
\begin{equation} \label{eq:X-ruelle}
\begin{split}
& \widetilde{\mathfrak{D}}_\ell(\xi - z, \xi ) =  \exp \left(- \sum_{n=1}^\infty
  \frac{(\xi - z)^n}{n}\trf({R}^{(\ell)}(\xi)^n)  \right) = \\
& =\exp \left( - \sum_{n=1}^{\infty}\frac{(\xi - z)^n}{n} 
\left( \sum_{\lambda_i \in B_2(\xi,\rho_{p,q,\ell})} \frac{1}{(\xi - \lambda_{i,\ell})^n}  + 
\mathcal{O}\left(C_{\xi,\lambda} \rho_{p,q,\ell}^{-n} \right) 
    \right)  \right)  \\
&=  \exp  \left(\sum_{\lambda_i\in B_2(\xi,\rho_{p,q,\ell})} 
\hskip-.3cm\log\left( 1 - \frac{\xi - z}{\xi -  \lambda_{i,\ell}} \right)  + 
\sum_{n=1}^{\infty}\frac{(\xi - z)^n}{n}\mathcal{O}\left(C_{\xi,\lambda} \rho_{p,q,\ell}^{-n} \right)   \right)\\
&  = \left( \prod_{ \lambda_{i,\ell} \in B_2(\xi,\rho_{p,q,\ell})} \frac{z -  \lambda_{i,\ell}}{\xi -  \lambda_{i,\ell}} \right) \psi(\xi,z)\\
\end{split} 
\end{equation}
where $ \psi(\xi,z)$ is analytic and non zero for $z\in B_2(\xi,\rho_{p,q,\ell})$. The results follows by optimizing the choice of $p,q$.
\end{proof}

Once Theorem \ref{main1} is established we can use the above machinery to obtain more information on the location of the zeros. 
\begin{proof}[{\bf Proof of Theorem \ref{thm:main2}}]
Equation \eqref{eq:X-ruelle} and Lemma \ref{damnedremark} show that the poles of $\zeta_{\text{Ruelle}}$ are a subset of the eigenvalues of the $X^{(\ell)}$. Proposition \ref{main2} implies that $\ell=d_s$ is the only term in \eqref{eq:prod} that can contribute a pole in the relevant strip. Thus, it suffices to study the poles of $\mathfrak{D}_{d_s}^{-1}$.
In Lemma \ref{lem:dolgo-est} we prove that there exist $\gamma_0, a_0,\Cnd>0$ such that, for $0<q<\min\left\{\frac{\olambda}{\lambda_+},\frac{4\olambda^2}{\lambda_+^2}\right\}$, $2a_0>\Re(z) -\htop>a_0$ and $n\geq \Cnd C_\#\ln|\Im(z)|$,
\[
\spqnorm[1,q,d_s]{R^{(d_s)}(z)^n}\leq C_\# (\Re(z)-\htop)^{-n}|\Im(z)|^{-\gamma_0}.
\]
In Lemma \ref{lem:quasicompactness} we have, for each $\Re(z)> \htop$ and $n\in\bN$,
\[
\spqnorm[1,q,d_s]{R^{(d_s)}(z)^n}\leq C_\# (\Re(z)-\htop)^{-n}.
\]
Accordingly, the resolvent identity $R^{(d_s)}(z-a)=\sum_{n=0}^\infty a^n R^{(d_s)}(z)^n$ yields
\[
\begin{split}
\|R^{(d_s)}(z-a)\|&\leq \sum_{n=0}^\infty |a|^n (\Re(z)-\htop)^{-n} \left[C_\#|\Im(z)|^{-\gamma_0}\right]^{\frac{n}{\Cnd C_\#\ln|\Im(z)|}-1}\\
&=\frac{C_\#|\Im(z)|^\gamma_0}{1-|a|(\Re(z)-\htop)^{-1} C_\#^{\frac 1{\Cnd C_\#\ln|\Im{z}|}} e^{-\frac{\gamma_0}{C_\#\Cnd}}}.
\end{split}
\]
From the above the have the statement of Theorem for $\Im(z)$ large enough. The Theorem then follows by the spectral characterization of Proposition \ref{lem:Nussbaum}, since smooth contact flows on connected manifolds are mixing.
\end{proof}

\begin{proof}[{\bf Proof of Theorem \ref{thm:main3}}] 
Since there is an obvious bijection between prime closed geodesics and prime closed orbits for the geodesics flow, $\pi(T)$ is precisely   the number of prime closed orbits $\tau$ for the geodesic flow whose period $\lambda(\tau)$ is at most $T$.  
The proof of the Theorem is based on the following estimate, established in Lemma \ref{lem:ruellebound}. For each $z\in\bC$, $\htop\geq\Re(z)>\htop-\frac{\tau_*}2$, $|\Im(z)|\geq 1$, we have that
\begin{equation} \label{eq:mark0}
\left|\frac{\zeta_{\textrm{Ruelle}}'(z)}{\zeta_{\textrm{Ruelle}}(z)}\right|
\leq C_\# |z|.
\end{equation}
In particular, on the line  $\Re(z) = \sigma_1 := \htop - \frac{\tau_*}{4}$, say,  we have the bound (\ref{eq:mark0}).
Moreover, on the line $\Re(z) = \sigma_2 := \htop + 1$, say, we have a uniform bound
\begin{equation} \label{eq:marker}
\left|\frac{\zeta_{\textrm{Ruelle}}'(z)}{\zeta_{\textrm{Ruelle}}(z)}\right|  
\leq \left|\frac{\zeta_{\textrm{Ruelle}}'(\sigma_2)}{\zeta_{\textrm{Ruelle}}(\sigma_2)}\right| < +\infty.
\end{equation}
By the Phragm\'en-Lindel\"of Theorem (\cite{Titchmarsh},\S5.65) the bound on the logarithmic derivative on any intermediate  vertical line is an interpolation of 
those from (\ref{eq:mark0}) and  (\ref{eq:marker}).  
In particular, we have that 
\begin{equation} \label{eq:marker1}
\left|\frac{\zeta_{\textrm{Ruelle}}'(z)}{\zeta_{\textrm{Ruelle}}(z)}\right|
\leq C_\# |z|^{\gamma_1}
\end{equation}
for $\Re(z) > \sigma_3:=  \htop - \frac{\tau_*}{8}$, say,  where 
$0 < \gamma_1 := \frac{\tau_*}{2 \tau_* + 4} < 1$.

Starting from \eqref{eq:marker1} we follow a classical approach in number theory. Let
\[ \begin{split}  \psi(T ) & \doteq \sum_{e^{n \htop\lambda(\tau)} \leq T} \htop\lambda(\tau) \\
 \psi_1(T) & \doteq \int_1^T \psi(x) dx =  \sum_{e^{n \htop) \lambda(\tau)} \leq T} \htop\lambda(\tau)  (T - e^{n \htop \lambda(\tau)}) \\
  \pi_0(T) & \doteq  \sum_{e^{n \htop \lambda(\tau)} \leq T}1 \hbox{ and }
 \pi_1(T) \doteq \sum_{e^{ \htop \lambda(\tau)} \leq T}1,
\end{split} \]
where the finite summation in each case is over prime closed orbits $\tau \in \cT_p$ for the  associated geodesic flow, 
and $n \geq 1$, subject to the bound in terms of $T$.
Next  we recall the following simple complex integral:  For $d_* > 1$ we have that
\begin{equation}\label{eq:mark1}
 \frac{1}{2\pi i} \int_{d_*-i\infty}^{d_*+i\infty} \frac{y^{z+1}}{z(z+1)} dz
 = \begin{cases}
 0 &\hbox{ if } 0 < y \leq 1 \\
 y-1 & \hbox{ if } y > 1
 \end{cases}.
\end{equation}
If we denote 
$\zeta_0(z) = \zeta_{\Ruelle}(\htop z)$ then we can write
\begin{equation}\label{eq:mark2}
\psi_1(T) =    \int_{d_*-i\infty}^{d_*+i\infty} \left( - \frac{\zeta'_0(z)}{\zeta_0(z)} \right)
\frac{T^{z+1}}{z(z+1)}dz 
\end{equation}
for any $d_*>1$ sufficiently large. 
 This comes by a term by term application of \eqref{eq:mark1}  to 
$$
- \frac{\zeta'_0(z)}{\zeta_0(z)} = \sum_{n=1}^\infty \sum_{\tau} \htop e^{-zn\htop\lambda(\tau)}.
$$
(N.B. The summation over the natural numbers in this identity comes from the second expression in \eqref{eq:ruelle-product},  not non-primitive orbits which are not being included in the counting).
Moving the line of integration from $\Re(z) = d_*>1$ to $\Re(z) = c$ gives
\begin{equation}\label{eq:mark3}
\psi_1(T) = \frac{T^2}{2} +  \int_{c-i\infty}^{c+i\infty} \left( - \frac{\zeta'_0(z)}{\zeta_0(z)} \right)
\frac{T^{z+1}}{z(z+1)}dz
\end{equation}
provided  $\sigma_3/\htop < c <1$.
In particular, we use 
\eqref{eq:marker1} to guarantee convergence of the integral in \eqref{eq:mark3}.
Note that the term $\frac{T^2}{2}$ comes from the line of integration crossing the pole for $\zeta_0(z)$ (now at $z = 1$). Moreover, 
the integral yields a factor $T^{c+1}$  i.e.  $\psi_1(T) = \frac{1}{2}T^2 + O(T^{c+1})$. Next we have the following asymptotic estimates.
\begin{lem} We can estimate, provided $1>c_1>c > 0$, that
 $$\psi(T) = T + O(T^{(c+1)/2});$$
 $$\pi_0(T) = \hbox{\rm li}(T) + O\left(\frac{T^{(c+1)/2}}{\log T}\right);$$
 $$\pi_1(T) = \hbox{\rm li}(T) + O(T^{(c_1+1)/2}).$$
\end{lem}
Since the proof of the above Lemma is completely analogous to the case of prime numbers, we omit it and refer the reader  to \cite[Chapter 4, Section 4]{Ellison85}. Finally, we can write
 $$
 \pi(T) = \pi_1\left(e^{\htop T} \right) = 
  \hbox{\rm li}(e^{\htop T} ) + O(e^{[\htop(c_1+1)/2] T})
 $$
 and the theorem follows with 
 $\htop(c_1+1)/2  = \htop- \delta$.
\end{proof}

\section{Cones and Banach spaces} \label{banachspace}

We want to introduce appropriate Banach spaces of currents\footnote{ See Federer \cite[Sections 4.1.1 - 4.1.7]{Federer69} for a detailed presentation of currents.} over a $d$-dimensional smooth compact Riemannian manifold $M$. We present a general construction of such spaces based only on an abstract cone structure. This is very convenient since later on we will apply this construction twice: once on $M$ to treat certain operators, and once on $M^2$ to treat some related, but different, operators.

The basic idea, going back to \cite{Blank02}, is to consider ``appropriate objects" that are ``smooth" in the unstable direction while  being  ``distributions" in the stable direction. This is obtained by defining norms in which such objects are integrated, against smooth functions, along manifolds close to the stable direction. Unfortunately, the realization of this program is fairly technical, since we have to define first the class of manifolds on which to integrate, then determine which  are the relevant objects and, finally,  explain what we mean by ``appropriate".
 
 \subsection{Charts and Notation}\label{subsec:charts}

We start with some assumptions and notation. More precisely, for any $r\in\bN$, we assume that there exists $\delta_0>0$, such that, for each $\delta\in (0,\delta_0)$  and $\rho\in (0,4)$, there exists an atlas $\{(U_\alpha, \Theta_\alpha)\}_{\alpha \in \cA}$,  where $\cA$ is a finite set, such that\footnote{ The relations are meant to be valid where the composition is defined. The $\|\cdot\|_{\Cs^r}$ norm is precisely defined in \eqref{def:crnorm}. Note that the explicit numbers used (e.g. $2, 30, \dots$) are largely arbitrary provided they satisfy simple relations that are implicit in the following constructions.} 
\begin{equation} \label{eq:chartproperty} 
\left\{ \begin{array}{ll}
 \Theta_\alpha (U_\alpha) =B_d(0,30\, \delta\sqrt{1+\rho^2}) \\
 \cup_\alpha \Theta_\alpha^{-1}(B_d(0,2 \delta))=M \\
\Theta_\alpha \circ \Theta_\beta^{-1}( \tilde{x} , x_d  + t ) = \Theta_\alpha \circ \Theta_\beta^{-1}( \tilde{x} , x_d ) + (0,t ) \;,\quad t\in\bR\\
\|(\Theta_\alpha)_*\|_\infty+\|(\Theta_\alpha^{-1})_*\|_\infty\leq 2\;;\quad \|\Theta_\alpha\circ\Theta_\beta^{-1}\|_{\Cs^r}\leq 2.  \\
\end{array} \right.
\end{equation}
Note that the above is equivalent to the existence of a global $\Cs^r$
vector field $V$ such that $\Theta_\alpha$ are flow box charts for the flow generated by $V$; i.e. for all $\alpha \in \cA$, 
$(\Theta_\alpha^{-1})_*(\frac{\partial}{\partial {x_d}})= V$.\footnote{ Indeed, the first, second and last relation can always be  satisfied (e.g. consider sufficiently small charts determined by the exponential map). Given $V$, the third conditions can be satisfied by applying the flow box Theorem.}

Let $\{\psi_\alpha \}_{\alpha \in \cA} $ be a smooth partition of
unity subordinate to  our atlas such that $ \textrm{supp}(\psi_{\alpha}) \subseteq
 \Theta_{\alpha}^{-1}(B_d(0,2\delta)) \subset U_{\alpha} $  and $\psi_\alpha|_{\Theta_\alpha^{-1}(B_d(0,\delta))}=1$.
In addition,  we assume $d_1 + d_{2} = \dim(M)$ are given, and let
\begin{equation}\label{eq:cones}
 \cC_\rho\doteq\{(s,u) \in \bR^{d_1} \times \bR^{d_2} : \|u\|\leq \rho\|s\|\}.
\end{equation}
Next, we assume that there exists $4>\rho_+>\rho_->0$ such that, for all $\alpha,\beta\in\cA$, 
\begin{equation}\label{eq:cone-condition}
\cC_0\subset (\Theta_\alpha)_*\circ (\Theta_\beta^{-1})_*\cC_{\rho_-}\subset \cC_{\rho_+},
\end{equation}
when it is well defined. Note that, by compactness, there must exists $\rho_1>0$ such that  $\cC_{\rho_1}\subset (\Theta_\alpha)_*\circ (\Theta_\beta^{-1})_* C_{\rho_-}$.

This concludes the hypotheses on the charts and the cones. 
For each $\ell \in \left\{0, \ldots,d \right\}$, let $\wedge^\ell(T^*M)$ be the algebra of
the exterior $\ell$-forms on $M$. We write
$\Omega^\ell_r(M)$ for the space of $\Cs^r$ sections of
$\wedge^\ell(T^*M)$.
Let $ h =\sum_{\alpha \in \cA } \psi_{\alpha} h \doteq \sum_{\alpha \in \cA} h_{\alpha} $ so that  
$h_{\alpha} \in \Omega^{\ell}_{r}(U_{\alpha})$. 

Let $\left\{ e_1, \ldots , e_d \right\}$ be the canonical basis of $\bR^d$. For all
$\alpha \in \cA$ and $x \in U_\alpha$ consider the basis of $T_x U_\alpha$  given by $\left\{ (\Theta_\alpha^{-1})_*\frac{\partial}{\partial x_1} , \ldots ,(\Theta_\alpha^{-1})_*\frac{\partial}{\partial x_d} \right\} $.
Let $\left\{ \hat{e}_{\alpha,1} , \ldots , \hat{e}_{\alpha,d} \right\}$
be the orthonormal basis of  $T_x U_\alpha$ obtained from the first one by applying the Gram-Schmidt procedure, setting as first element of the algorithm   $ \hat{e}_{\alpha,d} = (\Theta_\alpha^{-1})_*\frac{\partial}{\partial x_d} / \left\| 
(\Theta_\alpha^{-1})_*\frac{\partial}{\partial x_d} \right\|$.
Let $ \omega_{\alpha,1} , \ldots , \omega_{\alpha,d} $ be the dual basis of  $1$-forms such that
$\omega_{\alpha,i} (  \hat{e}_{\alpha,j} ) = \delta_{i,j} $. Thus we can
define a scalar product on $T_x^*U_{\alpha}$ by the expression
$\langle \omega_{\alpha,i} , \omega_{\alpha,j} \rangle = \delta_{i,j} $. Note that the above construction ``respects'' the special direction $V$.

Let $\cI_\ell=\{\bar{i} = (i_1, \dots, i_\ell)\in\{1,\dots,d\}^\ell\;:\; i_1 < i_2 < \dots < i_\ell\}$ be the set of  $\ell$-multi-indices
ordered by the standard lexicographic order. Let 
$e_{\bar  i} \doteq e_{i_1} \wedge\dots\wedge e_{i_\ell}$ in $\wedge^{\ell}(\bR^d)$  and $dx_{\bar i} \doteq dx_{i_1}\wedge\dots\wedge dx_{i_\ell} \in \wedge^{\ell} (\bR^d)^*$ so that  $dx_{\bar{i}}( e_{\bar{j}} )= \delta_{\bar{i},\bar{j}}$. Let
$\{ e_{\alpha,\bar{i}} \} \subset \wedge^{\ell}(T_x U_\alpha)$ and 
$\{\omega_{\alpha,\bar{i}} \} \subset \wedge^{\ell}(T_x^* U_\alpha)$ be defined in the same way starting from $\ell_{\alpha,i},\omega_{\alpha,i}$.

Given $h \in\Omega^\ell_r(U_\alpha)$  and $(\Theta_\alpha^{-1})^* h\in\Omega^\ell_r(B_d(0,30\delta\sqrt{1+\rho^2}))$
 we will write
\[
h =\sum_{\bar{i} \in \cI_{\ell}} h_{\alpha,\bar{i}}\omega_{\alpha,\bar{i}} .
\]
As usual we define the scalar product
\begin{equation}\label{eq:scalar0}
\langle h, g\rangle_{\Omega^{\ell}}  \doteq \int_M \langle h , g \rangle_x \omega_M(x)
\end{equation}
where $\omega_M$ is the Riemannian volume form on $M$ and  $\langle h,g \rangle_x$ is the usual scalar product for forms\footnote{ In the following we will drop the subscript $x$ in the scalar product and the subindex $M$ in the volume form whenever it does not create confusion.} (see \eqref{eq:scalarproduct} for a precise definition).

In the sequel we will often restrict ourselves to forms ``transversal" to the flow, that is
\begin{equation}\label{eq:ooo}
\Omega^{\ell}_{0,r}(M) \doteq \left\{h\in \Omega^{\ell}_{r}(M)\;:\; h(V , \ldots ) = 0\right\}.
\end{equation}
\begin{rmk}\label{rem:form-details} We use the convention $\Omega^0_{0,r}=\Omega^0_r$. Note that $\Omega_{0,r}^d=\{0\}$ and, more generally if $h\in\Omega^\ell_{0,r}$, then $h=\sum_{\bar{i} \in \cI_{\ell}^-} h_{\alpha,\bar{i}}\omega_{\alpha,\bar{i}}$, where $\cI_\ell^-=\{\bar{i} = (i_1, \dots, i_\ell)\in\{1,\dots,d-1\}^\ell\;:\; i_1 < i_2 < \ldots < i_\ell\}$. See Remark \ref{rmk:oldbpq} for further comments.
\end{rmk}
For $f : B_{d}(0,\delta) \to \mathbb{R}^{\bar d}$, $\bar d\in\bN$,  we use the following 
$\Cs^r$-norm
 \begin{equation} \label{def:crnorm} 
\| f \|_{\Cs^{r}}  =  \begin{cases}
\sup_x\|f(x)\|&\mathrm{if} \,r=0  \\
 \|f\|_{\Cs^0}+\sup_{x,y\in B_d(0,\delta)} \frac{ \|f(x)-f(y)\|}{\|x-y\|^{r}} &\mathrm{if} \, 0<r < 1 \\  
 \sum_{i = 0}^{\lfloor r\rfloor } 2^{\lfloor r\rfloor - i} \sum_{k_1,\dots k_i}\left\| \partial_{x_{k_1}}\cdots\partial_{x_{k_i}}f \right\|_{\Cs^{r - \lfloor r\rfloor} }&\mathrm{if} \,r \geq 1  .
\end{cases}
\end{equation}
\begin{rmk} \label{foo:r-norm}The reader can check, by induction, that for such a $\Cs^r$ norm, for $r\in\bRp$ and $\bar d=1$, we have  $\|fg\|_{\Cs^r} \leq \|f\|_{\Cs^r}\|g\|_{\Cs^r}$ and, for all $\bar d\in\bN$, $ \| f \circ g  \|_{\Cs^r} \leq C_\#\sum_{i=0}^{r} \|f \|_{\Cs^r} \|Dg\|_{\Cs^{r-1}} \cdots \|Dg\|_{\Cs^{r-i}}$.
\end{rmk}

\subsection{Banach Spaces}\label{subsec:banach}
 
Given the above setting, we are going to construct several Banach spaces. The strategy is to first define appropriate norms and then close the space of $\ell$-forms in the associated topology.

Fix $L_0>0$. For each $L>L_0$, let us define 
\begin{equation} \label{def:leaves} \begin{split}
\mathcal{F}_r(\rho,L) \doteq\{  F\,: &\, B_{d_1}(0,6\delta)  \to \mathbb{R}^{d_2}  \; : \; F(0)=0  ; \\ 
& \|DF\|_{\Cs^{0}(B_{d_1}(0,6\delta)) }\leq \rho \, ; \, \|F\|_{\Cs^{r}(B_{d_1}(0,6\delta) )}\leq L\}.
\end{split} \end{equation} 
For each $F\in\mathcal{F}_r(\rho,L)$, $x \in \bR^d$, $\xi \in \bR^{d_1}$, let  $G_{x,F}(\xi): B_{d_1}(0,6\delta) \to
\bR^{d}$ be defined by  $ G_{x,F}(\xi) \doteq x+(\xi,F(\xi))$.

Let us also define $\widetilde\Sigma(\rho, L) \doteq \{G_{x,F}:x \in B_{d_1}(0,2\delta), F \in \mathcal{F}_r (\rho,L)\}$.
\begin{rmk} \label{rem:rhoL} All the present constructions will depend on the parameters $\rho,L$. We will often not make it explicit,  to ease notation and since many computations hold for any choice of the parameters, 
but  we will state when a particular choice of such parameters is made.
\end{rmk}
For each $\alpha \in \cA$, and $G \in \widetilde\Sigma$ we define the leaf\footnote{ These $W_{\alpha,G}$ are not to be confused with stable manifolds.}
$W_{\alpha,G}=\{\Theta^{-1}_\alpha \circ G(\xi)\}_{\xi\in B_{d_1}(0 ,3\delta)}$ and the enlarged leaf $W_{\alpha,G}^+=\{\Theta^{-1}_\alpha \circ G(\xi)\}_{\xi\in B_{d_1}(0,6\delta)}.$
For each $\alpha\in\cA$,  $G\in\widetilde \Sigma$ note that $W_{\alpha, G} \subset \whU_{\alpha}\doteq \Theta_\alpha^{-1}(B_d(0,6\delta\sqrt{1+\rho^2})) \subseteq U_\alpha $.
Finally, we define $\Sigma_\alpha= \cup_{G\in\widetilde\Sigma_\alpha}\{W_{\alpha,G}\}$.  
We will be interested in the set of manifolds $\Sigma=\cup_{\alpha\in\cA}\Sigma_\alpha$ (our sets of ``stable" leaves).

Also, for each $G \in\widetilde \Sigma_\alpha$ we denote by $\widehat{\Gamma}^{\ell,s}_c(\alpha,G)$  the $\Cs^s$ sections of the fiber bundle on $W^+_{\alpha,G}$,  with fibers $\wedge^\ell(T^*M)$, which vanish in a neighborhood  of $\partial W_{\alpha,G}$. We define the  norm
\begin{equation} \label{testnorm} 
\| g \|_{\widehat{\Gamma}^{\ell,s}_c (\alpha,G)} \doteq 
\sup_{\bar{i}}  \| g_{\alpha,\bar{i}} \circ \Theta_{\alpha}^{-1} \circ G \|_{\Cs^s(B_{d_1}(0,2\delta))}.
\end{equation}
Consistent with this choice, we equip $\Omega^{\ell}_r (M) $ with the norms, for $s \leq r$, 
\begin{equation} \label{def:manifoldnorm}
\| h \|_{\Omega^\ell_s} = \sup_{\alpha \in \cA, \bar{i} \in
  \cI_{\ell}}  \| h_{\alpha,   \bar{i}} \circ \Theta_{\alpha}^{-1}
\|_{\Cs^s(\bR^d)}.
\end{equation} 
Let  $\widehat{\cV}^{s}(\alpha,G)$ be the set of $\Cs^{s}(U_{\alpha,G}) $ vector fields, where $ U_{\alpha, G} $ is any open set such that $ U_{\alpha} \supset U_{\alpha, G} \supset  W_{\alpha,G}^+$.

To allow enough flexibility in the previous construction (flexibility that will be essential in Section \ref{sec:splitting}) we introduce the possibility of further choices: we choose sets $\Gamma^{\ell,s}_c(\alpha,G)$ that are dense in $\widehat{\Gamma}^{\ell,s}_c(\alpha,G)$ in the $\Cs^s$ norm for each $s<r$. Also we choose sets $ \cV^{s}(\alpha,G)\subseteq \widehat{\cV}^{s}(\alpha,G)$ that contain the push forward of any constant vector field under the coordinate map $\Theta_\alpha^{-1}$ and with the property that there exists $C_s>0$ such that, for each $v,w\in \cV^{s}(\alpha,G)$,  $C_s\cdot [v,w]\in \cV^{s-1}(\alpha,G)$. 
Finally, we ask the following extension property: there exists $C_s\in (0,1)$ such that for each
\begin{equation}\label{eq:extensionproperty}
v\in \cV^{s}(\alpha,G)\Longrightarrow \exists\; \bar v\in\cV^s(M):\begin{cases} &\hskip-.3cm \bar v|_{U_{\alpha,G}}=v\\
& \hskip-.3cm\forall \beta\in\cA, G'\in \widetilde\Sigma, C_s \bar v\in\cV^{s}(\beta,G'),
\end{cases}
\end{equation}
where by $\cV^s(M)$ we mean the $\Cs^s$ vector fields on $M$.
\begin{rmk} The norms and the Banach spaces we are going to define should have an index specifying their dependencies on the choices of 
$\{\Gamma^{\ell,s}_c(\alpha,G)\}$ and $\{\cV^{s}(\alpha,G)\}$. We choose to suppress them to ease notation, since this creates no confusion.
\end{rmk}
Let $\volform$ be the $d_s$ volume form induced  on $W_{\alpha,G}$ by the push forward of Lebesgue measure via the chart $\Theta_\alpha^{-1}$.  Write $L_v$  for the Lie derivative along a vector field $v$. 
Finally, for all $\alpha \in \cA$,  $G\in \Sigma_\alpha$, 
$ g \in \Gamma^{\ell,0}_{c}(\alpha,G)$, $\bar v^p=(v_1,\dots,v_p)\in  \cV^{s}(\alpha,G)^p$ and $h \in \Omega^{\ell}_{r}(M)$ we define
\begin{equation} \label{eq:ldefinition}
 J_{\alpha,G,g, \bar v^p}(h) \doteq \int_{W_{\alpha,G}} \langle g,  L_{v_1} \cdots L_{v_p}
h\rangle \; \volform  \in \bR.
\end{equation}
Next, for all 
$p \in \mathbb{N}$, $q\in \bRp$, $ p + q < r-1 $,
$\ell \in \{ 0, \ldots , d \} $, let
\begin{equation}\label{eq:pre-dual}
\begin{split}
\mathbb{U}_{\rho,L, p,q,\ell} \doteq \bigg\{J_{\alpha,G,
    g,\bar v^p}\;\big|&\; \alpha\in\cA, G\in\Sigma_\alpha(\rho, L), g\in \Gamma^{\ell,p+q}_c, v_j\in \cV^{p+q},\\
 & \;\;   \|g\|_{\Gamma^{\ell,p+q}_{c}(\alpha,G)}\leq 1, \|v_j\|_{\Cs^{q+p}(U_{\alpha,G})}
\leq 1 \bigg\}, 
\end{split}
\end{equation}
where, for $v\in \cV^s(\alpha,G)$, $\|v\|_{\Cs^s(U_{\alpha,G})}=\sup_{\alpha,i}\|\langle v,e_{\alpha,i}\rangle\circ \Theta_\alpha^{-1}\|_{\Cs^s}$.

Lastly, for all $\rho\in[\rho_-,\rho_+]$, $L\geq L_0$, $p \in \bN,\, q\in \bRp$ and $h\in\Omega_{p+q}^{\ell}(M)$ we define the following norms
\begin{equation} \label{eq:norms0}
\|h\|^-_{\rho, L, p,q,\ell} \doteq \sup_{ J \, \in \, \mathbb{U}_{\rho,L, p,q,\ell} } J (h)
\quad \hbox{ and } \quad \|h\|_{\rho,L,p,q, \ell} \doteq \sup_{n\leq p}  \|h\|^-_{\rho,L,n,q,\ell}.
\end{equation}

\begin{rmk} \label{rem:rho} Let $L_+$ be such that $W\in\Sigma_\alpha(\rho_-,L_0)$ and $W\cap \widehat U_\beta\neq \emptyset$ imply $W\in \Sigma_\alpha(\rho_+,L_+)$.\footnote{ Such a $L_+$ exists by the third equation of \eqref{eq:chartproperty}.} We find it convenient to set $\pqnorm[-,p,q,\ell]{\cdot}=\pqnorm[\rho_-,L_0, p,q,\ell]{\cdot}$ and $\pqnorm[+, p,q,\ell]{\cdot}=\pqnorm[\rho_+,L_+,p,q,\ell]{\cdot}$. In the above norms and objects we will suppress systematically the indexes $\rho,L, +,-$, when it does not create confusion. In particular, any statement about a norm without $\pm$ is either meant for each $\rho\in[\rho_-,\rho_+],\;L\in[L_0,L_+]$ or for the $-$ norm. The introduction of these annoying $\pm$-norms has the only purpose to allow a simple use of mollificators, which in general are not bounded operators (see Appendix \ref{subsec:average} for details).
\end{rmk}
\begin{defn}\label{def:banach}
For all $p\in\mathbb{N}$, $q\in\bRp$, $\ell \in \{0, \ldots, d-1 \}$  we define the spaces 
$\mathcal{B}^{p,q, \ell}$ to be the closures   of $\Omega^{\ell}_{0,r}(M)$ with respect to the norm $\pqnorm[-,p,q,\ell]{\cdot}$  and the spaces 
$\mathcal{B}^{p,q, \ell}_+$ to be the closures   of $\Omega^{\ell}_{0,r}(M)$ with respect to the norm $\pqnorm[+,p,q,\ell]{\cdot}$. 
\end{defn}
\begin{rmk} Note that, by definition, the forms in $\cB^{p,q,\ell}, \cB^{p,q,\ell}_+$ are zero in the flow direction. This retains the relevant properties of Poincar\'e sections while working directly with the flow.
\end{rmk}

\begin{rmk}
For the norms just defined the reader can easily check that 
\begin{equation}
\label{eq:0q-norm} \begin{split}
 &\|h\|_{p,q,\ell}
\leq \|h\|_{p+1,q, \ell} \quad \hbox{\rm and }\quad  \|h\|_{p,q+1, \ell} \leq \|h\|_{p,q, \ell}  \\
& C_{p,q} \|h\|_{p,q,\ell}^-\leq \sup_{\{v_1,\dots,v_p\in \cV^{p+q}\;:\;   \|v_j\|_{\Cs^{p+q}}\leq 1\}} \|L_{v_1}\cdots L_{v_p} h\|_{0,p+q}^- \leq  \|h\|_{p,q,\ell}^-\\
&  \pqnorm[-,p,q,\ell]{h}\leq \pqnorm[+,p,q,\ell]{h}\; ;\quad \pqnorm{h}\leq C_\#\|h\|_{\Omega^\ell_p} . \\
\end{split} 
\end{equation}
\end{rmk}

\begin{rmk} \label{rmk:oldbpq}
The above spaces are the natural extensions of the spaces $\cB^{p,q}$  in \cite{GouezelLiverani06} to the case 
of $\ell$-forms.  
There the Banach spaces $\cB^{p,q}$ were defined  as the closure of $\Cs^{\infty}(M,\bR)$ with respect to
the following norm\footnote{ In \cite{GouezelLiverani06} the 
coordinate charts are chosen with slightly different properties. However for our purposes they are equivalent.}  
\begin{equation*}
\| h \|_{p,q} = \sup_{0 \leq k \leq p} \sup_{\stackrel{\scriptstyle \alpha \in \cA }{\scriptstyle G \in \tilde\Sigma}}
\;\sup_{\stackrel{\scriptstyle v_1, \ldots ,v_k \in \cV^{p+k}(\alpha,G)}{\scriptstyle  |v_i|_{\Cs^r} \leq 1 }} 
\sup_{\stackrel{\scriptstyle \varphi \in \Cs^{q}_0(W_{\alpha,G},\bR)}{|\varphi|_{\Cs^{k + q}} \leq 1}} 
\int_{W_{\alpha,G}} L_{v_1} \cdots L_{v_k} (h) \cdot \varphi \volform .
\end{equation*}
In particular, we can construct an isomorphism between the Banach space $\cB^{p,q}$ in \cite{GouezelLiverani06} and 
the present Banach space $\cB^{p,q,d-1}$. In fact, let $i_V:\Omega^{\ell+1}_{r}(M)\to\Omega^{\ell}_{0,r}(M)$ be the
interior product defined by  $i_V(h)(v_1,\dots,v_\ell)=h(V,
v_1,\dots,v_\ell)$. If $h\in \Omega^{\ell}_{0,r}(M)$, then
 $(-1)^{\ell}i_V(h \wedge dV)=h$ where $dV$ is any one form such that
 $dV(V)=1$. That is  $i_V(\Omega^{\ell+1}_{r}(M)) = \Omega^{\ell}_{0,r}(M)$. 
Since $i_V(\Omega^{\ell+1}_{0,r}(M)) = 0 $, we can define $i_{V,0}:
\Omega^{\ell + 1}_{r}(M) / \Omega^{\ell+1}_{0,r}(M)  \to \Omega^{\ell}_{0,r}(M)$ and obtain by a standard algebraic construction
that $i_{V,0}$ is a natural isomorphism. Next, we define
\begin{equation}\label{eq:vol-at-last}
\widetilde\omega \doteq i_{V}\omega,
\end{equation}
where $\omega$ is the Riemannian volume, and  the map ${\boldsymbol i}:\Cs^r(M)\to \Omega^{d-1}_{0,r}(M)$
by  ${\boldsymbol  i}(f) \doteq f \cdot \widetilde \omega$. Note that ${\boldsymbol i}$ is an isomorphism since $\Omega^d_{0,r}=\{0\}$.
It is easy to check that  ${\boldsymbol  i}$ extends to an isomorphism between $\cB^{p,q}$ and $\cB^{p,q,d-1}$.
It is even easier to construct an isomorphism between $\cB^{p,q}$ and $\cB^{p,q,0}$. Yet, such an isomorphism is not relevant here, indeed the reader can check that the transfer operators defined in section \ref{subsec:transfer} correspond to the transfer operators studied in \cite{GouezelLiverani06} only when acting on $\cB^{p,q,d-1}$.
\end{rmk}
We now prove some properties of the spaces $\mathcal{B}^{p,q,\ell}$. To this end we will use some
estimates on how fast an element of the space can be approximated by smooth forms. Such estimates are proven in Appendix \ref{subsec:average}. 

Given a form $h \in \Omega^\ell$, we can define a functional by
\begin{equation} \label{def:formsduality}
[\jmath(h)](g) \doteq   \langle h, g\rangle_{\Omega^{\ell}}   \qquad\; \textrm{where} \; \qquad g \in  
\Omega^\ell_s(M).
\end{equation}
The space of such functionals, equipped  with the $*$-weak topology of $\Omega^\ell_s(M)'$,\footnote{ As usual, given a Banach space $\cB$, by $\cB'$ we mean the dual space.} gives rise to the space $\cE_{s}^{\ell}$  of {\em currents} of regularity $s$.
 
The following extends \cite[Proposition 4.1]{GouezelLiverani06} and \cite[Lemma 2.1]{GouezelLiverani06}. 

\begin{lem}\label{lem:currents} For each $\ell \in\{0,\dots,d - 1\}$, there is a canonical injection
from the space $\mathcal{B}^{p,q,\ell}$ to a subspace of $\cE_{p+q}^{\ell}$.
\end{lem}
\begin{proof}
Since we can foliate $M$  by manifolds in $\Sigma$, given the definitions \eqref{def:formsduality} and \eqref{eq:scalar0} we have
\begin{equation}\label{eq:current-0}
[\jmath(h)](g) \leq C_\# \pqnorm{h}\|g\|_{\Omega^\ell_{p+q}(M)}.
\end{equation} 
Thus $\jmath$ can be extended to a continuous immersion of $\cB^{p,q,\ell} $ in $\cE_{p+q}^{\ell}$.
 
To show injectivity consider a sequence  $\{h_n\}\subset
\Omega^{\ell}_{0,p+q}$  that converges to $h$ in $\cB^{p,q,\ell}$ such that $\jmath(h)=0$. Let $W_{\alpha,G}\in\Sigma_\alpha$.
In the following it is convenient to introduce the operators $\widetilde \bM_{\alpha,\ve}$ defined as in Defintion \ref{def:molli} with $\bar\Psi(x,y)_{\bar i,\bar j} =\tilde \psi_\alpha(\Theta_\alpha^{-1}(x))\delta_{\bar i,\bar j}$, $\tilde\psi_\alpha(\xi)=1$ if $\|\xi\|\leq 3\sqrt{1+\rho_-^2}$ and $\tilde\psi_\alpha(\xi)= 0$ if $\|\xi\|\geq 6\sqrt{1+\rho_-^2}$. Note that, for each $h'\in \Omega^{\ell}_{0,p+q}$,
\begin{equation}\label{eq:test-test}
\int_{W_{\alpha,G}}\langle  g , \widetilde\bM_{\alpha,\ve} h' \rangle \volform =[\jmath(h')](g_\ve)
\end{equation}
where 
\[ 
g_\ve(x) \doteq \sum_{\bar i}\omega_{\alpha, i}(x) J\Theta_\alpha(x) \int_{W_{\alpha,G}}\hskip-.3cm \volform (y) \;\tilde\psi_\alpha(y)\kappa(\Theta_\alpha(x)-\Theta_\beta(y)) \langle \omega_{\alpha, i}, g\rangle_{y} \in \Omega^\ell_r.
\]
By Lemma \ref{lem:molli} and Remark \ref{rem:loss} we have that
\[ 
\begin{split}
 & \int_{W_{\alpha, G}}\langle g , h_n   \rangle \volform = 
\lim_{\ve\to 0}\int_{W_{\alpha,G}}  \langle  g , \widetilde \bM_{\alpha,\ve} h_n \rangle \volform 
 =\lim_{\ve\to 0}[\jmath(h_n)](g_\ve),
\end{split} 
\] 

Moreover, by equation \eqref{eq:test-test}, Lemma \ref{lem:molli} and Remark \ref{rem:loss} it follows that 
\[ 
\begin{split}
\left| [\jmath(h_n)](g_\ve)-[\jmath(h_m)]( g_\ve )\right| &=\left| \int_{W_{\alpha, G}}\langle g , \widetilde \bM_{\alpha,\ve}(h_n-h_m)   \rangle \volform\right|\\
&\leq C_\# \| g \|_{\Gamma^{\ell,q}_c(\alpha,G)} \|h_n-h_m\|_{0,q,\ell}. 
\end{split}
\end{equation*}
Thus, we can exchange the limits with respect to  $n$ and $\ve$ to obtain
\[ \begin{split}
\int_{W_{\alpha, G}} \langle g , h   \rangle \volform &=\lim_{n \to \infty}\int_{W_{\alpha, G}}\langle g , h_n   \rangle \volform \\ 
 &= \lim_{\ve\to 0}\lim_{n\to\infty}[\jmath(h_n)]( g_\ve ) =
 \lim_{\ve\to 0}[\jmath(h)]( g_\ve )=  0.
\end{split} \]
which implies $\|h\|_{0,q,\ell}=0$. 
By similar computations, using equations \eqref{eq:lie-molli}  and \eqref{eq:molli3} to deal with the derivatives, we obtain
$\|h\|_{p,q,\ell}=0$. Thus  $\jmath$ is injective and we obtain the
statement of the theorem. 
\end{proof}
We conclude this section by proving a compactness result which is essential in the implementation of the usual Lasota-Yorke strategy.  The proof is exactly as the proof of \cite[Lemma 2.1]{GouezelLiverani06}.
\begin{lem}\label{lem:compact}
For each $q , p  > 0$, $p+q<r-1$, the unit ball of $\cB^{p,q,\ell}$ is relatively compact in $\cB^{p-1,q+1,\ell}$.
\end{lem}

\section[Transfer operators]{Transfer operators and Resolvents} \label{secresolvent}

Let $\phi_t: M \to M$ be a $\Cs^r$ Anosov flow on a smooth Riemannian $d$-dimensional
compact manifold with $r \geq 2 $. The flow induces canonically an action $\phi_t^*$ on  $\ell$-forms. Such an action has nice spectral properties only when acting on Banach spaces of the type described  in section \ref{banachspace}. To apply Section \ref{banachspace} to the present context it suffices to specify all the  choices involved in the definition of the norms $\pqnorm{\cdot}$. 

First of all,  the ``special" direction is obviously given by the vector field $V$ generating the flow. In addition, we note the following.
\begin{rmk}
Without loss of generality, we can assume
\begin{equation} \label{eq:anconditions} 
 \left\{  
\begin{array}{l}
D_0\Theta_{\alpha}^{-1}\{(0,u,0)\;:\; u\in{\mathbb{R}}^{d_u}\}=E^u(\Theta_{\alpha}^{-1}(0))\\
D_0\Theta_{\alpha}^{-1}\{(s,0,0)\;:\; s\in{\mathbb{R}}^{d_s}\}=E^{s}(\Theta_{\alpha}^{-1}(0))\\
\Theta_{\alpha}^{-1}((s,u,t))=\phi_t\Theta_{\alpha}^{-1}((s,u,0)). 
\end{array}  
\right.
\end{equation}
\end{rmk}
Let $d_1=d_s$ and $d_2=d_u+1$. Given the continuity of the stable and unstable distribution we can choose $\delta$ so that equations \eqref{eq:cones} and \eqref{eq:cone-condition} are satisfied with $\rho_-=1, \rho_+=2$ and, for all $\alpha\in\cA$ and $x\in U_\alpha$, $(\Theta_\alpha)_*E^s(x)\subset \cC_{\frac 12}$. Choose $L_0$ large enough so that all the stable manifolds belong locally to $\Sigma(1, L_0/2)$. Note that there exists $\vuo>0$ such that, for all $ t \geq \vuo $,\footnote{ Indeed, if $x\in U_\alpha$, $v\in\bR^d$ and $v=v^u+v^s$, $v^u\in (\Theta_\alpha)_*(E^u(x)\times E^c(x))$, $v^s\in (\Theta_\alpha)_*(E^s(x))$, then $\|(\phi_{-t}\circ \Theta^{-1})_*(v^u)\|\leq \Cnz $ while $\|(\phi_{-t}\circ \Theta^{-1})_*(v^s)\|\geq \Cnz e^{\lambda t}$ and the result follows  by the fourth line of \eqref{eq:chartproperty}.}
\begin{equation} \label{eq:coneconditions} 
(\Theta_\beta\circ \phi_{-t}\circ\Theta_\alpha^{-1})_*(\cC_2)\subset \cC_1 .
\end{equation}
In addition,  for each $v \in (\Theta_\alpha^{-1})_*\cC_2$, $t\in\bRp$ we have  $\| (\phi_{-t})_* (v) \|> C_\#e^{\olambda t}\|v \|$.
Finally, we set $\cV^s=\widehat{\cV}^s $ and $\Gamma^{\ell,q}_c=\widehat{\Gamma}^{\ell,q}_c$. 
\begin{lem}\label{lem:extension}
With the above choices condition \eqref{eq:extensionproperty} is satisfied.
\end{lem}
\begin{proof}
To prove the Lemma it suffices to have a uniform estimate on the norm of an extension of a vector field.
This is in general a hard problem, but here we can exploit the peculiarities of our situation that allows to give a simple proof based on the well known reflection method that goes back, at least, to \cite{Lichtestein29, Hestenes41, Seeley64}. 

First, notice that, given $W_{\alpha,G}$, we can make a uniformly bounded $\Cs^{r}$ change of coordinates $\Xi$ such that $\Xi\circ \Theta_\alpha (W_{\alpha,G}^+)=B(0,6\delta)\times \{0\}$. Given any  vector field $v\in\widehat\cV^{s}(\alpha,G)$ the domain $\Xi\circ \Theta_\alpha(U_{\alpha,G})$ will always contain a ``cylinder" of the type $Y_{\bar a}=B(0,6\delta)\times (\times_{i=1}^{d_2}[-a_i,a_i])$, for some $\bar a\in[0,\delta]^{d_2}$, while $v$ will be mapped into the extension of a function $f\in\Cs^{s}(Y_a,\bR^d)$ with $C_\# \|v\|_{\Cs^{s}(U_{\alpha,G})}\leq \|f\|_{\Cs^{s}(Y_a)}\leq C_\# \|v\|_{\Cs^{s}(U_{\alpha,G})}$. Clearly it suffices that we extend $f$ to a fixed, sufficiently large, domain so that, after multiplying it by a function equal one on $Y_{(2\delta,\dots,2\delta)}$ and supported in the image, under $\Xi$, of the domain of the chart, we obtain the wanted vector field by the push forward $(\Theta_\alpha^{-1}\circ\Xi^{-1})_*$ and setting it to zero on the rest of $M$.

Let $I_{\bar a, \beta,i,\pm}(x)=(x_1, \dots, x_{i-1}, -\beta(x_i\mp a_i)\pm a_i, x_{i+1},\dots, x_d)$. Then, given $g\in \Cs^{s}(Y_{\bar a})$, we can define, for $i>d_1$,
\[
g^i(x)=\begin{cases}g(x)&\text{ for }\quad x\in Y_{\bar a}\\
\sum_{k=0}^{\lfloor s\rfloor} b_k f\circ I_{\bar a, 2^k,i,+}(x)& \text{ for }\quad x\in Y_{\bar a^{i}}, x_i>0\\
\sum_{k=0}^{\lfloor s\rfloor} b_k f\circ I_{\bar a, 2^k,i,-}(x) &\text{ for }\quad x\in Y_{\bar a^{i}}, x_i<0
\end{cases}
\]
where $a^{i}_k=a_k$ for $k\neq i$ and $a^i_i=[1+2^{1-p-\lfloor q\rfloor}]a_i$ and $\sum_{k=0}^{\lfloor s\rfloor} b_k2^{kn}=1$ for all $n\in\{0,\dots,\lfloor s\rfloor\}$. Note that the previous condition determines the $b_k$ uniquely since the Vandermonde determinant $\det( 2^{ij})$ is well known to be non zero. It is easy to check that $g^i\in\Cs^{s}(Y_{\bar a^i})$ and $\|g^i\|_{\Cs^{s}(Y_{\bar a^i})}\leq C_s \|g\|_{\Cs^{s}(Y_{\bar a})}$. Proceeding in such a way we can extend $g$ to the domain $Y_{(A,\cdots, A)}$ for any fixed $A>0$. We are left with the boundary of the domain determined by $B(0,6\delta)$, which is smooth. By localizing via a partition of unity and a smooth change of coordinates we can reduce ourselves to the case of extending a function from the half space to the all space, which can be handled as above \cite{Seeley64}. This concludes the proof.
\end{proof}

\begin{rmk}\label{rmk:small-times}
By standard hyperbolicity estimates $\vuo$ in \eqref{eq:coneconditions}  can be chosen so that, for all $t\geq \vuo$, the image of a manifold in $\Sigma$ can be covered by a collections of manifolds in $\Sigma$ with a uniformly bounded number of overlaps. The fact that this may be false for small times is a little extra problem present in the study of flows with respect to maps. We will handle such a problem by a further modification of the Banach space, see \eqref{eq:strong-norms-sigh}. 
\end{rmk}
This concludes the definition of the Banach spaces. Next we define the action of the flow.

\subsection{Properties of the transfer operator}\label{subsec:transfer}
Let $h \in \Omega^{\ell}_{0,r-1} (M)$, then $\forall t \in \bR$
\begin{equation} \label{eq:inj}
\phi_t^*h(V,v_2, \ldots , v_\ell) = h(V,(\phi_{t})_* v_2 , \ldots  ,(\phi_{t})_* v_\ell ) = 0. 
\end{equation}
Thus $\phi_{t}^*(\Omega^{\ell}_{0,r-1}) = \Omega^{\ell}_{0,r-1} $ for all $t\in \bR$.
In \eqref{def:reducedoperator} we defined the  operators $\cL_{t}^{(\ell)}: \Omega^{\ell}_{0,r-1} (M) \to \Omega^{\ell}_{0,r-1} (M)$, $t\in\bRp$, by
\[ 
\cL_{t}^{(\ell)} h  \doteq   \phi_{-t}^* h. 
\]

Following the Remark \ref{rmk:oldbpq},  in the special case of a $d-1$ form,  for
 $h \in \Omega^{d-1}_{0,r-1}(U_{\alpha}) $  we have $ h = \bar{h} \widetilde\omega $
where $\bar{h} \in \Cs^s(M,\bR)$ and $\widetilde\omega$ is defined in
\eqref{eq:vol-at-last}. Then\footnote{ Indeed, for $\omega \in \Omega^{d}_{s}(M)$ we have $\det(D\phi_t)\omega=\phi_{t}^*\omega$. Hence by
  \eqref{eq:vol-at-last}  we have  $\phi_t^*\widetilde\omega(v_1,\dots,v_{d-1}) =\det(D\phi_t)\widetilde\omega(v_1,\dots,v_{d-1}).$
That is $\phi_t^*\widetilde\omega=\det(D\phi_t)\widetilde\omega$.}
\begin{equation} \label{eq:oldto}
\cL_t^{(d-1)} h= \bar{h}\circ\phi_{-t}\det(D\phi_{-t})\widetilde\omega\doteq(\cL_t \bar{h})\widetilde\omega.
\end{equation}
Thus we recover the transfer operator acting on densities studied in \cite{ButterleyLiverani07}.

We begin with a Lasota-Yorke type inequality for $\mathcal{L}_t^{(\ell)}$.
\begin{defn} Let $p \in \bN$, $q \in \bRp$, $0 \leq \ell \leq d - 1 $ and  $ 0<\lambda < \overline{\lambda} $ where $\overline{\lambda}$ is as in the Anosov splitting. Let us define
\[
\begin{split}
&\sigma_\ell \doteq \htop -\lambda |d_s-\ell |  \\
&\sigma_{p,q} \doteq \min\{p,q\}\lambda\;.
\end{split}
\]
\end{defn}
\begin{lem} \label{lem:LY}
There exists $\vuo>0$ such that, for each $p\in\bN ,q\in\bRp$, $p+q<r-1$, $\ell \in\{0,\dots,d-1\}$, $\lambda\in (0,\olambda)$ and $ t > \vuo$ the linear operators $\cL^{(\ell)}_t$ can be uniquely extended to a bounded operator\footnote{ Which, by a harmless abuse of notation, we still designate by 
$\cL^{(\ell)}_t$.} $\cL^{(\ell)}_t\in L( \mathcal{B}^{p,q,\ell}_- ,\mathcal{B}^{p,q,\ell}_+)$. More precisely,
\[
\pqnorm[+,0,q,\ell]{\cL_t^{(\ell)}h} \leq C_{q} e^{\sigma_\ell t}\pqnorm[-,0,q,\ell]{h} .
\]
In addition, if $0<q\leq r-2$, 
\[
\pqnorm[+,0,q,\ell]{\cL_t^{(\ell)}h} \leq  C_{q,\lambda} e^{(\sigma_\ell -\lambda q)t}\pqnorm[-,0,q,\ell]{h}+ C_{q,\lambda} e^{\sigma_\ell t}\pqnorm[-,0,q+1,\ell]{h}
\]
and, if $p,q>0$,
\[
\begin{split}
\pqnorm[+,p,q,\ell]{ \cL_t^{(\ell)} h }\leq & C_{p,q,\lambda} e^{ (\sigma_\ell - \sigma_{p,q})t} \pqnorm[-,p,q,\ell]{ h} + C_{p,q,\lambda} e^{\sigma_\ell t} \pqnorm[-,p-1,q+1,\ell] {h} \\
&+ C_{p,q,\lambda}e^{\sigma_\ell t} \pqnorm[-,p-1,q+1,\ell]{X^{(\ell)}h}.
\end{split}
\] 
\end{lem}
The proof of this lemma is the content of Subsection \ref{subsec:LY}.
In particular, note that from the first and last equations  of the above Lemma and equation \eqref{eq:0q-norm} we have,
\begin{equation}\label{eq:pq-growth}
\pqnorm{\cL_t^{(\ell)}h} \leq  C_{p,q}e^{\sigma_\ell t} \pqnorm{h}.
\end{equation}

\begin{rmk}
Up to now we have extended, yet followed closely, the arguments in  \cite{ButterleyLiverani07}.
Unfortunately, in Section $7$ of \cite{ButterleyLiverani07}, the authors
did not take into consideration that for $t \leq \vuo$, $\phi_{t}(W_{\alpha,G})$ is not necessarily controlled by the cones defined by \eqref{eq:cones}, and  thus could be an inadmissible  manifold.
Such an issue can be easily fixed, in the present context, by introducing a dynamical norm (similarly to \cite{BaladiLiverani11}) as done below. This however is not suitable for studying perturbation theory, to this end a more radical change of the Banach space is required \cite{ButteleryLiveranierrata}.
\end{rmk}

To take care of the $t\leq \vuo$, we introduce the dynamical norm $\spqnorm{\cdot}$. For each $h \in \Omega_{r}^\ell(M)$, we set
\begin{equation}\label{eq:strong-norms-sigh}
\spqnorm{h} \doteq \sup_{s \leq \vuo} \pqnorm{\cL^{(\ell)}_{s} h}. 
\end{equation}
Thus we can define $\widetilde{\cB}^{p,q,\ell} \doteq \overline{\Omega_{0,r}^\ell}^{\spqnorm{\cdot}}\subset \cB^{p,q,\ell}$. 

\begin{lem} \label{lem:LYstrong}
For each  $ p \in \bN, q \in \bRp$, $ p + q < r - 1$, $\ell \in \{ 0, \ldots, d -1 \}$,  $\lambda<\olambda$, for each $t\in\bRp$, we have that $\cL_t^{(\ell)}\in L( \widetilde{\cB}^{p,q,\ell}, \widetilde{\cB}^{p,q,\ell} )$ and, more precisely,
\[
\spqnorm[p,q,\ell]{\cL_t^{(\ell)}h} \leq  C_{p,q} e^{\sigma_\ell t}\spqnorm[p,q,\ell]{h}.
\]
Moreover, if $p,q> 0$,
\[
\begin{split}
\spqnorm[p,q,\ell]{ \cL_t^{(\ell)} h }\leq & C_{p,q,\lambda} e^{(\sigma_\ell -\sigma_{p,q})t} \spqnorm[p,q,\ell]{ h} + C_{p,q,\lambda} e^{\sigma_\ell t} \spqnorm[p-1,q+1,\ell]{h} \\
&+ C_{p,q}e^{\sigma_\ell t} \spqnorm[p-1,q+1,\ell]{ X^{(\ell)}h}.
\end{split}
\] 
In addition, $\{\cL_t\}_{t\in\bRp}$ forms a strongly continuous semigroup.
\end{lem}
\begin{proof} 
For $h\in\Omega_{0,r}^\ell$ equations \eqref{eq:pq-growth}, \eqref{eq:strong-norms-sigh} imply, for $t< \vuo$,
\[
\spqnorm{\cL_t^{(\ell)}h} \leq \max\left\{\spqnorm{h}, C_{p,q}e^{|\sigma_\ell|\vuo} \spqnorm{h}\right\}\leq C_{p,q}\spqnorm{h} ,
\]
while for $t\geq \vuo $ the required inequality holds trivially. The boundedness of $\cL^{(\ell)}_t$ follows. 

The second inequality follows directly from the above, for small times, and from Lemma \ref{lem:LY} for larger times.

To conclude, note that $\cL_t^{(\ell)}$ is strongly continuous on $\Omega_r^{\ell}(M)$ in the $\Cs^r$ topology.
Let $\{h_n\}\subset \Omega_{0,p}^\ell$ such that $\lim_{n\to\infty}\spqnorm{h_n-h}=0$, then, by the boundedness of $\cL^{(\ell)}_t$, 
\[ 
\spqnorm[p,q,\ell]{ \cL^{(\ell)}_t h-h} \leq C_\# \spqnorm[p,q,\ell]{h - h_n}+\|\cL^{(\ell)}_t h_n-h_n\|_{\Omega_p^{\ell}}
\]
which can be made arbitrarily small by choosing $n$ large and $t$ small. 
\end{proof}

\subsection{Properties of the Resolvent}\label{subsec:resol}
By standard results, see for example \cite{Davies1980}, the semigroups $\cL^{(\ell)}_t$ have generators $X^{(\ell)}$ which are closed operators on $\widetilde{\cB}^{p,q,\ell}$ such that $X^{(\ell)}\mathcal{L}_t^{(\ell)}=\frac{d}{dt}\mathcal{L}_t^{(\ell)}$.
Setting $R^{(\ell)}(z)=(z\Id-X^{(\ell)})^{-1}$ we have, for $\Re(z)>\sigma_\ell$, the following identity
\begin{equation}\label{eq:resoldef}
R^{(\ell)}(z)^n = \frac{1}{(n-1)!} \int_0^\infty t^{n-1} e^{-zt} \mathcal{L}_t^{(\ell)} dt,
\end{equation}
that can be easily verified by computing $X^{(\ell)}R^{(\ell)}(z)^n$ and $R^{(\ell)}(z)^nX^{(\ell)}$.
The next Lemma, which proof can be found at the end of Section \ref {subsec:LY}, gives an effective Lasota-Yorke inequality for the resolvent. 
\begin{lem} \label{lem:quasicompactness} 
Let $p \in \bN, q \in \bRp$, $p + q < r-1 $, $ z \in \bC $ such that $a = \Re(z) > \sigma_\ell $ then we have
\[
\spqnorm[p,q,\ell]{R^{(\ell)}(z)^n} \leq C_{p,q,\lambda} (a-\sigma_\ell)^{-n},
\]
and, for $p,q>0$,
\[
 \spqnorm{R^{(\ell)}(z)^n h} \leq C_{p,q,a,\lambda}\left\{ (a-\sigma_\ell+\sigma_{p,q})^{-n} \spqnorm{h} 
+  \frac{(|z|+1)}{(a-\sigma_\ell)^{n}} \spqnorm[p-1,q+1,\ell]{h}\right\}.
\]
The operator $R^{(\ell)}(z)$ has essential spectral radius bounded by $(a-\sigma_\ell+\sigma_{p,q})^{-1}$.
\end{lem}
We are now in the position to obtain the required spectral properties.
\begin{prop} \label{lem:Nussbaum}
For $p+q <r-1$ the spectrum of the generator $X^{(\ell)}$ of the semigroup $\mathcal{L}_t^{(\ell)}$ acting on $\widetilde{\cB}^{p,q,\ell}$ lies on the left of the line $\{\sigma_\ell +ib\}_{b\in\bR}$ and  in the strip 
$\sigma_\ell \geq \Re(z) > \sigma_\ell-\sigma_{p,q} $ consists of isolated eigenvalues of finite multiplicity. Moreover, for $\ell=d_s$, $\htop $ is an eigenvalue of $X$. If the flow is topologically transitive $\htop$ is a simple eigenvalue. If the flow is topologically mixing, then $\htop$ is the only eigenvalue on the line $\{\htop+ib\}_{b\in\bR}$.
\end{prop}
\begin{proof}
The first part of the Lemma is proven exactly as in \cite[Proposition 2.10, Corollary 2.11]{Liverani04}.
Next, let us analyze the case $\ell=d_s$. As explained in Remark \ref{rem:orientability-s}, we restrict ourselves to the case in which $E^s$ is orientable. Indeed, in the non orientable case we must study a slightly different operator (see Section \ref{app:orientable} for more details). We can then choose in each $E^s(x)$ a volume form $\tilde \omega_s$ on $E^s$ normalized so that $\|\tilde\omega_s\|=1$ and it is  globally continuous. Also let $\pi_s(x):T_xM\to E^s(x)$ be the projections on $E^s(x)$ along $E^u(x)\oplus E^c(x)$. Remember that $\pi_s$ is  $\varpi$-H\"older (see Appendix \ref{app:holo-est}). Next, define $\omega_s(v_1,\dots,v_{d_s})\doteq \tilde \omega_s(\pi_s v_1,\dots,\pi_s v_{d_s})$, by construction $\omega_s\in\Omega^{d_s}_{0,\varpi}$. Note that  $\phi_{-t}^* \omega_s=J_s\phi_{-t}\omega_s$, where $J_s\phi_{-t}$ is the Jacobian restricted to the stable manifold.
Note that, setting $\omega_{s,\ve}=\bM_\ve\omega_s$, for $\ve$ small enough, we have $\langle \omega_{s,\ve}, \omega_{s}\rangle\geq \frac12$. Hence,
\begin{equation}\label{eq:lower-op-bound0}
\int_{W_{\alpha, G}}\langle \omega_{s,\ve}, \cL_t \omega_s \rangle \geq \int_{W_{\alpha, G}} \hskip-.3cm\frac{ J_s\phi_{-t}}2\geq C_\#\int_{W_{\alpha, G}} \hskip-.3cm J_W\phi_{-t}\geq C_\# \vol(\phi_{-t}W_{\alpha, G}).
\end{equation}
Taking the sup on the manifolds and integrating in time, Appendix \ref{app:topent} (see in particular Remark \ref{rem:h0}) implies that the spectral radius of $R^{(d_s)}(a)$ on $\widetilde\cB^{q,\varpi, d_s}$ is exactly $(a-\sigma_{d_s})^{-1}$. 

By Lemma \ref{lem:quasicompactness} $R^{(d_s)}(a)$ is quasicompact on $\widetilde\cB^{p,q,d_s}$ and its peripheral spectrum does not contain Jordan blocks.\footnote{ A Jordan block would imply that $\pqnorm[p,q,d_s]{R^{(d_s)}(a)^n}$ grows at least as $C_\# n (a-\sigma_{d_s})^{-n}$ contrary to the first inequality in Lemma \ref{lem:quasicompactness}.} In turn, this implies that
\[
\lim_{n\to\infty}\frac 1n\sum_{k=0}^{n-1}(a-\sigma_{d_s})^k R^{(d_s)}(a)^k=\begin{cases}\Pi &\text{ if } (a-\sigma_{d_s})^{-1}\in\sigma_{\widetilde\cB^{p, q, d_s}}(R^{(d_s)}(a))\\
0 &\text{ otherwise,}
\end{cases}
\]
where $\Pi$ is the eigenprojector on the associated eigenspace and the convergence takes place in the strong operator topology of $L(\widetilde\cB^{p,q, d_s},\widetilde\cB^{p,q, d_s})$. Also by the second inequality in the statement of Lemma \ref{lem:quasicompactness} it follows that $\Pi$ extends naturally to an operator in $L(\widetilde\cB^{p-1,q+1, d_s}, \widetilde\cB^{p,q, d_s})$.
Since, by \eqref{eq:lower-op-bound0},
\[
\int_{W_{\alpha, G}}\langle\omega_{s,\ve},\Pi\omega_s\rangle >0
\]
we have that $\Pi\neq 0$ and $(a-\sigma_{d_s})^{-1}$ belongs  to the spectrum.
The other claimed properties follow by arguing as in \cite[Section 6.2]{GouezelLiverani08}.
\end{proof}

\subsection[Proof of the Lasota-Yorke]{Proof of the Lasota-Yorke inequality}\label{subsec:LY}
\begin{proof}[{\bf Proof of Lemma \ref{lem:LY}}]
For each $\mathfrak{a}\in (0,1)$, we introduce the norms
\[
\|h\|_{\rho,L, p,q, \ell,\mathfrak{a}} \doteq \sum_{\wp=0}^p \mathfrak{a}^\wp \|h\|^-_{\rho,L,\wp,q,\ell},
\]
From now on we will suppress the indexes $\rho,L$ as their values will be clear from the context.
As the above norms are equivalent to $\|\cdot\|_{p,q,\ell}$, it suffices to prove the required inequality for some fixed $\mathfrak{a}$ (to be chosen later).

For any $\wp,q \in \bN$, $\alpha \in \cA$, $G \in
\widetilde{\Sigma}(\rho_+,L_+)$, $v_i \in \cV^{\wp+q}(\alpha,G)$,  $g \in \Gamma^{\ell,\wp+q}_c(\alpha,G)$ 
such that $\|g\|_{\Gamma^{\ell,\wp+q}_c(\alpha,G)}\leq 1 $, we must estimate
\begin{equation} \label{eq:lybegin}
\int_{W_{\alpha,G}} \langle g, L_{v_1} \cdots L_{v_\wp} \cL_t^{(\ell)}h 
\rangle\;  \volform.
\end{equation}
First of all, we consider the atlas introduced at the beginning of Section \ref{banachspace}. 
For $\beta \in \cA$ and for each $t \in \bRp$, let $\{ V_{k,\beta} \}_{k \in K_{\beta}}$, $K_\beta\subset \bN$, be
the collection of connected components of $\Theta_\beta(\phi_{-t}(W_{\alpha,G}) \cap U_\beta)$.
Let $ \widetilde{K}_\beta \doteq \{ k \in K_\beta :  V_{k,\beta} \cap 
B(0,2\delta) \neq \emptyset \} $. For all $V_{k,\beta}$, let 
$\widetilde{V}_{k,\beta}$  be the connected component
of $\Theta_\beta(\phi_{-t} (W^+_{\alpha,G}) \cap U_\beta)$ which
contains $V_{k,\beta}$. Next, for each $k \in \widetilde{K}_{\beta}$ 
choose $x_{k,\beta} \in ( V_{k,\beta} \cap B(0,2\delta) )$. 
We choose $\vuo$ as in Remark \ref{rmk:small-times}. Then, for all $t\geq \vuo$, there exists $F_{k,\beta} \in
\cF_r(1,L_0)$ (see definition \eqref{def:leaves}) such that
$G_{x_{k,\beta},F_{k,\beta}} (\xi)= x_{k,\beta} + (\xi, F_{k,\beta}(\xi))$, $\xi \in
B(0,6\delta)$, and the graph of $G_{x_{k,\beta},F_{k,\beta}}$
is contained in $\widetilde{V}_{k,\beta}$. Next we consider the manifolds
\begin{equation}\label{eq:def-family}
\{ W_{\beta,G_{x_{k,\beta},F_{k,\beta}}} \}_{k \in \tilde{K}_\beta} \subset U_\beta.
\end{equation}
From now on let us write $W_{\beta,G_k} $ for $W_{\beta,G_{x_{k,\beta},F_{k,\beta}}}$. By the above construction $W_{\beta,G_k}\in\Sigma_-$.
Note that the above construction provides an explicit analogue of Lemma $7.2$ in \cite{ButterleyLiverani07}.
In particular, $\cup_{\beta\in\cA}\cup_{k\in\tilde K_\beta} W_{\beta,G_k}\supset \phi_{-t}W_{\alpha,G}$ with a uniformly bounded number of overlaps.
  
In the following it will be convenient to use the {\em Hodge operator} ``$*$'', see Appendix \ref{appendiceLie} for details. Let us start with the case $\wp = 0$.  By using formula \eqref{hodgedualform}, changing variables we obtain\footnote{ Since $\langle g,h\rangle$ can be seen as a function on $M$, we will often use the notation $F^*\langle g,h\rangle$ for  $\langle g,h\rangle\circ F$. Also, will use interchangeably  the notations $\det(D\phi_t)$ and $J\phi_t$, since, for all $t\in\bR$, $\det(D\phi_t)>0$.}
\begin{equation}\label{eq:mani-split} 
\int_{W_{\alpha,G}} \hskip-.3cm \langle g,  \cL_t^{(\ell)}h 
\rangle  \volform  =  \sum_{\stackrel{\scriptstyle  \beta \in  \cA}{k \in \widetilde{K}_\beta}}
\int_{W_{\beta,G_k}}  \hskip-.4cm(-1)^{(d-\ell)\ell}\psi_\beta \frac{J_W\phi_t }{ J\phi_{t}}
\langle *\phi_{t}^* (*g),h \rangle 
\volform\,,
\end{equation}
where $J_W\phi_t$ is the Jacobian of the change of variable $\phi_t \,:\, W_{\beta, G_k}\mapsto W_{\alpha,G}$.
Note that 
\[
\begin{split}
*\phi_{t}^* (*g)&=\sum_{\bar i,\bar j\in\cI_\ell}\langle \omega_{\beta,\bar i}, *\phi_{t}^{*} \! * \! \omega_{\alpha, \bar j}\rangle g_{\bar j}\circ \phi_t \cdot \omega_{\beta,\bar i}
\\
&=(-1)^{(d-\ell)\ell} J\phi_{t}\sum_{\bar i,\bar j\in\cI_\ell}\langle  \phi_{-t}^*\omega_{\beta,\bar i},\omega_{\alpha, \bar j}\rangle\circ \phi_t\cdot  g_{\bar j}\circ \phi_t\cdot  \omega_{\beta,\bar i}\,,
\end{split}
\]
where we have used \eqref{hodgedualform} again. Letting 
\[
\tilde g_t=(-1)^{(d-\ell)\ell} J\phi_{t}\sum_{\bar i\in\cI^-_\ell}\sum_{\bar j\in\cI_\ell}\langle  \phi_{-t}^*\omega_{\beta,\bar i},\omega_{\alpha, \bar j}\rangle\circ \phi_t\cdot  g_{\bar j}\circ \phi_t\cdot  \omega_{\beta,\bar i},
\]
by Remark \ref{rem:form-details} if follows that $\langle *\phi_{t}^* (*g), h\rangle=\langle \tilde g_t,h\rangle$.
Accordingly, we have
\begin{equation}\label{eq:est-base}
\begin{split}
&\|[ J\phi_{t}^{-1} J_W\phi_t \cdot\psi_\beta\cdot \tilde g_t]\circ \Theta_\beta^{-1}\circ G_k\|_{\Cs^q}
\leq C_\#  \sup_{\bar{j}, \bar{i}}   \| g_{\bar j}\circ \phi_t \circ\Theta_\beta^{-1}\circ G_k\|_{\Cs^q} \\
& \qquad \times \|[J_W\phi_t\cdot\langle  \phi_{-t}^*\omega_{\beta,\bar i},\omega_{\alpha, \bar j}\rangle\circ\phi_t ]\circ\Theta_\beta^{-1}\circ G_k\|_{\Cs^q}.
\end{split}
\end{equation}
Next, we estimate the factors of the above product. Note that there exist maps $\Xi_t\in \Cs^r(\bR^d,\bR^d)$ such that 
\begin{equation}\label{eq:Xit}
\Theta_\alpha\circ\phi_t \circ\Theta_\beta^{-1}\circ G_k=G\circ \Xi_t.
\end{equation}
Moreover, by the contraction properties of the Anosov flow, $\|D\Xi_t\|_{\Cs^{r-1}}\leq C_\# e^{-\overline{\lambda} t}$.\footnote{ See the Appendix in \cite{ButterleyLiverani07} for details.} Thus, by the properties of $r$-norms (see Remark \ref{foo:r-norm}),
\begin{equation}\label{eq:gphit-norm}
\| g_{\bar j}\circ \phi_t \circ\Theta_\beta^{-1}\circ G_k\|_{\Cs^q}\leq C_\# \|g\circ \Theta_\alpha^{-1}\circ G\|_{\Cs^q}.
\end{equation}
Next, note that $|\langle  \phi_{-t}^*\omega_{\beta, \bar i},\omega_{\alpha, \bar j}\rangle|$ is bounded by the growth of $\ell$-volumes while $J_W\phi_t$ gives the contraction of $d_s$ volumes in the stable direction. Clearly,  the latter is simply the inverse of the maximal growth of  $|\langle  \phi_{-t}^*\omega_{\beta, \bar i},\omega_{\alpha, \bar j}\rangle|$ which takes place for $d_s$ forms. By the Anosov property and Remark \ref{rem:form-details} follow that
\begin{equation}\label{eq:zero-test}
\|J_W\phi_t \circ\Theta_\beta^{-1}\circ G_k\cdot\langle  \phi_{-t}^*\omega_{\beta, \bar i},\omega_{\alpha, \bar j}\rangle\circ\phi_t \circ\Theta_\beta^{-1}\circ G_k\|_{\Cs^0}\leq C_\# e^{-|d_s-\ell|\olambda t}.
\end{equation}
To compute the $\Cs^q$ norm, for $q \geq 1$, we begin by computing the Lie derivative with respect to the vector fields $Z_i=\partial_{\xi_i}$, $i\in\{1,\dots, d_s\}$. Let us set $\Upsilon_t\doteq  \phi_t \circ\Theta_\beta^{-1}\circ G_k$. By \eqref{eq:lie-scalar}  and
\eqref{eq:AM2219} we have
\begin{equation} \label{eq:inductionreduction}
\begin{split}
L_{Z_i}&[\langle \phi_{-t}^*v,w\rangle \circ \Upsilon_t]=\Upsilon_t^*L_{\Upsilon_{t*}Z_i}\langle \phi_{-t}^*v,w\rangle\\
&=-\Upsilon_t^*[\operatorname{div}(\Upsilon_{t*}Z_i)\langle (\phi_{-t}^*v,w\rangle]+(-1)^{\ell(d-\ell)}\Upsilon_t^*\langle *L_{\Upsilon_{t*}Z_i} \!* \! w,\phi_{-t}^*v\rangle\\
&\quad+\Upsilon_t^*\langle w,L_{(\Theta_\beta^{-1}\circ G_k)_*Z_i}v\rangle.
\end{split}
\end{equation}
Hence, recalling again \eqref{def:crnorm} and its properties, for $q \in \bN$, 
\[
\begin{split}
&\| \langle \phi_{-t}^*v,w\rangle \circ \Upsilon_t \|_{\Cs^q}\leq C_{\#} \sup_{i} \| L_{Z_i}[\langle \phi_{-t}^*v,w\rangle
\circ \Upsilon_t] \|_{\Cs^{q-1}}\\
&\hskip5cm+2^{\lfloor q  \rfloor + 1 }\| \langle \phi_{-t}^*v,w\rangle \circ \Upsilon_t\|_{\Cs^0}\\
&\leq C_q\big\{\|\operatorname{div}(\Upsilon_{t*}Z_i)\langle \phi_{-t}^*v,w\rangle\|_{\Cs^{q-1}}+\|\langle L_{(\Theta_\beta^{-1}\circ G_k)_*Z_i}v,w\rangle\|_{\Cs^{q-1}}\\
&\quad+ \|\langle\phi_{-t}^*v,*L_{\Upsilon_{t*}Z_i}*w\rangle\|_{\Cs^{q-1}}+\| \langle \phi_{-t}^*v,w\rangle \|_{\Cs^0}\big\},
\end{split} 
\]
which can be used inductively to show that the $\Cs^q$ norm is bounded by the $\Cs^0$ norm. The case $q\in (0,1)$ can be computed directly as well, but is also follow by interpolation theory (e.g., see \cite[(2),(4) page 201]{Triebel}). The computation for the case $q \in \bRp$ easily follows.
Finally, by \eqref{eq:zero-test}, \eqref{eq:gphit-norm}, and \eqref{eq:est-base} we have that
\begin{equation} \label{eq:q-test}
\|\psi_\beta J\phi_t^{-1}J_W\phi_{t}\;*\phi_t*g\|_{\Gamma^{\ell,q}_c}\leq C_\#e^{-\olambda|d_s-\ell|t}.
\end{equation}
This provides an estimate for each term in \eqref{eq:mani-split}. We are left with the task of estimating the number of manifolds. Note that such number at time $t$ is bounded by a constant times the volume of the manifold $\phi_{-t}(W_{\alpha,  G})$. Moreover, it turns out that such volumes grow proportionally to $e^{h_{\text{\tiny  top}}(\phi_1)t}$ (see Appendix \ref{app:topent}). 
Thus
\begin{equation}\label{eq:est-two}
\pqnorm[0,q,\ell]{  \cL_t^{(\ell)}h }^-\leq C_qe^{-\olambda |d_s-\ell|t}e^{h_{\text{\tiny top}}(\phi_1)t} \pqnorm[0,q,\ell]{h}^-.
\end{equation} 
To establish the second inequality note that, by analogy with \cite[Lemma 6.6]{GouezelLiverani06}, for each $\ve > 0 $ there exists $g_\ve\in \Gamma_0^{\ell,r}(W_{\alpha, G})$ such that, setting $q'=\max\{0,q-1\}$,
\[ 
\begin{split}
& \left\| g_\ve - g  \right\|_{_{\Gamma^{\ell,q'}_c}} \leq  C_q   \ve^{q-q'}  \; ; 
 \left\| g_\ve \right\|_{\Gamma^{\ell,q}_c}  \leq C_q
\; ; \;  \left\|g_{\ve} \right\|_{\Gamma^{\ell,q+1}_c}  \leq  C_q \ve^{-1}. 
 \end{split} 
 \]
Thus, writing $g=(g-g_\ve)+g_\ve$ we can estimate to the two pieces separately. For $g_\ve$ we apply all the above computations using $q+1$ instead of $q$. To estimate the equivalent of \eqref{eq:gphit-norm} for the first piece, set $\psi_{t,k}=\phi_t\circ\Theta_\beta^{-1}\circ G_k$. For $ q \not \in \bN $, we have\footnote{ By $H^\eta(f)$ we mean the H\"older constant of $f$, i.e. $\sup_{\|x-y\|\leq \delta}\frac{\|f(x)-f(y)\|}{\|x-y\|^\eta}$.}
\[
\begin{split}
\|(g-g_\ve)_{\bar j}\circ\psi_{t,k} \|_{\Cs^q}&\leq C_q\ve^{q-q'}+\sup_{i_1,\dots, i_{\lfloor q\rfloor}}\hskip-.2cm H^{q-\lfloor q\rfloor}\left(L_{Z_{i_1}}\dots L_{Z_{i_{\lfloor q\rfloor}}}[ (g-g_\ve)_{\bar j}\circ\psi_{t,k}]\right)\\
&\leq C_q(\ve^{\min\{1,q\}}+e^{-\olambda q t}),
\end{split}
\]
the case $q \in \bN$ being easier. Choosing $\ve$ depending on $t$ yields the inequality
\begin{equation}\label{eq:zero-bis-norm}
\pqnorm[0,q,\ell]{\cL_t^{(\ell)}h}^- \leq  C_q e^{(h_{\text{\tiny top}}(\phi_1)-\olambda |d_s-\ell|-\olambda q)t} \|h\|_{0,q,\ell}^- +  
C_{q,t} \pqnorm[0,q+1,\ell]{h}^- .
\end{equation}
The second inequality of the Lemma follows then by iterating \eqref{eq:zero-bis-norm} for time steps $t_*>\vuo$ such that $C_qe^{(\sigma_\ell-\olambda q)t_*}\leq e^{(\sigma_\ell-\lambda q) t_*}$ and using \eqref{eq:est-two}.

We are left with the case $\wp > 0$. We can apply Lemma $7.4$ in \cite{ButterleyLiverani07} to decompose each  $v \in \cV^{\wp+q}(\phi_{t}(W_{\beta,G_{x_k, F_k}}))$ as $v = v^s + v^u + v^V$ where $v^s$ is tangent to
$\phi_{t}(W_{\beta,G_{x_k, F_k}})$, $v^u$ is ``close'' to the unstable
direction and  $v^V$ is the component in the flow direction.
Let  $\sigma \in \{s,u,V\}^\wp$, then \eqref{eq:lybegin} reads
\begin{equation} \label{eq:stable-begin}
\sum_{\sigma \in \{s,u,V\}^\wp}
\int_{W_{\alpha,G}} \langle g,
 L_{v_1^{\sigma_1}} \cdots L_{v_\wp^{\sigma_\wp}} \cL_t^{(\ell)}h \rangle
\volform .
\end{equation}

Since $ L_v L_w h =  L_w L_v h + L_{[v,w]}h  $, we can reorder the derivatives so as to have first those with respect to the stable components, then those with respect to the unstable, and then finally those  in the flow direction. By permuting the Lie derivative in such a way we introduce extra terms  with less than $\wp$ derivatives, and  these extra terms contribute only to weaker norms. 
For each $\sigma \in \{s,u,V\}^\wp$ let $p_s(\sigma) = \# \{i \, | \, \sigma_i = s \} $, $p_u(\sigma) = \# \{i \, | \, \sigma_i = u \} $ and $p_V(\sigma) = \# \{i \, | \, \sigma_i = V \} $. For each $\bar p=(p_s,p_u,p_V)\in\bN^3$, $p_s+p_u+p_V=\wp$, let $\Sigma_{\bar p}=\{\sigma\;:\; p_s=p_s(\sigma),\, p_u=p_u(\sigma), p_V= p_V(\sigma)\}$. Next, for each $\sigma\in\Sigma_{\bar p}$ we define the permutation $\pi_\sigma$  by $\pi_\sigma(i)<\pi_\sigma(i+1)$ for $i\not\in\{p_s, p_s+p_u\}$;  $\sigma_{\pi(i)}=s$ for $i\in\{1, \ldots , p_s\}$; $\sigma_{\pi(i)}=u$ for $i\in\{p_s+1, \ldots , p_s+p_u\}$ and $\sigma_{\pi(i)}=V$ for $i\in\{\wp - p_V + 1, \ldots , \wp \}$.
We can then bound the absolute value of \eqref{eq:stable-begin} by
 \begin{equation}\label{eq:stable-two}
 \begin{split} 
&\sum_{p_s + p_u + p_V = \wp } \sum_{\sigma\in\Sigma_{\bar p}} 
\bigg|\int_{W_{\alpha,G}} 
\langle g , L_{v_{\pi_\sigma(1)}^s} \cdots L_{v_{\pi_\sigma(p_s)}^s} 
L_{v_{\pi_\sigma(p_s + 1)}^u} \cdots L_{v_{\pi_\sigma(p_s + p_u)}^u} \\ 
&  \times L_{v_{\pi_\sigma(p_s + p_u + 1)}^V} \cdots L_{v_{\pi_\sigma(\wp)}^V}  
\cL_t^{(\ell)}h \rangle \volform\bigg| + \mathfrak{a}^{-\wp+1}C_{\wp,q} \pqnorm[\wp-1,q , \ell, \mathfrak{a}]{ \cL_t^{(\ell)} h } .
\end{split} 
\end{equation}
To simplify the notation we define\footnote{ We should write $\tilde{v}_{\sigma,i}$. However, since the following arguments are done at fixed $\sigma$, we drop the subscripts to ease the notation.} 
\[\tilde{v}_i = \left\{ 
\begin{array}{lcl} 
v_{\pi_\sigma(i)}^s &\textrm{for} & i \leq p_s \\
v_{\pi_\sigma(i)}^u &\textrm{for} & p_s + 1 \leq i \leq p_s + p_u  \\
v_{\pi_\sigma(i)}^V &\textrm{for} & i \geq p_s + p_u + 1 .
\end{array} \right.
\]
We will analyze the terms of \eqref{eq:stable-two} one by one. First of all, suppose $p_V\neq 0$, then $v_i=\gamma_i V$ for $i>p_s+p_u$. Using equation \eqref{LieFlow}, 
\begin{equation}\label{eq:flow-est}
\begin{split} 
&\left| \int_{W_{\alpha,G}} \langle g, L_{\tilde{v}_1} \cdots L_{\tilde{v}_\wp}  
\cL_t^{(\ell)}h \rangle \volform \right| \\
& =\left| \int_{W_{\alpha,G}} 
\langle  g, L_{\tilde{v}_1} \cdots L_{\tilde{v}_{\wp-1}} (\gamma_\wp\cL_t^{(\ell)}X^{(\ell)} h) \rangle  \volform\right|  \\
& \leq C_{\wp,q}\pqnorm[\wp-1,q +1, \ell]{\cL_t^{(\ell)} X^{(\ell)} h }^{-}+\mathfrak{a}^{-\wp+1}C_{\wp,q}\pqnorm[\wp-1,q, \ell, \mathfrak{a}]{\cL_t^{(\ell)} h }.
\end{split} 
\end{equation}
Next, suppose $p_V = 0 $ but $p_s \neq 0$. By \eqref{eq:lie-scalar},
\[ 
\begin{split} 
& \left|\int_{W_{\alpha,G}} 
\langle g, L_{\tilde{v}_1} \cdots L_{\tilde{v}_\wp}  
\cL_t^{(\ell)}h \rangle  \volform\right| \\ 
&\hskip-.2cm\leq   \left|\int_{W_{\alpha,G}} \hskip-.5cm L_{\tilde{v}_1}  \langle g,  L_{\tilde{v}_{2}} \cdots L_{\tilde{v}_{\wp}} 
\cL_t^{(\ell)}h \rangle  \volform \right| +
 C_{\wp,q}\left\| \cL^{(\ell)}_t h  \right\|_{\wp-1,q,\ell}^{-}. 
\end{split} 
\]
Then, by intrinsic differential calculus in the manifold $W_{\alpha,G}$, 
\begin{equation}\label{eq:stokes}
d(f i_v\volform)=L_v(f\volform)=(-1)^{d_s+1}(L_v f)\volform+f d(i_v\volform).
\end{equation}
Thus, by Stokes theorem, the integral is bounded by $C_{\wp,q}\left\| \cL^{(\ell)}_t h  \right\|_{\wp-1,q,\ell}^-$ and
\begin{equation}\label{eq:stable-final}
\begin{split} 
\left|\int_{W_{\alpha,G}} 
\langle g, L_{\tilde{v}_1} \cdots L_{\tilde{v}_\wp}  
\cL_t^{(\ell)}h \rangle  \volform\right|  
\leq &C_{\wp,q}\pqnorm[\wp-1,q,\ell]{ \cL^{(\ell)}_th}^{-}. 
\end{split} 
\end{equation}
Finally, suppose $p_u=\wp$. 
Recalling \eqref{eq:AM2219} yields
\[ 
\int_{W_{\alpha,G}} \hskip-.5cm \langle g,  L_{\tilde{v}_{1}} \cdots L_{\tilde{v}_{\wp}} 
\cL_t^{(\ell)}h \rangle  \volform=  \int_{W_{\alpha,G}}\hskip-.5cm  \langle g,  \cL_t^{(\ell)}L_{\phi_{-t*}\tilde{v}_{1}} \cdots L_{\phi_{-t*}\tilde{v}_{\wp}} 
h \rangle  \volform.
\]
By construction we have $\|\phi_{-t*}\tilde{v}_{i}\|_{\Cs^{\wp+q}}\leq C_\# e^{-\olambda t}$.  
Hence, using \eqref{eq:zero-bis-norm} and \eqref{eq:0q-norm} yields, for $ \wp < p$, \footnote{ By an abuse of notation we use $\tilde v_i$ also for the extensions of the vector fields.}
\begin{equation}\label{eq:unst-est}
\begin{split}
\left|\int_{W_{\alpha,G}} \hskip-.5cm \langle g,  L_{\tilde{v}_{1}} \cdots L_{\tilde{v}_{\wp}} 
\cL_t^{(\ell)}h \rangle  \volform  \right|
&\leq C_{\wp,q,t}\|L_{\phi_{-t*}\tilde{v}_{1}} \cdots L_{\phi_{-t*}\tilde{v}_{\wp}} h\|_{0,q+\wp+1,\ell}^-\\ 
&\hskip-2.5cm+C_{\wp,q}e^{(h_{\textrm{top}}(\phi_1)-\olambda|d_s-\ell|-q\olambda)t}\|L_{\phi_{-t*}\tilde{v}_{1}} \cdots L_{\phi_{-t*}\tilde{v}_{\wp}} h\|_{0,q+\wp,\ell}^-\\
&\hskip-3cm\leq C_{\wp,q}e^{(h_{\textrm{top}}(\phi_1)-\olambda|d_s-\ell|-(q+\wp)\olambda)t}\| h\|_{\wp,q,\ell}^-+C_{\wp,q,t}\| h\|_{\wp,q+1,\ell}^-.   
\end{split} 
\end{equation}
Note that the last term can be estimated by the $\|\cdot\|_{p-1,q+1,\ell,\mathfrak{a}}$ norm, except in the case $\wp=p$. In such a case we use instead \eqref{eq:est-two} which yields, for each $\wp\leq p$,
\begin{equation}\label{eq:unst-est2}
\begin{split}
\left|\int_{W_{\alpha,G}} \hskip-.5cm \langle g,  L_{\tilde{v}_{1}} \cdots L_{\tilde{v}_{\wp}} 
\cL_t^{(\ell)}h \rangle  \volform \right|
&\leq C_{\wp,q}e^{(h_{\textrm{top}}(\phi_1)-\olambda|d_s-\ell|-\wp\olambda)t}\| h\|_{\wp,q,\ell}^-.   
\end{split}
\end{equation} 
Collecting together \eqref{eq:zero-bis-norm}, \eqref{eq:flow-est}, \eqref{eq:stable-final}, \eqref{eq:unst-est} and \eqref{eq:unst-est2} we have,  for $t > \vuo$,
\[
\begin{split}
&\| \cL_t^{(\ell)} h \|_{p,q,\ell,\mathfrak{a}} \leq \,C_{q} e^{(h_{\textrm{top}}(\phi_1)-\olambda|d_s-\ell|-q\olambda)t}\|h\|_{0,q,\ell}^- +C_{q,t} \|h\|_{0,q+1,\ell}^-\\
&+C_{p,q}\mathfrak{a}\pqnorm[p-1,q, \ell,\mathfrak{a}]{\cL_t^{(\ell)}  h }+ \sum_{\wp=1}^p \mathfrak{a}^\wp C_{\wp,q}\pqnorm[\wp-1,q +1, \ell]{\cL_t^{(\ell)} X^{(\ell)} h }^{-}\\
&+\sum_{\wp=1}^{p-1} \mathfrak{a}^\wp C_{\wp,q}e^{(h_{\textrm{top}}(\phi_1)-\olambda|d_s-\ell|-\olambda (q+\wp))t}\| h\|_{\wp,q,\ell}^- +C_{p,q,\mathfrak{a},t}\|\cL_t^{(\ell)}h\|_{p-1,q+1,\ell,\mathfrak{a}}\\
&+\mathfrak{a}^p C_{p,q}e^{(h_{\textrm{top}}(\phi_1)-\olambda|d_s-\ell|-\olambda p)t}\| h\|_{p,q,\ell}^- .\\
\end{split}
\]
Then
\begin{equation}\label{eq:ultimapalla}
\begin{split}
\| \cL_t^{(\ell)} h \|_{p,q,\ell,\mathfrak{a}} \leq& \,C_{p,q} e^{(h_{\textrm{top}}(\phi_1)-\olambda|d_s-\ell|-\min\{p,q\}\olambda)t}\|h\|_{p,q,\ell,\mathfrak{a}} \\
&+C_{p,q,\mathfrak{a},t} \|h\|_{p-1,q+1,\ell,\mathfrak{a}}+C_{p,q,\mathfrak{a},t}\|\cL_t^{(\ell)}h\|_{p-1,q+1,\ell,\mathfrak{a}}\\
&+  C_{p,q} \mathfrak{a} \pqnorm[p-1,q +1, \ell,\mathfrak{a}]{\cL_t^{(\ell)} X^{(\ell)} h }\hskip-.5cm+C_{p,q} \mathfrak{a}\left\| \cL^{(\ell)}_t h  \right\|_{p,q,\ell,\mathfrak{a}}.
\end{split}
\end{equation}
While, if we do not use \eqref{eq:zero-bis-norm} and  \eqref{eq:unst-est}, we have, for $ \mathfrak{a} \in (0,1)$,
\[
\begin{split}
\| \cL_t h \|_{p,q,\ell,\mathfrak{a}} 
\leq &\,C_{p,q} e^{(h_{\textrm{top}}(\phi_1)-\olambda|d_s-\ell|)t}\|h\|_{p,q,\ell,\mathfrak{a}} \\
& +C_{p,q}\mathfrak{a}\pqnorm[p-1,q +1, \ell,\mathfrak{a}]{\cL_t^{(\ell)} X^{(\ell)} h }+C_{p,q,} \mathfrak{a}\left\| \cL^{(\ell)}_t h  \right\|_{p-1,q,\ell,\mathfrak{a}}.
\end{split}
\]
By induction, for which \eqref{eq:est-two} is the initial step, it follows that, for all $t\in\bRp$,
\begin{equation}\label{eq:almost-one}
\| \cL_t h \|_{p,q,\ell,\mathfrak{a}} 
\leq \,C_{p,q} e^{(h_{\textrm{top}}(\phi_1)-\olambda|d_s-\ell|)t}\|h\|_{p,q,\ell,\mathfrak{a}}.
\end{equation}
Next, we choose $t_*$ such that, for a given $\lambda < \olambda$, we have that\footnote{ In this case $C_{p,q}$ refers exactly to the first constant in \eqref{eq:ultimapalla}.}
\[
C_{p,q} e^{(h_{\textrm{top}}(\phi_1)-\olambda|d_s-\ell|)t_*}\leq e^{(h_{\textrm{top}}(\phi_1)-\lambda|d_s-\ell|)t_*}.
\]
Finally, choosing $\mathfrak{a}=e^{-\min\{p,q\}\olambda t_*}$, from \eqref{eq:ultimapalla} and \eqref{eq:almost-one} we have
\begin{equation}\label{eq:almost-two}
\begin{split}
\| \cL_{t_*} h \|_{p,q,\ell,\mathfrak{a}} \leq&\,  e^{(h_{\textrm{top}}(\phi_1)-\lambda|d_s-\ell|-\min\{p,q\}\lambda)t_*}\|h\|_{p,q,\ell,\mathfrak{a}} \\
&+C_{p,q,\mathfrak{a},t_*} \|h\|_{p-1,q+1,\ell,\mathfrak{a}}+C_{p,q}\mathfrak{a}\pqnorm[p-1,q +1, \ell,\mathfrak{a}]{\cL_{t_*}^{(\ell)} X^{(\ell)} h }.
\end{split}
\end{equation}
Iterating \eqref{eq:almost-two}, recalling  \eqref{eq:almost-one}, and using \eqref{eq:ultimapalla} for the remaining $t-\lfloor t/t_*\rfloor t_*$ time interval, the Lemma follows.
\end{proof}
\begin{proof}[{\bf Proof of Lemma \ref{lem:quasicompactness}}]
By Lemma \ref{lem:LYstrong} and equation \eqref{eq:resoldef} we have \footnote{ Note that for each $a > 0 ,\alpha \geq 0 $,  we have
$\int_{\alpha}^{\infty}  x^{n} e^{-ax} dx  = e^{-a \alpha} \sum_{k=0}^{n}  \frac{n!}{k!} \frac{\alpha^k}{a^{n-k+1}}$.
} 
\[ 
\spqnorm{R^{(\ell)}(z)^nh} \leq\frac {C_{p,q,\lambda}}{(n-1)!} \spqnorm{h} \int_0^\infty\hskip-.3cm
dt \; e^{-at+\sigma_\ell t}t^{n-1}= \frac{C_{p,q,\lambda}}{(a-\sigma_\ell)^{n}}\spqnorm{h}.
\]
Thus, provided $\Re(z)>\sigma_\ell$ , we have that $R^{(\ell)}(z)^n \in L(\widetilde\cB^{p,q,\ell},\widetilde\cB^{p,q,\ell})$.

Let us introduce a truncated resolvent
\[
R_{n,\ell}(z) \doteq \frac{1}{(n-1)!}\int_{\vuo}^\infty t^{n-1}e^{-zt}\cL_t^{(\ell)}\,.
\]
Then
\begin{equation}\label{eq:res-small-t}
\begin{split}
\hskip-.3cm\spqnorm{R^{(\ell)}(z)^n-R_{n,\ell}(z)}&\leq \int_0^{\vuo} \frac{C_{p,q,\lambda}t^{n-1}e^{-(a-\sigma_\ell)t}}{(n-1)!}\\
&\leq  C_{p,q,\lambda}\frac{\vuo^n}{n!}\leq C_{p,q,\lambda}(a-\sigma_\ell+\sigma_{p,q})^{-n},
\end{split}
\end{equation}
provided $n\geq \vuo e(a-\sigma_\ell+\sigma_{p,q})$.\footnote{ Since $n!\geq n^ne^{-n}$.}
To prove the second inequality of the Lemma it suffices to estimate $R_{n,\ell}(z)$. To this end we follow the same computation used to estimate the norm $\pqnorm{\cdot}$ in Lemma \ref{lem:LY}. The only difference being in the treatment of \eqref{eq:flow-est}. Indeed, since $X\cL_t=\frac{d}{dt}\cL_t$, integration  by parts yields
\[
\begin{split} 
& \left| \int_{W_{\alpha,G}} \int_{\vuo}^\infty\frac{t^{n-1} e^{-zt}}{(n-1)!}
\langle \gamma_\wp g, L_{\tilde{v}_1} \cdots L_{\tilde{v}_{\wp-1}} (\cL_t^{(\ell)}X^{(\ell)} h) \rangle  \volform\right|  \\
& \leq C_{\wp,q}|z|\pqnorm[\wp-1,q+1, \ell]{h}.
\end{split} 
\]
Hence
\begin{equation}\label{eq:ly-res-lim}
\pqnorm{R_n(z) h} \leq \frac{C_{p,q,a,\lambda}\pqnorm{h} }{(a-\sigma_\ell+\sigma_{p,q})^{n}} 
+  \frac{C_{p,q,a,\lambda}(|z|+1)}{(a-\sigma_\ell)^{n}} \pqnorm[p-1,q+1,\ell]{h}.
\end{equation}
The first part of the Lemma then follows  as in Lemma \ref{lem:LYstrong}.

To prove the bound on the essential spectral radius of $R^{(\ell)}(z)^n$ we argue  by analogy with \cite[Proposition 2.10, Corollary 2.11]{Liverani04}.
Nussbaum's formula \cite{Nussbaum70} asserts that if $r_n$ is the infimum of the $r$
 such that $\{R^{(\ell)}(z)^n h\}_{\|h\|\leq 1}$ can be covered by a finite
number of balls of radius $r$, then the essential spectral radius of 
$R^{(\ell)}(z)^n$ is given by $\liminf_{n\to\infty}\sqrt[n]{r_n}$.
Let 
\[
B_1\doteq\{h\in\widetilde\cB^{p,q,\ell}\;|\;\spqnorm{h}\leq 1\}\subset\{h\in \cB^{p,q,\ell}\;|\;\pqnorm{h}\leq 1\}\doteq B_2 . 
\]
By Lemma \ref{lem:compact},
$B_2$ is relatively compact in $\cB^{p-1,q+1,\ell}$. Thus, for each $\epsilon>0$ there are
$h_1,\dots,h_{N_\epsilon}\in B_2$ such that
$B_2\subseteq\bigcup_{i=1}^{N_\epsilon} 
U_\epsilon(h_i)$, where $U_\epsilon(h_i)=\{h\in\cB^{p-1,q+1,\ell}\;|\;\pqnorm[p-1,q+1,\ell]{h-h_i}<\epsilon\}$. 
For $h\in B_1\cap U_\epsilon(h_i)$, computing as in \eqref{eq:ly-res-lim} we have 
\[
\spqnorm{R_n(z)(h-h_i)}\leq\frac{C_{p,q,a,\lambda}}{(a-\sigma_\ell+\sigma_{p,q})^{n}} \pqnorm{h-h_i} 
+  \frac{C_{p,q,a,\lambda} e^{a\vuo}|z|}{(a-\sigma_\ell)^{n}} \epsilon.
\]
Choosing $\epsilon$ appropriately and recalling  \eqref{eq:res-small-t} we conclude that
for each $n\in\bN$  the set $R^{(\ell)}(z)^n(B_1)$ can be covered by a finite number of  
$\protect{\spqnorm{\cdot}-\text{balls}}$ of radius $C_{p,q,a,\lambda}(a-\sigma_\ell+\sigma_{p,q})^{-n}$, which implies the statement.
\end{proof}

\section{Flat Traces}\label{sec:flat}

In this section we define a flat trace and we prove some of its relevant properties. Formally, for $A\in L(\cB^{p,q,\ell},\cB^{p,q,\ell})$ we would like to define a trace by
\begin{equation}\label{eq:distrib-trace}
\int_M \sum_{\alpha,\bar i} \langle\omega_{\alpha,\bar i}, A\delta_{x,\bar i}\rangle_x
\end{equation}
where, for $x\in U_\alpha$,  $\delta_{x,\bar i}(f) = \langle \omega_{\alpha,\bar i},f\rangle_x\omega_{\alpha,\bar i}(x)$, for each $f\in\Omega_{r}^\ell$.
Unfortunately, the $\delta_{x,\bar i}$ do not belong to the  space $\cB^{p,q,\ell}$. We are thus forced to employ an indirect strategy. 
For each $x \in M$, $\bar{i} \in \cI_{\ell}$, $\alpha \in \cA$ and $\ve > 0 $ small enough, let\footnote{ Note that the definition is, given the freedom in the choice of $\kappa$, quite arbitrary.  We are not interested in investigating the equivalence of the various possible definitions.}  
\begin{equation}  \label{eq:normalization1}
j_{\ve, \alpha,\bar i,x}(y) \doteq \left\{
\begin{array}{lr} 
J\Theta_\alpha(x)  \psi_{\alpha}(x)   
\kappa_\ve \left(\Theta_\alpha(x) - \Theta_\alpha(y) \right) \omega_{\alpha,\bar i}(y) &  \textrm{ if } y \in U_{\alpha} \phantom{.}\\
0 & \textrm{ if } y \notin U_{\alpha},
\end{array}  \right.
\end{equation} 
where $\kappa_\ve(x)=\ve^{-d}\kappa(\ve^{-1}x)$.
At this stage we choose a particular $\kappa$ (this refers also to Definition \ref{def:molli}).
Let $ \kappa $ be 
$\kappa(x_1 , \ldots  , x_d ) = \kappa^{d-1}(x_1, \ldots , x_{d-1}) \kappa^1(x_d) $.
Given an operator $A\in L(\cB^{p,q,\ell}, \cB^{p,q,\ell})$ we define the {\em flat trace} by
\begin{equation} 
\begin{split} \label{traceoriginal}
\trf(A) \doteq  & \lim_{\epsilon \to 0}
\int_{M} \sum_{\alpha,\bar{i}}  \langle \omega_{\alpha,\bar i},A  (j_{\epsilon, \alpha,\bar i,x} ) \rangle_x\omega_M(x),
\end{split} 
\end{equation}
provided the limit exists.
To get a feeling for what we are doing, consider the operator $A$ acting on $ h \in \Omega^{\ell}$ defined by 
\begin{equation}\label{eq:kernel}
Ah(x) \doteq \int_M  \, a(x,y) [h(y)] 
\end{equation} 
where $a(x,y)\in L(\wedge^\ell T_y^*M, \wedge^\ell T_x^*M)$ depends continuously on $x,y$. A direct computation shows that
\[
\trf(A)=\int_M  \, \tr(a(x,x)).
\]
Also, one can verify directly that for a finite rank operator $A$ we have $\trf(A)=\tr(A)$, although we will not require this fact.

\begin{lem} \label{lem:traceorbits}
For each $s\in\bRp$, $0 \leq \ell \leq d -1 $, $\Re (z)>\sigma_\ell$ and $n \in \mathbb{N}$, 
we have that $ {\trf}(R^{(\ell)}(z)^n\cL_{s}^{(\ell)})$ is well defined. Moreover,
\begin{equation*}
 \trf(R^{(\ell)}(z)^n\cL^{(\ell)}_{s}) =   \frac{1}{(n-1)!}
\sum_{\tau \in \cT(s)} \frac{\chi_{\ell}(\tau)}{\mu(\tau)}    \left[\lambda(\tau)-s\right]^{n-1}\lambda(\tau)  e^{-z (\lambda(\tau)-s)},
\end{equation*} 
where $\cT(s)\doteq\{\tau\in\cT\;;\; \lambda(\tau) > s\}$, provided $ s \not\in \{ \lambda(\tau) \}_{\tau \in \cT}$ .
\end{lem}
\begin{proof}

By definitions \eqref{eq:resoldef}, \eqref{eq:normalization1} and  \eqref{traceoriginal} we have
\[
\begin{split}
 \trf(R(z)^n\cL_{s}) & =\lim_{\ve\to 0}    \int_{M \times \bRp }  \omega_M(x)  dt  \sum_{\alpha,\bar i} \frac{t^{n-1}e^{-zt}}{(n-1)!}\langle \omega_{\alpha,\bar i},\phi_{-t-s}^*j_{\ve,\alpha,\bar i,x}\rangle_x  \\
& =\lim_{\ve\to 0} \sum_{\alpha,\bar i}   \int_{ U_{\alpha} \times \bRp } \omega_M(x)  dt \frac{t^{n-1}e^{-zt}}{(n-1)!}J\Theta_{\alpha}(x) \psi_\alpha(x) \\
&\quad  \times \kappa_\ve(\Theta_{\alpha}(x) - \Theta_{\alpha}(\phi_{-t-s}(x)))  \langle \omega_{\alpha,\bar i},\phi_{-t-s}^*\omega_{\alpha,\bar i}\rangle_x .
\end{split} 
\]
We find convenient to rewrite the above as
\begin{equation}\label{eq:ftrace-one} 
\begin{split}
& \trf(R(z)^n\cL_{s})   =\lim_{\ve\to 0} \sum_{\alpha,\bar i}   \int_{ U_{\alpha} \times [s,\infty) } \hskip-.5cm F_{\alpha,\bar{i},z}(x,t-s) \frac{ \kappa_\ve(\Theta_{\alpha}(x) - \Theta_{\alpha}(\phi_{-t}(x)))}{J\Theta_{\alpha}(x)^{-1}}   \\
&F_{\alpha,\bar{i},z}(x,t)  \doteq  \frac{t^{n-1}}{(n-1)!}e^{-zt} \psi_\alpha(x) \langle \omega_{\alpha,\bar{i}}, \phi_{-t-s}^*\omega_{\alpha,\bar{i}} \rangle_x .
\end{split}
\end{equation}
Note that the integrand is non zero only for
$ \kappa_\ve(\Theta_{\alpha}(x) - \Theta_{\alpha}(\phi_{-t}(x))) \neq 0 $, that is for  $d(x,\phi_{-t}(x))< \ve$.  

Our next step is to break up the domain of the above integral in a convenient way. To this end for each closed orbit $\tau\in \cT$ and $\epsilon>0$ we define
\[
\begin{split}
\Omega_\tau^\epsilon\doteq\{(x,t)\in M\times\bRp \; | \; &
|t-\lambda(\tau)|\leq \Cnt \ve,\\
&  d(\phi_{-t_1}(x), \tau)\leq \Cnse\epsilon \quad \forall t_1\in(0, t)\},
\end{split} 
\]
where $\Cnse, \Cnt$ are constant to be specified shortly.
\begin{sublem}
There exists $\ve_0>0$ such that, for all $\ve<\ve_0$, the support of the integrand in \eqref{eq:ftrace-one} is contained in $\cup_{\tau\in\cT}\Omega^\ve_\tau$. In addition, if $\Omega^\ve_\tau\cap\Omega^\ve_{\tau'}\neq\emptyset$, then $\tau=\tau'$.
\end{sublem}
\begin{proof}
Let us recall the shadowing theorem for flows (Bowen \cite{Bowen72}), in the formulation explicitly given by 
Pilyugin in \cite[Theorem 1.5.1]{Pilyugin99}, adapted to our case. First of all we define the 
$(\ve,T)$-pseudo-orbits  $\mathfrak{t}(t): \bR \to M $ to be maps such that for any $t' \in \bR$ we have
$ d(\phi_{t}(\mathfrak{t}(t')),\mathfrak{t}(t + t') ) \leq \ve \text{ if }  |t| < T $.\footnote{ Note that we do not
require $\mathfrak{t}(t)$ to be continuous.} Then we have the following theorem.

\begin{thm}[\cite{Pilyugin99}] \label{pilucco}
Let $M$ be a smooth manifold and $\phi_t$ a $\Cs^2$ Anosov flow. 
There exists  $\ve_0>0 , \Cnse\geq 1 $ such that given a $(\ve, 1)$-pseudo orbit $\mathfrak{t}(t)$ with
$\ve < \ve_0$ there exists an orbit $\tau$,  a point  $p \in \tau$ and a reparametrization $\sigma(t)$ such that, for all $t,s \in \bR$, 
\[
 d( \mathfrak{t}(t) ,\phi_{\sigma(t)}(p) ) \leq \Cnse \ve \text{ where }  |\sigma(t)-\sigma(s) - t+s| \leq \Cnse \ve|t-s|.
\]
\end{thm}

Now we can start our proof. Given $x\in M$ and  $t\in\bR$ such that $d(x, \phi_{-t}(x))<\ve$, we can construct a closed $(\ve,1)$-pseudo orbit. Let $S$ be a codimension one manifold containing $x$ and transversal to the flow.
Then, by  Theorem \ref{pilucco}, there exists a $p\in S\cap \tau$ and a $\sigma$ such that $d(\phi_{\sigma(t_1)}(p), \phi_{-t_1+nt}(x))\leq \Cnse\ve$ for all $t_1\in[ nt, (n+1)t]$, $n\in\bZ$. Next, let $t_*$ be the time of first return to $S$: i.e. $\phi_{\sigma(t_*)}(p)=q\in S$. 
Since, for all $t_1\in\bR$, $\phi_{t_1}(p), \phi_{t_1}(q)$ are always $\Cnse\ve$ close to the $(\ve,1)$-pseudo orbit, then they must belong to each other's weak stable and unstable manifolds, which implies $p=q$. Accordingly, there exists a prime periodic orbit $\tau_p\in\cT_p$ that shadows the pseudo orbit. Also, by the same argument, such an orbit is unique. It follows that
\begin{equation}\label{eq:lenght-or}
|\sigma(t)-\lambda(\tau_p)|\leq 2\Cnse \ve.
\end{equation}
 Moreover, for each $t_1\in [0,t]$, the trajectory of $\phi_{-t_1}(x)$ is $\Cnse\ve$ close to $\tau$ for a time at least $\min\{t_1, t-t_1\}$ both backward and forward in time. By the hyperbolicity this implies that $d(\phi_{-t_1}(x), \tau)\leq \Cnz\Cnse\ve
e^{-\bar\lambda \min\{t_1, t-t_1\}}$. If $q\in\tau$ is the closest point to $z= \phi_{-\frac t2}(x)$, let $\xi(t_1)=\phi_{\bar\sigma(t_1)}(q)-\phi_{t_1}(z)$ where $\bar\sigma$ is defined by\footnote{ Here we are really working in charts and using the euclidean structure, we do not write this explicitly since it obvious (but a bit notationally unpleasant) how to proceed.}
\[
\langle \xi(t_1), V(\phi_{\bar\sigma(t_1)}(q))\rangle=0,
\]
where $V$ is the vector field generating the flow.
Differentiating we have 
\[
|1-\bar\sigma'|\|V(\phi_{\bar\sigma(t_1)}(q))\|^2\leq C_\# \|\xi(t_1)\|
\]
which, by the integral Gronwald inequality, yields $|\bar\sigma(t_1)-t_1|\leq C_\# \ve$. Hence, for each $t_1\in [0,t]$,
\[
\begin{split}
|\phi_{\sigma(-t_1)}(p)-\phi_{\frac t2-t_1}(q)|&\leq |\phi_{\sigma(-t_1)}(p)-\phi_{\bar\sigma(\frac t2-t_1)}(q)|+ C_\#\ve\\
&\leq |\phi_{\frac t2-t_1}(z)-\phi_{\bar\sigma(\frac t2-t_1)}(q)|+C_\# \ve\leq C_\# \ve.
\end{split}
\]
Accordingly, there exists $\Cnt>0$ such that $|\sigma(t)-t|\leq \Cnt \ve$ whereby improving the bound in Theorem \ref{pilucco}. 
Recalling \eqref{eq:lenght-or} it follows that $(x,t)\in\Omega^\ve_\tau$.

The above proves the first part of the Sub-Lemma, to prove the second suppose $(x,t)\in \Omega^\ve_\tau\cap\Omega^\ve_{\tau'}$. 
Let $z\in\tau$ be the closest point to $x$. Then, 
\[
d(x,\phi_{t}(x))\leq 2\Cnse\ve+d(z,\phi_{t}(z))\leq (2\Cnse+\Cnt)\ve.
\]
We can then construct a closed $((2\Cnse+\Cnt)\ve,1)$ pseudo orbit. If $\ve$ is chosen small enough, by the arguments above such a pseudo orbit can shadow a unique prime orbit $\tau_p$. Hence both $\tau$ and $\tau'$ are multiples of the same orbit, but $|\lambda(\tau)-\lambda(\tau')|\leq 2\Cnt \ve$. Hence if $2\Cnt \ve$ is smaller than the period of the shortest periodic orbit, then $\tau=\tau'$.
\end{proof}

We can then rewrite \eqref{eq:ftrace-one} as
\begin{equation} \label{eq:trace1} 
\sum_{\alpha,\bar{i} ,\tau\in\cT_\ve(s)} \int_{\Omega^\ve_\tau} 
 \kappa_\ve (\Theta_{\alpha}(x) - \Theta_{\alpha}(\phi_{-t}(x))) J\Theta_{\alpha}(x) F_{\alpha, \bar{i},z}(x,t-s)\omega_M(x),
\end{equation}
where we have introduced the notation $\cT_\ve(s)\doteq \cT(s-\Cnt\ve)$.
Next, it is convenient to pass to charts, $ (\xi,t) \doteq \widetilde{\Theta}_{\alpha}(x,t) \doteq (\Theta_{\alpha}(x),t) $, $\phi_{-t}^{\alpha} \doteq \Theta_\alpha \circ \phi_{-t} \circ \Theta_{\alpha}^{-1} $. Thus, part of the expression \eqref{eq:trace1} can be rewritten as,
 \begin{equation} \label{eq:trace2}
 \sum_{\alpha,\bar{i} ,\tau\in\cT_\ve(s)} \int_{\widetilde{\Theta}_{\alpha}(\Omega_\tau \cap (U_\alpha \times \bR))}
 \kappa_\ve (\xi -  \phi^{\alpha}_{-t}(\xi))   F_{\alpha, \bar{i},z} \circ \widetilde{\Theta}_{\alpha}^{-1}(\xi,t-s).
 \end{equation}
Note that $\widetilde{\Theta}_\alpha(\Omega_\tau\cap (U_\alpha\times\bRp))$ is contained in the $\ve$-neighborhood of a finite number of lines (the connected pieces of $\Theta_\alpha(\tau \cap U_\alpha)$). Let us call
 $\{\overline{\Omega}_{\tau,\alpha,m}\}$ the collection of such connected components.

Let us set $\xi=(\tilde\xi,\xi_d)$ and the map $\Xi_{\tau,\alpha,m}:\overline{\Omega}_{\tau,\alpha,m}\to\bR^{d+1}$ given by
\begin{equation}\label{eq:Xi-coordinates}
\begin{split}
\Xi_{\tau,\alpha,m}(\xi,t)& \doteq (\xi_1-\phi^\alpha_{-t}(\xi)_{1} ,\dots,
\xi_{d-1}-\phi^\alpha_{-t}(\xi)_{d-1}, \xi_{d}, t)\\
& \doteq (\zeta_{1},\dots,\zeta_{d-1},\xi_{d},t) \doteq (\zeta,\xi_{d},t).
\end{split}
\end{equation}
Note that, if $\xi\in M$ belongs to a periodic orbit $\tau$, then $\det(D\Xi_{\tau,\alpha,m}(\xi,\lambda(\tau)))\neq 0$ by the hyperbolicity of the flow. Thus the map $\Xi_{\tau,\alpha,m}$ is locally invertible. What is not clear is if it is invertible on all $\overline{\Omega}_{\tau,\alpha,m}$.
\begin{sublem} There exists $\ve_0>0$ such that, for all $\ve\leq \ve_0$ and $\alpha, m, \tau$, the map $\Xi_{\tau,\alpha,m}$ is a diffeomorphism from $\overline{\Omega}_{\tau,\alpha,m}$ onto its image.
\end{sublem}
\begin{proof}
Suppose there exist $(z,\xi,t),(y,\eta, t')\in \overline{\Omega}_{\tau,\alpha,m}$ such that $\Xi_{\tau,\alpha,m}(z,\xi, t)=\Xi_{\tau,\alpha,m}(y,\eta, t')$. Then $t=t'$ and $\xi=\eta$ and  
\begin{equation}\label{eq:base-zy}
z-y=f(z)-f(y)
\end{equation}
where $f=\phi^\alpha_{-r_{\tau,\alpha,m}(z)}$ is the return map to the zero section. Let $W^s(v), W^u(v)$ be the local stable and unstable manifolds of $f$ at $v$ and $E^s(v), E^u(v)$ be the stable and unstable spaces. For $v\in\bR^{d-1}$ let $\Pi(v)$ be the projector on $E^s(v)$ along $E^u(v)$.
Let $x=W^s(z)\cap W^u(y)$, then $w=f(x)=W^s(f(z))\cap W^u(f(y))$. By hyperbolicity,
\begin{equation}\label{eq:contraction-f}
\|f(z)- f(x)\|\leq \Cnz e^{-\bar\lambda t}\|z-x\|\;;\quad \|y-x\|\leq \Cnz e^{-\bar\lambda t}\|f(y)- f(x)\|.
\end{equation}
By the uniform bounds on the curvature of the invariant manifolds (see Appendix \ref{app:holo-est}), 
\begin{equation}\label{eq:Pi-est}
\begin{split}
&\hskip-.2cm\|\Pi(x)(z-x)-(z-x)\|\leq C_\# \|z-x\|^2;\\
&\hskip-.2cm \|(\Id-\Pi(f(x)))(f(z)-f(x))-(f(z)-f(x))\|\leq C_\# \|f(z)-f(x)\|^2,
\end{split}
\end{equation}
and the analogous for $y$. Indeed, we can argue in the plane containing the triangle of vertices  $x, z$ and $q=x+\Pi(x)(z-x)$. Let $p$ be the orthogonal projection of $z$ on the $x,q$ line. Since the side $x, q$ is tangent to $W^s(x)$, and the stable and unstable manifolds are uniformly transversal, at $x$, $\|z-q\|+\|z-p\|\leq C_\# \|x-z\|^2$. Thus $\|z-x\|-\|x-p\|\leq C_\#\|z-x\|^3$. On the other hand $\|p-q\|\leq C_\#\|z-x\|^2$.

By the H\"older continuity of the foliations we have, for some $\varpi\in (0,1]$,
\begin{equation}\label{eq:Pio-est}
\|\Pi(x)(\Id-\Pi(f(x))\|+\|(\Id-\Pi(f(x))\Pi(x)\|\leq C_\#\ve^\varpi.
\end{equation}
Accordingly, by \eqref{eq:Pi-est}, \eqref{eq:base-zy}, \eqref{eq:contraction-f} and \eqref{eq:Pio-est},
\[
\begin{split}
\|z-x\|&\leq \|\Pi(x)(z-x)\|+ C_\# \|z-x\|^2\\
&=\|\Pi(x)[(y-x)+(f(z)-f(x))-(f(y)-f(x))]\|+ C_\# \ve\|z-x\|\\
&\leq C_\#[e^{-\bar\lambda t}+\ve^\varpi](\|z-x\|+\|f(y)-f(x)\|).
\end{split}
\]
If $\lambda(\tau)$ is large enough and $\ve$ small, then $\|z-x\|\leq \frac 12\|f(y)-f(x)\|$. By a similar argument we have $\|f(y)-f(x)\|\leq \frac 12\|z-x\|$. This proves the Lemma for all the periodic orbits with period larger than a fixed constant. For the finitely many remaining orbits the statement follows trivially from the local invertibility of the map by choosing $\ve_0$ small enough.
\end{proof}
Having established that \eqref{eq:Xi-coordinates} is a good change of coordinates, we can consider some more of its properties. 
For each $\alpha, m$, there exists a unique 
$p_{\tau,\alpha,m}\in\bR^{d-1}$ such that $(p_{\tau,\alpha,m},0)\in \Theta_{\alpha}(\tau \cap U_{\alpha})$ and
$(p_{\tau,\alpha,m},0,\lambda(\tau))\in \overline{\Omega}_{\tau,\alpha,m}$, note that  $\Xi_{\tau,\alpha,m}(p_{\tau,\alpha,m},0,t)=(0,0,t)$.
Next, define the (smooth) return time function $r_{\tau,\alpha,m}(\tilde\xi)$ by
$[\phi^\alpha_{-r_{\tau,\alpha,m}(\tilde\xi)}(\tilde\xi,0)]_d=0$ such that $r_{\tau,\alpha,m}(p_{\tau,\alpha,m})=\lambda(\tau)$. 
Note that $\phi^{\alpha}_{-\lambda(\tau)}(p_{\tau,\alpha,m},0) = (p_{\tau,\alpha,m},0)$, moreover, for $t$ close to $\lambda(\tau)$, we have 
\[
\phi_{-t}^{\alpha}(\tilde{\xi},\xi_d)=\phi^\alpha_{-r_{\tau,\alpha,m}(\tilde\xi)}(\tilde\xi,0)+(0,-t+r_{\tau,\alpha,m}(\tilde{\xi})+\xi_d).
\]
Also $r_{\tau,\alpha,m}\circ \Xi_{\tau,\alpha,m}^{-1}(\zeta,\xi_d,t)=\tilde r_{\tau,\alpha,m}(\zeta)$.
It is then natural to define the return map to the zero section $K_{\tau,\alpha,m}=\phi^\alpha_{r_{\tau,\alpha,m}(\tilde\xi)}(\tilde\xi,0)$. Clearly such a map is hyperbolic.
It is convenient to set $\Lambda_{\tau,\alpha,m}(\zeta)=[DK_{\tau,\alpha,m}]\circ \Xi_{\tau,\alpha,m}^{-1}(\zeta,0,\tilde r_{\tau,\alpha,m}(\zeta))$. 
In the sequel we will need several estimates on the regularity of the above object that we summarize in the next result.\footnote{ In fact, Sub-Lemma \ref{sublem:change-smooth} contains much more than what is presently needed (a rough bound of the type $e^{C_\#\lambda(\tau)}$ would suffice), yet its full force will be necessary in Lemma \ref{lem:tracebound}.}

\begin{sublem}\label{sublem:change-smooth}
There exists $\ve_0>0$ such that, for all $\ve<\ve_0$, $\tau\in\cT$, $\alpha\in\cA$ and $m\in\bN$, we have
\[
\begin{split}
&\|(\Id-\Lambda_{\tau,\alpha,m})^{-1}\|_\infty\leq C_\#\\
& \|\partial_{\zeta}\tilde r_{\tau,\alpha,m}\|_\infty\leq C_\#\\
&\|\partial_{\zeta}\det(\Id-\Lambda_{\tau,\alpha,m})^{-1}\|_\infty\leq C_\#\|\det(\Id-\Lambda_{\tau,\alpha,m})\|_\infty.
\end{split}
\]
\end{sublem}
\begin{proof}
Note that 
\[
D\phi_{-t}^\alpha(\xi)=\begin{pmatrix}\Lambda_{\tau,\alpha,m} \circ \Xi_{\tau,\alpha,m} (\tilde\xi, \xi_d,t) &0\\g_{\tau,\alpha,m}\circ \Xi_{\tau,\alpha,m}(\tilde\xi, \xi_d,t)&1\end{pmatrix}
\]
where the left entries do not depend on $\xi_d,t $ (hence $\Lambda_{\tau,\alpha,m}, g_{\tau,\alpha,m}$ depend only on $\zeta$). The first inequality of the Sub-Lemma follows by hyperbolicity. In addition, $g_{\tau,\alpha,m}$, by the existence of an invariant cone field, satisfies, for each $\eta\in \bR^{d-1}$, $|\langle g_{\tau,\alpha,m},\eta\rangle|\leq C_\# \|\Lambda_{\tau,\alpha,m}\eta\|$. That is
\begin{equation}\label{eq:g-der-bound}
\|[\Lambda_{\tau,\alpha,m}^{-1}]^*g_{\tau,\alpha,m}\|\leq C_\#.
\end{equation}
Also, we have
\begin{equation}\label{eq:Xi-derivative}
[D\Xi_{\tau,\alpha,m}]\circ \Xi_{\tau,\alpha,m}^{-1}(\zeta,\xi_d,t)=\begin{pmatrix} \Id-\Lambda_{\tau,\alpha,m}(\zeta)&0&0\\
                                                                                     0    & 1&0\\ 0&0&1\end{pmatrix}.
\end{equation}
Hence $\det([D\Xi_{\tau,\alpha,m}]\circ \Xi_{\tau,\alpha,m}^{-1}(\zeta,\xi_d,t))=\det(\Id-\Lambda_{\tau,\alpha,m}(\zeta))\neq 0$ by hyperbolicity.
By applying the implicit function theorem we then have
\[
\partial_{\zeta}\tilde r_{\tau,\alpha,m}=\left[ \Lambda_{\tau,\alpha,m}(\Id-\Lambda_{\tau,\alpha,m})^{-1}\right]^*[\Lambda_{\tau,\alpha,m}^{-1}]^*g_{\tau,\alpha,m}.
\]
The above formula and \eqref{eq:g-der-bound} imply the second inequality of the Lemma.
To prove the last we adapt the arguments in \cite[Sub-Lemma 3.2]{Liverani05} (see also \cite{pollicott0}).

We divide $\lambda(\tau)$ in $l$ time intervals long enough so that the Poincar\'e map is hyperbolic. Let $\{t_i\}_{i=1}^{l-1}$ be the times at which we insert a Poincar\'e section while the first and last are our zero section. In each new section we can choose coordinates so that the periodic orbit is at $(0,0)$ and, at such a point, $\{(\xi, 0)\}$ corresponds to the stable manifold and $\{(0,\eta)\}$ to the unstable one. Let $\Lambda_i$ be the derivative of the map between the $i$ and $i+1$ section. By hypothesis each $\Lambda_i$ expands vectors close to the  $\{(\xi, 0)\}$ subspace and contracts vectors close to the $\{(0,\eta)\}$ subspace. Also
\[
\Lambda_{\tau,\alpha,m}=\prod_{i=0}^{l-1}\Lambda_i.
\]
Given a point $\tilde\xi$ in the zero section, let $\tilde\xi^i$ be its image in the $i$-th section. With this notation we can write
\[
\begin{split}
&\partial_{\zeta_m}\Lambda_{\tau,\alpha,m}=\sum_{p=1}^{d-1}\frac{\partial \tilde\xi^j_p}{\partial \zeta_m}\sum_{j=0}^{l-1}P_j\Gamma_{j,p}Q_j \\
&P_j=\prod_{i>j}\Lambda_i\;\;;\quad Q_j=\prod_{i\leq j}\Lambda_i\;\;;\quad\Gamma_{j,p}=\partial_{\tilde\xi^j_p}\Lambda_j\Lambda_j^{-1}.
\end{split}
\] 
Note that 
\[
\frac{\partial \tilde\xi^j_p}{\partial \zeta_m}=Q_j(\Id-\Lambda_{\tau,\alpha,m})^{-1}=Q_j(\Id-P_jQ_j)^{-1}=(Q_j^{-1}-P_j)^{-1}
\]
and, by the hyperbolicity, it  follows that $\|\frac{\partial \tilde\xi^j_p}{\partial \zeta_m}\|\leq C_\#$. On the other hand
\[
\partial_{\zeta_m}\det(\Id-\Lambda_{\tau,\alpha,m})=-\tr\left(\partial_{\zeta_m}\Lambda_{\tau,\alpha,m}(\Id-\Lambda_{\tau,\alpha,m})^{-1}\right)\cdot \det(\Id-\Lambda_{\tau,\alpha,m}).
\]
We can then compute
\[
\begin{split}
|\tr\left(\partial_{\zeta_m}\Lambda_{\tau,\alpha,m}(\Id-\Lambda_{\tau,\alpha,m})^{-1}\right)|&\leq C_\# \sup_p\sum_{j=0}^{l-1}\left|\tr\left( P_j\Gamma_{j,p}Q_j (\Id-P_jQ_j)^{-1}\right)\right|\\
&=C_\# \sup_p\sum_{j=0}^{l-1}\left|\tr\left( \Gamma_{j,p}(P_j^{-1}Q_j^{-1}-\Id)^{-1}\right)\right|\leq C_\#
\end{split}
\]
where we have used the hyperbolicity again. The last inequality of the Sub-Lemma is thus proven.
\end{proof}
By first performing the change of variables $\Xi_{\tau,\alpha,m}$, then using the properties of the chosen $\kappa$,  and finally using the change of coordinates   $\eta=\ve^{-1}\zeta$, $v=\ve^{-1}(t-\tilde r_{\tau,\alpha,m}(\zeta))$, we can write \eqref{eq:trace2} as 
\begin{equation}\label{eq:finalizing0}
\begin{split}
\sum_{\alpha,\bar{i},m,\tau\in\cT_{\ve}(s)}\int_{\bR^{d+1}}&
F_{\alpha,\bar{i},z} \circ \widetilde{\Theta}_{\alpha}^{-1} \circ \Xi_{\tau,\alpha,m}^{-1}(\ve\eta ,\xi_d , \ve v +\tilde r_{\tau,\alpha,m}(\ve\eta) -s) \\
&\times \kappa^{d-1}(\eta)\kappa^1(v)|\det(\Id-\Lambda_{\tau,\alpha,m}(\ve\eta))|^{-1}.
\end{split} 
\end{equation}
To continue we need to estimate the $\ve$ dependence in the integrand.
Remembering the definition \eqref{eq:ftrace-one}, we see that the derivative of $F$ is, in general, of order $C_\# e^{C_\#|t|}$. Then, setting $a=\Re(z)$, Sub-Lemma \ref{sublem:change-smooth} yields
\begin{equation}\label{eq:finalizing}
\begin{split}
\sum_{\alpha,\bar{i},m,\tau\in\cT_{\ve}(s)}\int_{-\delta}^\delta&\frac{F_{\alpha,\bar{i},z} \circ \widetilde{\Theta}_{\alpha}^{-1} \circ \Xi_{\tau,\alpha,m}^{-1}(0,\xi_d , \lambda(\tau) -s)}{|\det(\Id-\Lambda_{\tau,\alpha,m}(0))|}\\
&+\cO\left(\ve|z|\frac{[\lambda(\tau)-s]^{n-1}e^{C_\#\lambda(\tau)-a[\lambda(\tau)-s]}}{(n-1)!}\right).
\end{split} 
\end{equation}
Define $D_{\textrm{hyp}}\phi_{\lambda(\tau)}=\Lambda_{\tau,\alpha,m}(0)$. If we compute this  matrix in a different chart or at another point in the orbit we simply obtain a matrix conjugated to $D_{\textrm{hyp}}\phi_{\lambda(\tau)}$. Thus $|\det(\Id-D_{\textrm{hyp}}\phi_{\lambda(\tau)})|$ depends  only on $\tau$. In addition, we have\footnote{ Given a vector space $V^d$ over $\bR$
 and a matrix representation of a linear operator $A: V^d \to V^d $, we can construct, by the standard external product, a matrix
 representation of $\wedge^\ell A :\wedge^\ell(V^d) \to \wedge^\ell(V^d)$
with elements  $a_{\bar{i},\bar{j}} =  \det(A_{\bar{i},\bar{j}}) $  (see definition \eqref{eq:scalarproduct} or \cite{MacLaneBirkhoff99}).  
By  $\det(A)_{\bar{k},\bar{k}}$ we mean the determinant of the matrix $(A_{k_i,k_j})$. Note that $\det (D\tphi_{-\lambda(\tau)}^{\alpha})_{\bar{k},\bar{k}}=0$ unless $k_\ell=d$. See Remark \ref{rem:form-details} for the definition of $\cI_\ell^-$. }
\begin{equation}\label{eq:trace-ok}
\begin{split}
\textrm{tr}(\wedge^\ell D \tphi_{-\lambda(\tau)}^{\alpha}) &= \sum_{\bar{k}\in\cI_\ell}\det (D\tphi_{-\lambda(\tau)}^{\alpha})_{\bar{k},\bar{k}} 
  = \sum_{\bar{k}\in\cI_\ell^-} \det(D_{\textrm{hyp}}\phi_{-\lambda(\tau)})_{\bar{k},\bar{k}}\\
  &  = \textrm{tr}(\wedge^\ell( D_{\textrm{hyp}}\phi_{-\lambda(\tau)} )),
\end{split}
\end{equation}
which, again, depends only on the orbit $\tau$.
Accordingly,
\[ 
\begin{split}
&\sum_{\bar i}\langle \omega_{\alpha, \bar{i}},\, \phi_{-\lambda(\tau)}^* \omega_{\alpha, \bar{i}} \rangle_{p_{\tau,\alpha,m}} = \tr(\wedge^\ell(D_{\textrm{hyp}}\phi_{-\lambda(\tau)}).
\end{split} 
\]
To conclude, remember that the number of closed orbits with period between $t$ and $t+1$  is bounded by $C_\# e^{\htop t}$ (see \cite[Theorem 18.5.7]{KH}). In fact, there exists an asymptotic formula,  see \cite{Margulis04}, which we are substantially improving in Theorem \ref{thm:main3}. Hence the error term in \eqref{eq:finalizing} is bounded by 
\[
\sum_{k\geq s}\ve|z|\frac{[k-s]^{n-1}e^{C_\#k-a[k-s]}}{(n-1)!}\leq C_\#\ve |z|(a-C_\#)^{-n}e^{C_\# s}
\]
provided $a >C_\#$.
By summing over  the connected components of $\overline{\Omega}_{\tau,\alpha,m}$, taking into account the multiplicity of orbits, resumming on $\alpha$ and recalling  \eqref{eq:tracechar} we finally obtain that there exists $\Cns,\ve_0>0$ such that, for all $\ve\leq \ve_0$ and $a>\Cns$,
\begin{equation}\label{eq:trace-quant} 
\begin{split}
&\left|\int_{M} \sum_{\alpha,\bar{i}}   \langle \omega_{\alpha,\bar i},R^{(\ell)}(z)^n\cL^{(\ell)}_s  (j_{\epsilon, \alpha,\bar i,x} ) \rangle_x -  
\sum_{\tau \in \cT(s)}  \frac{ \chi_{\ell}(\tau)\left[\lambda(\tau)-s\right]^{n-1}\lambda(\tau)}{(n-1)!\mu(\tau)e^{z (\lambda(\tau)-s)}}\right|\\
&\quad\leq C_\#\ve|z| (a-\Cns)^{-n}e^{ \frac a 2 s} .
\end{split}
\end{equation}
Taking the limit $\ve\to 0$ yields the Lemma for $a>\Cns$. On the other hand the above estimates show that both sides of the equation in the statement of the Lemma are well defined analytic functions for $a> \sigma_\ell$. Indeed, equation \eqref{eq:finalizing0} shows that the approximations to the flat trace, which are analytic, are uniformly bounded in any region $a\geq a_0>\sigma_\ell$ hence they have analytic accumulation points that must agree since they agree for large $a$. The statement of the Lemma follows.
\end{proof}

By a direct computation using Lemma \ref{lem:traceorbits} we have  the natural formula\footnote{ Indeed,
\[
\int_0^\infty e^{-zs} \frac{s^{n-1}}{(n-1)!}\trf(R(z)\cL_s)=\sum_{\tau\in\cT} \int_0^{\lambda(\tau)}  \frac{s^{n-1}}{(n-1)!}
\frac{\chi_{\ell}(\tau)}{\mu(\tau)}   \lambda(\tau) e^{-z \lambda(\tau)}.
\]
}
\begin{equation}\label{eq:natural-trace}
\int_0^\infty e^{-zs} \frac{s^{n-1}}{(n-1)!}\trf(R(z)\cL_s)=\trf(R(z)^{n+1}).
\end{equation}

\section[Tensorial Operators]{Tensorial Transfer Operators} \label{sec:splitting}

In this section we extend the methods of Liverani-Tsujii \cite{LiveraniTsujii06} to the case of flows. The goal is to provide a setting in which a formula of the type \eqref{eq:distrib-trace} makes sense and can be used to compute the trace  (see Lemma \ref{splitR1} for the exact implementation). The first step is to note that, by equation \eqref{hodgedualform}, the adjoint (with respect to $\langle\cdot, \cdot\rangle_{\Omega^\ell}$) of $\cL^{(\ell)}_t$ is given by, for $g \in \Omega^{\ell}_{0,r}$,
\begin{equation} \label{eq:formaladjoint}
\overline{\cL}^{(\ell)}_t g \doteq (-1)^{\ell(d-\ell)} *\! ( \phi_{t}^* (*g) ).
\end{equation}
Next we would like to take the tensor product of $\cL^{(\ell)}_t$ times $ \overline{\cL}^{(\ell)}_t$, and define a Banach space, connected to the product space $\Omega_r^{2\ell}(M^2)$,\footnote{ By $M^2$ we mean the Riemannian manifold with the product metric.} on which it acts naturally. Note that, contrary to the discrete case, in the continuous setting this procedure naturally yields a $\bRp^2$ action in the variables $s,t \geq 0$ (rather than a flow). We start with the construction of the Banach space.

\subsection{Spaces and Operators}\label{subsec:spaceop}

We use the construction developed in Section \ref{banachspace} applied to the manifold $M^2$. 

First of all consider the atlas  $\{U_\alpha, \Theta_\alpha\}_{\alpha\in\cA}$ chosen at the beginning of Section \ref{secresolvent}. We define the map $I(u,s,t, u',s',t')=(s,u',u,s',t, t')$, $u, u'\in\bR^{d_u}$, $s, s'\in\bR^{(d_s)}$, $t, t'\in\bR$, the atlas
$\{U_{\alpha} \times U_{\beta} , I\circ(\Theta_\alpha \times \Theta_\beta)\}_{\alpha,\beta\in\cA}$ and the partition of unity $\{\psi_{\alpha,\beta}\}_{\alpha,\beta\in \cA}$, where $\psi_{\alpha,\beta}(x,y)=\psi_\alpha(x)\psi_\beta(y)$. We are thus in the situation of Section \ref{banachspace} with $d_1=d_u+d_s=d-1$ and $d_2=d+1$. Note that  the conditions \eqref{eq:chartproperty} are automatically satisfied.\footnote{ Note that the third condition holds both for $x_{2d}$ and $x_{2d-1}$.}
We choose the cones given by the choice $\rho_+=2, \rho_-=1$ and $L_0$ as in Section \ref{secresolvent} to define the set of ``stable" leaves, which we  denote by  $\Sigma_2$.
\begin{rmk} From now on we will ignore the map $I$ since it is just a trivial permutation of the coordinates.
\end{rmk}
By appendix \ref{appendiceLie} we can consider the exterior forms $\Omega^{2\ell}_r (M^2)$ and the related scalar product. 
 We define the projections $\pi_i:M^2\to M$, $i\in\{1,2 \}$, such that $\pi_1(x,y)=x$ and $\pi_2(x,y)=y$. For each pair of $\ell$-forms $f,g$ in $\Omega^\ell_r(M)$ we have that $\pi_1^*f \wedge\pi_2^*g \in\Omega^{2\ell}_{r}(M^2)$.
In addition, given  $a,b,f,g\in\Omega_{r}^{\ell}(M)$, by equation \eqref{eq:scalarproduct}, we have
 \begin{equation} \label{eq:scalarxy}
\langle \pi_1^* a \wedge  \pi_2^* b , \pi_1^* f \wedge \pi_2^* g \rangle_{(x,y)} = \langle a , f \rangle_x \langle b, g \rangle_y .
\end{equation}
Moreover, by \eqref{eq:hdgedef}, it follows that 
\begin{equation}\label{eq:hodgexy}
*(\pi_1^* f \wedge \pi_2^* g)=(-1)^{\ell(d-\ell)}\pi_1^*(* f) \wedge \pi_2^* (*g).
\end{equation}
Next we define the vector space $\Omega_{2,r}^\ell(M) \doteq \operatorname{span}\{\pi_1^*f \wedge\pi_2^*g\;:\; f,g\in\Omega^\ell_{0,r}(M)\}$ on which we intend to base our spaces. Note that locally the $\Cs^{s}$ closure of  $\Omega_{2,s}^\ell(M)$ contains all the forms $h$ such that $i_{(V,0)}h=i_{(0,V)}h=0$ and that can be written as 
\begin{equation}\label{eq:h2-rap}
h = \sum_{\alpha,\beta} h_{\alpha,\beta} = \sum_{\alpha,\beta,\bar i,\bar j} \psi_{\alpha,\beta}(x,y) h^{\alpha,\beta}_{\bar i,\bar j}(x,y) 
\omega_{\alpha,\bar i}(x) \wedge \omega_{\beta,\bar j}(y),
\end{equation} 
where $h_{\alpha,\beta} \doteq \psi_{\alpha,\beta}h$,  and $h^{\alpha,\beta}_{\bar i,\bar j} \in\Cs^{\bar s}(M^2)$, for each $\bar s > s$.\footnote{ \label{foo:little-holder} This follows since $\Cs^\infty$ is dense in $\Cs^{\bar s}$ in the $\Cs^s$ topology and smooth functions can be approximated by tensor functions, e.g. by Fourier series, in any $\Cs^s$ norm. Note that if $s\in\bN$, then one can choose $s=\bar s$ (for a more refined description, not used here, see the theory of {\em little H\"older spaces}).}

Lastly, we must choose an appropriate set of test functions and vector fields. For the set of test functions we choose $\Gamma_{2,c}^{\ell,s} (\alpha,\beta, G)$, defined by the restriction of $\operatorname{span}\{\pi_1^*f \wedge\pi_2^*g\;:\; f,g\in\Omega^\ell_{s}(M), \pi_1^*f \wedge\pi_2^*g|_{\partial W_{\alpha,\beta, G}}=0\}$ to $W_{\alpha,\beta, G}$. The set of vector fields have only the  restriction that\footnote{ Let $(x,y)\in M^2$, then we are requiring  $v(x,y)=(v_1(x),v_2(y))$.}
\begin{equation}  \label{eq:vector-product} 
\begin{split}
&(\pi_1)_*(v)=0\Longrightarrow  w\in\cV^s(M^2) \text{ and } (\pi_2)_*(w)=0 \longrightarrow [w,v]=0 \\
&(\pi_2)_*(w)=0\Longrightarrow  v\in\cV^s(M^2) \text{ and } (\pi_1)_*(v)=0 \longrightarrow [v,w]=0. 
\end{split}
\end{equation}
Note that this implies
\begin{equation}  \label{eq:Liesplit} 
[v,w]=([(\pi_1)_*v, (\pi_1)_*w], [(\pi_2)_*v, (\pi_2)_*w]).
\end{equation}

The reader can easily check that the above choices satisfy all the conditions specified in Section \ref{banachspace}. In particular, Lemma \ref{lem:extension} holds also in the present context, since the arguments in its proof respects the property \eqref{eq:vector-product}.
Thus we can apply  the construction in Section \ref{banachspace} and  call $\pqnormbis{\cdot}$ the resulting norms and  $\cB^{p,q,\ell}_2 \doteq  \overline{\Omega^\ell_{2,r}(M)}^{\|\cdot\|_{p,q,\ell,2}}$ the corresponding   Banach spaces. 

Finally, we define the required  operators $\cL^{(\ell)}_{t,s}$.
For each $ f,g\in\Omega^\ell_{r}(M)$
\begin{equation} \label{eq:wedgetransfer}
\mathcal{L}_{t,s}^{(\ell)} (\pi_1^*f \wedge\pi_2^*g) \doteq(\pi_1^*(\cL_t^{(\ell)}f) )
\wedge(\pi_2^*(\overline{\cL}^{(\ell)}_s g)) \end{equation}
which extends by linearity to an operator $\cL^{(\ell)}_{t,s}: \Omega_{2,r}^\ell(M) \to \Omega_{2,r}^\ell(M)$.

To avoid the problem of small times, for which the cone contraction might fail, we apply the same strategy used in the previous section: we define  
 \[ 
\spqnormbis{h} = \sup_{t,s \leq t_0} \pqnormbis{\cL^{(\ell)}_{t,s} h} ,
\] 
and we define the spaces  $\widetilde{\cB}_2^{p,q,\ell} \doteq \overline{\Omega_{2,r}^\ell}^{\spqnormbis{\cdot}}\subset \cB_2^{p,q,\ell}$. 

Following Remark \ref{rmk:oldbpq}, equation \eqref{hodgedualform}, setting $\omega_2=\pi_1^*\tilde\omega\wedge\pi_2^*\tilde\omega$, yields
\begin{equation*}
\mathcal{L}_{t,t}^{(d-1)} (f\omega_2)  =  f  \circ ( \phi_{-t} \times \phi_t )  J\phi_{-t}\omega_2. 
\end{equation*}
That is, we recover the same type of operators studied in \cite{LiveraniTsujii06}.

\subsection{Lasota-Yorke inequalities}\label{subsec:strongLY}
We can now obtain several results parallel to those  in Sections \ref{subsec:transfer}, \ref{subsec:resol}. As the proofs are almost identical to those  in Section \ref{subsec:LY} we will not give full details and highlight only the changes that need to be made.

\begin{lem} \label{lem:LYstrong2}
For each $0<p+q<r-1$, $t,s\in\bRp$, $\cL^{(\ell)}_{t,s}\in L( \widetilde{\cB}^{p,q,\ell}_2, \widetilde{\cB}^{p,q,\ell}_2 )$. More precisely  $\cL^{(\ell)}_{t,s}$ is an $\bRp^2$ action over $\widetilde{\cB}^{p,q,\ell}$. Moreover, $\cL^{(\ell)}_{t,0}$ and $\cL^{(\ell)}_{0,s}$ are strongly  continuous semigroups with generators $X_1, X_2$, respectively. In addition, we have\footnote{ Here and in the sequel  we suppress the $\lambda$ dependence in the constants to simplify the  notation.}
\[
\begin{split}
\spqnormbis[p,q,\ell]{\cL^{(\ell)}_{t,s}h} \leq&  C_{p,q}e^{\sigma_{\ell}(t+ s)}\spqnormbis[p,q,\ell]{h} \\
\spqnormbis[p,q,\ell]{ \cL^{(\ell)}_{t,s} h } \leq & C_{p,q}e^{\sigma_{\ell}(t+ s)} e^{- \sigma_{p,q}\min\{t,s\}} \spqnormbis[p,q,\ell]{h} + 
   C_{p,q}e^{\sigma_{\ell}(t+ s)}\spqnormbis[p-1,q+1,\ell]{h}\\
& +\sum_{j=1}^2 e^{\sigma_{\ell}(t+ s)}\spqnormbis[p-1,q+1,\ell]{X_j h}.
\end{split}
\] 
\end{lem}
\begin{proof}
As in Section \ref{secresolvent}, we  first  prove an analogue  of Lemma \ref{lem:LY}, and then  we prove a stronger version of it, as in Lemma \ref{lem:LYstrong}. 

The proof contained in Subsection \ref{subsec:LY} can be followed almost verbatim. By the  same construction one  obtains the equivalent of \eqref{eq:mani-split}, that is 
\[
\begin{split}
& \int_{W_{\alpha, \alpha',G}}  \langle g,  \cL_{t,s}^{(\ell)}h 
\rangle  \volform  =  \sum_{\stackrel{\scriptstyle  \beta, \beta' \in  \cA}{k \in \widetilde{K}_{\beta,\beta'}}}
\sum_{\bar i,\bar j, \bar i,\bar j'}
\int_{W_{\beta,\beta',G_k}}  (-1)^{(d-\ell)\ell} \psi_{\beta,\beta'} J\phi_{t-s}^{-1}\\
&\quad \times J_W(\phi_t \times \phi_{-s})
g_{\bar i',\bar j'}\circ (\phi_t\times\phi_{-s})\cdot \langle \omega_{\beta,\bar i}, *\phi_{t}^{*}*\omega_{\alpha, \bar j}\rangle 
\langle \omega_{\beta,\bar i'}, \phi_{-s}^{*}\omega_{\alpha, \bar j'}\rangle \volform. \\ 
\end{split} 
\]
For the case $ \wp = 0$, the argument is exactly the same as in  Section \ref{subsec:LY}, apart from the estimate of the volume of $\phi_{-t}\times\phi_s(W_{\alpha, \alpha',G})$ to which Appendix \ref{app:topent} cannot be applied directly.  To estimate such a volume, after setting $(z_1,z_2)=(\Theta_{\alpha}\times \Theta_{\alpha'})^{-1}(G(0))$, let us consider $W(z)\doteq W^s_{6\delta}(z_1)\times W^u_{6\delta}(z_2)$ and the holonomy from $W_{\alpha, \alpha',G}$ to $W(z)$ determined by the weak unstable foliation in the first coordinate and the weak stable in the second. Clearly the distance between the corresponding points in the images of  $W_{\alpha, \alpha',G}$ and $W(z)$ is uniformly bounded; hence the required  volume is proportional to the volume of  $\phi_{-t}\times\phi_s(W(z))$, which is bounded by $e^{\htop (t+s)}$. 

For the case $\wp > 0 $,  by equations \eqref{eq:Liesplit}, we can reorder the vector fields so as to have first the vector fields tangent to $W_{\beta,\beta',G_k}$, then the vector fields in the unstable direction of $\phi_t\times\phi_{-s}$ and then the two neutral directions.
We can then proceed as in equation \eqref{eq:stable-two}. All the following computations hold verbatim apart from two issues. 

First, in equation \eqref{eq:stokes} the last term yields a multiplicative factor that does not produce a new legal test function. Indeed, such a factor is the sum of the a divergence of the vector field (which gives no problems) and the scalar product of the vector field time a $\Cs^{r-1}$ vector, call it $A$. As $A$ comes from taking derivatives of $G$, if follows that it may not be of the required product structure. The problem is easily solved: since $r-1>p+q$, we can approximate $A$ in the $\wp-1+q$ topology by vectors $A_n\in\Gamma_{2,c}^{\ell,\wp-1+q}(\alpha,\beta,G)$ with the appropriate tensor product structure. The required  inequality follows.

Second, the weakest  contraction is now given by the case in which all the vector fields act on the component with the smallest time, hence the $\min\{t,s\}$ factor.
\end{proof}

\begin{lem} \label{lem:quasicompact2} 
For each $0<p+q<r-1$, $\ell\in\{0,\dots, d-1\}$, and for each  $\Re(z)\doteq a>\sigma_\ell$, the operator
\begin{equation}\label{eq:r2def} 
R_{2}^{(\ell)}(z)^n = \frac{1}{(n-1)!^2}
\int_0^\infty \int_0^\infty  (ts)^{n-1} e^{-z(t+s)} \mathcal{L}_{t,s}^{(\ell)} \;dt ds 
\end{equation}
satisfies
\[
\begin{split}
 &\spqnormbis[p,q,\ell]{R_2^{(\ell)}(z)^n} \leq C_{p,q} (a-\sigma_{\ell})^{-2n} \\
 &\spqnormbis{R_2^{(\ell)}(z)^n h} \leq C_{p,q,a}\left\{ \frac{ \spqnormbis{h} }{(a-\sigma_{\ell}+ \frac{\sigma_{p,q}}2)^{2n}}
+  \frac{(|z|+1)}{(a- \sigma_{\ell})^{2n}} \spqnormbis[p-1,q+1,\ell]{h}\right\}.
\end{split}
\]
Hence $R^{(\ell)}(z)$ is a linear operator on $\widetilde{\cB}^{p,q,\ell}_2$ with spectral radius bounded by 
$(\Re(z)-\sigma_\ell)^{-2}$ and essential spectral radius bounded by $(\Re(z)-\sigma_\ell +\frac{\sigma_{p,q}}2)^{-2}$. 
\end{lem}
\begin{proof}

Again the proof follows closely Subsection \ref{subsec:resol}, and more precisely Lemma \ref{lem:quasicompactness}. 
 The only difference rests in the need to decompose the domain of integration of \eqref{eq:r2def} into four pieces: $A_1=\{(t,s)\in\bRp^2\;:\; t\leq t_0,s\leq t_0\}$, $A_2=\{(t,s)\in\bRp^2\;:\; t\leq t_0,s > t_0\}$, $A_3=\{(t,s)\in\bRp^2\;:\; t > t_0,s\leq t_0\}$ and $A_4=\{(t,s)\in\bRp^2\;:\; t > t_0,s > t_0\}$. 
 
 The estimate of the integration over $A_4$ follows verbatim the argument in Lemma \ref{lem:quasicompactness}, except that one then obtains (instead of equation \eqref{eq:ly-res-lim})
 \begin{equation} \label{eq:r2n}
\begin{split}
& \int_{t_0}^\infty \! \! \! \int_{t_0}^\infty \frac{C_{p,q}(ts)^{n-1} }{(n-1)!^2e^{[a - \sigma_\ell](t+s) +\min\{t,s\}\sigma_{p,q}}}  \spqnormbis[p,q,\ell]{h}    dt ds  \\
 & \qquad \quad + \frac{C_{p,q}(|z|+1)}{(a - \sigma_\ell)^{2n}}
 \spqnormbis[p-1,q+1,\ell]{h} .
\end{split} 
 \end{equation}
The integral in \eqref{eq:r2n} is estimated as follows. 
\[
\begin{split}
& \int_{t_0}^\infty \! \! \! \int_{t_0}^\infty \frac{e^{-[a - \sigma_\ell](t+s) -\min\{t,s\}\sigma_{p,q}}(ts)^{n-1} }{(n-1)!^2} 
=2 \int_{t_0}^\infty \hskip-.4cm dt \int_{t}^\infty\hskip-.4cm ds \frac{e^{-[a - \sigma_\ell](t+s) -t\sigma_{p,q}}(ts)^{n-1} }{(n-1)!^2}    \\
& \leq   \sum_{k=0}^{n-1} \frac{2 }{(n-1)! k! }
\! \int_{0}^\infty  \!  dt  \, \frac{t^{n + k - 1} e^{-(2a - 2\sigma_\ell + \sigma_{p,q})t}}{(a -\sigma_\ell)^{n-k}} \\
& = 2 \sum_{k=0}^{n-1} \binom{n + k - 1}{n-1} \frac{1}{ (a -\sigma_\ell)^{n-k}(2a -2\sigma_\ell+\sigma_{p,q})^{n+ k}}  \\
& \leq 2 (2a - 2\sigma_\ell+ \sigma_{p,q})^{-n}(a -\sigma_\ell)^{-n}  
 \sum_{k=0}^{n-1} \binom{n + k - 1}{n-1}\frac{(a -\sigma_\ell)^{k}}{(2a -2\sigma_\ell+\sigma_{p,q})^{k}}  \\
&  \leq (a -\sigma_\ell+\frac{\sigma_{p,q}}2)^{-n}(a -\sigma_\ell)^{-n}.
\end{split} 
\]
Similarly, the integrals over  $A_2, A_3$ are bounded by 
\[ 2 \int_{0}^\infty dt  \int_0^{t_0} ds\frac{e^{-[a - \sigma_\ell](t+s) -s\sigma_{p,q}}(ts)^{n-1} }{(n-1)!^2} \leq \frac{C_{p,q} \vuo^n}{(a-\sigma_\ell)^{n}n!}\leq \frac{C_{p,q}}{(a-\sigma_\ell+\frac {\sigma_{p,q}}2)^{2n}},
 \]
 provided $n\geq \frac{(a-\sigma_{\ell}+\frac{\sigma_{p,q}}{2})^2 \vuo e}{(a-\sigma_\ell)}$, by analogy with \eqref{eq:res-small-t}. The estimate of  the integral over $A_1$ is
 treated similarly.
 \end{proof}

\subsection[Trace representation]{Tensor representation of the trace}\label{subsec:tensortrace}

Following the scheme of \cite{LiveraniTsujii06},  we define a suitable delta like functional acting on
 $\widetilde{\cB}^{p,q,\ell}_2$ and we construct an approximation scheme for such a functional. For each $f,g\in\Omega_r^\ell(M)$,  we define
 \begin{equation} \label{eq:dscalar}
\delta_2^\ell(\pi_1^*f \wedge\pi_2^*g ) \doteq \int_M \langle f,  g \rangle_x \omega_M(x) \,.
\end{equation} 
Such a definition extends by linearity and density to all sections in the closure of $\Omega_{2,r}^\ell(M)$ with respect to the $\Cs^r$ topology. Thus we obtain
\begin{equation}\label{eq:d2c}
\begin{split}
\delta_2^\ell (h)&=\sum_{\alpha,\beta}\sum_{\bar i, \bar j}\int _{M}
\psi_{\alpha,\beta}(x,x) h^{\alpha,\beta}_{\bar i,\bar j}(x, x) \langle \omega_{\alpha,\bar i}(x),\omega_{\beta,\bar j}(x)\rangle \omega_M(x)\\
&=\sum_{\alpha, \bar i}\int _{M} \psi_{\alpha}(x)h^{\alpha,\alpha}_{\bar i,\bar i}(x, x)  \omega_M(x).
\end{split}
\end{equation}
Given $\kappa_\ve$ as in Definition \ref{def:molli}, we set $J_\ve(x,y)=0$ if $d(x,y) > \delta$,  otherwise we set 
\begin{equation}\label{eq:maybe}
J_\ve(x,y) \doteq \sum_{\alpha , \bar i}\psi_{\alpha}(y)\cdot J\Theta_\alpha(y)\cdot\kappa_\ve(\Theta_\alpha(x)-\Theta_\alpha(y))\omega_{\alpha,\bar i}(x)\wedge \omega_{\alpha,\bar i}(y) .
\end{equation}
\begin{lem}\label{lem:leftje}
For each $h\in\Omega_{2,s}^\ell(M)$, $s<r$,we have that
\begin{equation} \label{eq:leftje} 
\lim_{\ve \to 0} \langle J_\ve,h\rangle_{\Omega^{(2\ell)}} = \delta^{\ell}_2(h).
\end{equation}
\end{lem}
\begin{proof}
For $f,g\in\Omega_\infty^\ell(M)$, let $h  = \pi_1^*f\wedge\pi_2^*g$, then
\[ \begin{split} 
& \lim_{\ve \to 0} \langle  J_\ve,h\rangle_{\Omega^{(2\ell)}}  = 
\int_{M^2}   \langle J_\ve,  \pi_1^*f \wedge\pi_2^*g \rangle_{(x,y)} 
\omega_M(x) \wedge \omega_M(y) \\
& =  \lim_{\ve \to 0} \int_{M^2} \sum_{\alpha}  \psi_{\alpha}(y)  J\Theta_\alpha(y ) \cdot \kappa_\ve(\Theta_\alpha(x)-\Theta_\alpha(y))     \langle f,g\rangle_x  \omega_M(x) \wedge \omega_M(y).
\end{split} 
\]
The result follows, for $h$ of the above form, by integrating with respect to $y$.  We can extend this result to a generic $h$ by linearity of the scalar product and by density (see footnote \ref{foo:little-holder}).
\end{proof}

Next we need the equivalent of Lemma \ref{lem:currents},  the proof is omitted since it is  exactly the same as before, 
 starting from the space $\Omega^\ell_{2,r}(M)$ and choosing $k$ appropriately.
 \begin{lem} \label{lem:immersion2} 
There exists an injective immersion $ \jmath_2 : \widetilde{\cB}^{p,q,\ell}_2 \to \left(\Omega^\ell_{2,r}(M)\right)' $.
\end{lem}

\begin{lem} \label{lem:delta2ext} 
The current $\delta_2^\ell$ extends uniquely to an element of $\left( \widetilde \cB_2^{p,q,\ell} \right)'$. 
\end{lem}
\begin{proof}
Since $\widetilde{\cB}_2^{p,q,\ell}$ is defined by the closure of the sections in $\Omega^\ell_{2,r}$, it 
suffices to prove that there exists $c>0$ such that $\delta_2^\ell(h) \leq c\spqnormbis{h}$ for each $h\in \Omega^\ell_{2,r}$. 

Let $W_D \doteq \{(x,y)\in M^2\;:\; x=y\}$ and recall that $\delta^\ell_2$ corresponds to integrating on such a
manifold by \eqref{eq:d2c}. If $x\in U_\alpha$ we can foliate $W_D$, in the local chart $V_\alpha$, with
 the manifolds $W_{\alpha, G_{\varsigma}}\in\widetilde{\Sigma}$ given by the graph of the functions 
$G_{\varsigma}(x^s,y^u)=(x^s,y^u,\varsigma,x^s,y^u,\varsigma)$. Accordingly, 
\begin{equation*} 
\delta^\ell_2(h) \leq \sum_{\alpha,\bar i} 
\int d\varsigma \left|\int_{W_{\alpha, G_{\varsigma}}} \langle
  \psi_{\alpha}(x) \omega_{\alpha,\bar i}\wedge \omega_{\alpha,\bar i},h\rangle\right| 
\leq C_\#\|h\|_{0,q,\ell, 2}. \hspace{+10mm} \qedhere
\end{equation*} \end{proof}

At this point, we would like to make sense of the limit of $J_\ve$ in $\widetilde{\cB}^{p,q,\ell}_2$. Unfortunately, this can be done only at a price.

\begin{lem}\label{lem:delta-app}
For $\Re(z)$ sufficiently large, the sequence $R^{(\ell)}_2(z)\cL_{t_0,t_0}J_{\ve}$ is a Cauchy sequence 
in $\widetilde\cB^{p,q,\ell}_{+,2}$.  We call $\bar\delta^\ell_2(z)$ the limit of such a sequence. 
\end{lem}
\begin{proof}
Let us start by showing that the sequence is bounded in $\widetilde{\cB}^{0,q,\ell}_{+,2}$.  Let $W_{\alpha,\beta, G}\in\Sigma_2(2,L_+)$ and $ g \in \Gamma_{2,c}^{\ell,r} $, then by equations \eqref{hodgedualform}, \eqref{eq:scalarxy} and \eqref{eq:hodgexy} we have that
\begin{equation}  \label{lastchange}
\begin{split}
&\int_{W_{\alpha,\beta, G}}\langle g, R^{(\ell)}_2(z)\cL_{t_0,t_0} J_\ve   \rangle = \int_{\bRp^2}ds\,dt\;  e^{-z(t+s)} \int_{W_{\alpha,\beta, G}} J\phi_{-t-t_0}(x) \\
& \qquad \quad \times J \phi_{s+t_0}(y) \langle (*\phi_{t+t_0}^*  \! \! \! *  \times  \phi_{-s-t_0}^* )g,  J_\ve)\rangle\circ  \phi_{-t-t_0}(x) \times \phi_{s+t_0}(y).
\end{split}
\end{equation}
Note that, provided $t_0$ has been chosen large enough, the tangent spaces of the manifold
$\phi_{-t-t_0}\times\phi_{s+t_0}(W_{\alpha,\beta, G})$ can be covered by manifolds in $\Sigma_2(1,L_0)$.

Let $\widetilde{p} \in C^{\infty}(\bR,\bRp)$, $\supp(\widetilde{p}) \subset (- \delta, \delta )$, $\widetilde p(-t)=\widetilde p(t)$,  be such that $\sum_{n \in \bZ} \widetilde{p}(t + n\delta) = 1 $ for all $t \in \bR$. Let
\[  \begin{split}
&F_{n,m}(x,y,t,s)  \doteq   J\phi_{-t-\vuo}(\phi_{n\delta + t_0}(x)) J \phi_{s+\vuo}(\phi_{ -m\delta  -t_0}(y)) \\ 
&\quad\times \det\left( \left. \! D \left( \phi_{-n\delta-t_0 } \times \phi_{m\delta +t_0} \right) \right|_{T{W}_{\alpha,\beta, G}} \right)^{-1} \circ \phi_{n\delta+t_0}\times\phi_{-m\delta-t_0}(x,y).
\end{split} \]
By the change of variable $(x', y') \doteq \phi_{-n\delta-t_0}\times\phi_{m\delta+t_0}(x,y)$,  \eqref{lastchange} becomes 
\[ 
\begin{split}
& \sum_{n,m \in \bN}    \int_{\bRp^2}   ds\,dt \, \widetilde{p}(-t + n\delta) \widetilde{p}(s  - m\delta)   e^{-z(t+s)}  \\
& \times  \int_{\phi_{-n\delta-\vuo}\times\phi_{m\delta+\vuo}(W_{\alpha,\beta,G}) } \hskip-2cm F_{n,m} \cdot  \langle (*\phi_{t+ t_0}^* * \times  \phi_{-s-t_0}^* )g, J_\ve\rangle\circ\phi_{-t +n\delta}\times\phi_{s -m\delta}  .\\
\end{split}  
\]
Note that if $(x,y)\in\phi_{-n\delta-\vuo}\times\phi_{m\delta+\vuo}(W_{\alpha,\beta,G})$ and  $y\in\supp\psi_\gamma$, then the integrand is different from zero only if $x\in U_\gamma$. Let $\{W_{\gamma,\gamma,k}\}_{K_{\gamma,n,m}}\subset \Sigma_2(1,L_0)$ be a covering of $\phi_{-n\delta-\vuo}\times\phi_{m\delta+\vuo}(W_{\alpha,\beta,G})$.
Recalling equation \eqref{eq:maybe}, we can then rewrite the previous formula\footnote{ To be precise we should treat separately the terms with $n=0$ or $m=0$, but  we leave this as an exercise  for  the reader since it is quite simple to handle, although a bit cumbersome.} by setting $t' = t  - n\delta $, $ s' = s -m\delta $. Thus we obtain
\begin{equation}\label{eq:time-partitioning}
\begin{split}
& \sum_{n,m \in \bN}\; \sum_{\gamma\in\cA, \bar i\in\cI}\;\sum_{k\in K_{\gamma,n,m}}   \int_{\bR^2}   ds'\,dt' \, \widetilde{p}(t') \widetilde{p}(s')   e^{-z( t'+s' +(n+m)\delta)}   \\
& \times \int_{W_{\gamma,\gamma,k}} F_{n,m}(x,y,  t' +n\delta, s' +m\delta)  \cdot  \psi_\gamma\circ \phi_{s'}(y) J\Theta_\gamma\circ \phi_{s'}(y)\\
&\times \kappa_\ve(\Theta_\gamma(\phi_{-t'}(x))-\Theta_\gamma(\phi_{s'}(y)))\\
&\times \langle (*\phi_{ t' + t_0+ n\delta}^* * \times  \phi_{-s' -t_0-m\delta}^* )g, \omega_{\gamma,\bar i}\wedge\omega_{\gamma,\bar i}\rangle\circ\phi_{-t'}\times\phi_{s'}  .\\
\end{split} 
\end{equation}
Recall that the manifolds $\Theta_\gamma(W_{\gamma,\gamma,k})$  are graphs of the type\footnote{ We drop the subscript $n,m, \gamma,  k $ from  $\hat{x} ,\hat{y} , H^u, H^s, H^{0,1} , H^{0,2} $ in the following equation since it is clear that such graphs depend on  all these choices.  Moreover we drop $n,m, \gamma $ from $G_k$ since it will always be clear which open set we are considering. }
\[ 
\begin{split}
G_{k} (x^s , y^u )  & \doteq  (x^s + \hat{x}^s ,  H^{u}(x^s, y^u)  , H^{s}(x^s, y^u)  , y^u + \hat{y}^u , H_1^{0}(x^s, y^u)  ,  H_2^{0} (x^s, y^u) )
\end{split} 
\]
where, since $T \Theta_\gamma^{-1}(W_{\gamma,\gamma,k})\in\cC_{\frac 12}$, $\max\{\|\partial_xH^u\|+\|\partial_xH^s\|,\partial_yH^u\|+\|\partial_yH^s\|\}\leq \frac 12$.
It is then natural to set $\overline G_k=\Theta_\gamma^{-1}\circ G_k$ and 
\[
\begin{split}
&\widehat{F}_{\gamma,n,m,k, \bar{i}} (x^s,y^u,s',t')=
F_{\gamma,n,m} (\overline G_k(x^s,y^u), t' + n\delta, s' +m\delta)\\
&\quad \times[J\Theta_\gamma \cdot \psi_\gamma]\circ \Theta_\gamma^{-1}(G_k(x^s,y^u)+(0,0,0,s')) \\
&\quad\times \langle   ( *\phi_{ t' + t_0 + n\delta}^* * \times  \phi_{-s' -t_0 -m\delta}^* )g,  
\omega_{\gamma ,\bar i} \wedge \omega_{\gamma ,\bar i}  \rangle\circ \Theta_\gamma^{-1}(G_k(x^s,y^u)+(0,0,-t',s')).
\end{split}
\]
We can then continue our computation and write
\begin{equation} \label{eq:worstthanugly} 
\begin{split}
&\sum_{n,m \in \bN}\; \sum_{\gamma\in\cA, \bar i\in\cI}\;\sum_{k\in K_{\gamma,n,m}}   \int_{\bR^2} ds'\,dt' \; \widetilde{p}(t' ) \widetilde{p}(s')   e^{-z(t'+s' +(n+m)\delta)}     \\
&  \quad   \times \int_{\bR^{d-1}}  dx^s dy^u\widehat{F}_{\gamma,n,m,k, \bar{i}} ( x^s , y^u, t',s')    \\
& \quad \quad\times \ve^{-d}\kappa^{d-1}( \ve^{-1} (x^s  - H^{s} (x^s, y^u) + \hat{x}^s,  H^{u} (x^s, y^u)   - y^u  - \hat{y}^u))  \\
& \quad \quad\times  \kappa^1( \ve^{-1}(H^{0,1} (x^s, y^u)  - H^{0,2} (x^s, y^u) - t' - s' )).
\end{split} 
\end{equation}
Note that we have $\| \widehat{F} \|_{\cC^0} \leq C_\# \|g \|_{\Cs^0}$  as in equation \eqref{eq:zero-test}.
In addition, the map $\Xi$ defined by
\begin{equation}\label{eq:simplechange}
\begin{split}
&\xi =x^s-H^s(x^s, y^u)  + \hat{x}^s ,\\
&\eta =H^u (x^s, y^u)-y^u   - \hat{y}^u, \\
&\tau=-H^{0,1} (x^s, y^u)  + H^{0,2} (x^s, y^u) + t' + s',\\
&\varsigma=s'
\end{split}
\end{equation}
is locally invertible, hence it can be used as a  change of variables. Thus, setting $a=\Re(z)$, we can bound \eqref{eq:worstthanugly} by
\[
\begin{split}
\sum_{n,m \in \bN}\; \sum_{\gamma\in\cA }\;\sum_{k\in K_{\gamma,n,m}} C_z e ^{-a(n+m)\delta} \|g\|_{\Cs^0}&\leq
C_z\sum_{n,m \in \bN} e ^{(\htop-a)(n+m)\delta} \|g\|_{\Cs^0}\\
&\leq C_z (a-\htop)^{-2} \|g\|_{\Cs^0}.
\end{split}
\]
Which, taking the sup on the manifolds and test forms, yields 
\[
\spqnormbis[0,q,\ell]{R^{(\ell)}_2(z)\cL^{(\ell)}_{\vuo,\vuo}J_\ve}\leq C_z (a-\htop)^{-2}.
\]
Next, given $\ve>\ve'>0$, by using equation \eqref{eq:worstthanugly}, \eqref{eq:simplechange} and the intermediate value theorem we have, for $\Re(z)$ large enough, 
\[ 
\begin{split}
&\left| \int_{W_{\alpha,\beta, G}}\langle g, R^{(\ell)}_2(z)\cL^{(\ell)}_{t_0,t_0}J_\ve-R^{(\ell)}_2(z)\cL^{(\ell)}_{t_0,t_0}J_{\ve'} \rangle\right| \\
& \leq C_\#\sum_{n,m \in \bN}\; \sum_{\gamma\in\cA, \bar i\in\cI}\;\sum_{k\in K_{\gamma,n,m}}   \int_{\bR^{d+1}}\hskip-.3cm   d\xi \, d\eta\,d\tau\,d\varsigma\;    e^{-z \{\tau+[H^{0,1}-H^{0,2}]\circ\Xi^{-1}(\xi,\eta,\tau,\varsigma)+(n+m)\delta\}}\\
&\quad \times \left[\bF_{\gamma,n,m,k, \bar{i}} (\xi,\eta,t',s')-\bF_{\gamma,n,m, k,\bar{i}} (0,0,0,s') \right]\cdot \left[ \kappa_{\ve}((\xi,\eta,t')) -  \kappa_{\ve'} ((\xi,\eta,t')) \right] \\
&   \leq C_{z,q} \ve^{\min\{q,1\}}  \|g\|_{\Cs^q} (a-\htop)^{-2}
\end{split} 
\]
where $\bF_{\gamma,n,m,k, \bar{i}}=\widehat F\circ \Xi^{-1}\cdot J\Xi^{-1}$. Note that, as in \eqref{eq:q-test},  $\|\bF_{\gamma,n,m, k,\bar{i}}\|_{\Cs^q}\leq C_q $.
This proves the Lemma for $\widetilde{\cB}^{0,q,\ell}_2$.
The extension to $\widetilde{\cB}^{p,q,\ell}_2$ is treated similarly after integrating by parts $p$ times.
\end{proof}

We can finally  present  a description of the trace that does not involve any limit (although, unfortunately, not yet for the operators we are interested in).

\begin{lem} \label{splitR1}
For each $z\in\bC$, $\Re(z)$ large enough, $n\in\bN$ and $s,t\in\bRp$, we have that
\begin{equation}
\delta_2^{(\ell)} \left( R_{2}^{(\ell)}(z)^n\cL_{t,s}^{(\ell)}\bar{\delta}_2^{(\ell)} (z) \right)= 
  \trf(R^{(\ell)}(z)^{2n}\cL^{(\ell)}_{t+s+2t_0}).
\end{equation}
\end{lem}
\begin{proof} 
Let $\mathfrak{R}_n(z,s,t)\doteq \delta_2^{(\ell)} \left(R_{2}^{(\ell)}(z)^n\cL_{t,s}^{(\ell)}\bar{\delta}_2^{(\ell)} \right)$. From Lemmata \ref{lem:delta2ext},  \ref{lem:delta-app} and equation \eqref{eq:d2c} we obtain 
\[ 
\begin{split}
&\mathfrak{R}_n (z,s,t)= \lim_{\ve \to 0} \sum_{\alpha, \bar i}\int _{M} \omega_M (x) \,
\psi_{\alpha}(x) \left(R_2^{(\ell)}(z)^n\cL_{t+t_0,s+t_0}^{(\ell)}  J_{\ve}\right)^{\alpha,\alpha}_{\bar i,\bar i}(x, x) \\
&=  \lim_{\ve \to 0} \sum_{\alpha, \beta , \bar i ,\bar j}\int_{\bRp^2}ds'dt'  \;\frac{(t's')^{n-1} e^{-z(t'+s')}}{[(n-1)!]^2}  \int_{M}  \omega_M (x)
  \psi_{\alpha}(x) \left[\psi_{\beta}J\Theta_\beta\right]\circ \phi_{s'+s+t_0}(x) \\
&\quad \quad \quad \times \kappa_{\ve}(\Theta_\beta(\phi_{-t'-t-t_0}(x))-\Theta_\beta(\phi_{s'+s+t_0}(x)))\\
&\quad \quad \quad \times(-1)^{\ell(d-\ell)}\langle \phi_{-t'-t-t_0}^*\omega_{\beta,\bar j},\omega_{\alpha, \bar i}\rangle_x \langle \omega_{\alpha, \bar i},*\phi_{s'+s+t_0}^**\omega_{\beta,\bar j}\rangle_x.
\end{split}
\]
Next, we  sum over $\bar i$, $\alpha$ and use \eqref{hodgedualform} to obtain
\[ 
\begin{split}
\mathfrak{R} (z,s,t)&=  \lim_{\ve \to 0} \sum_{\beta  ,\bar j}\int_{\bRp^2}ds'dt'  \; \frac{(t's')^{n-1} e^{-z(t'+s')}}{[(n-1)!]^2}   \int_{M} 
 \left[\psi_{\beta}J\Theta_\beta\right]\circ \phi_{s'+s+t_0}(x) \\
&\quad \quad \quad \times \kappa_{\ve}(\Theta_\beta(\phi_{s'+s+t_0}(x))-\Theta_\beta(\phi_{-t'-t-t_0}(x)))\\
&\quad \quad \quad \times J\phi_{s'+s+t_0}\langle \phi_{-t'-t-s'-s-2t_0}^*\omega_{\beta,\bar j},\omega_{\beta,\bar j}\rangle\circ \phi_{s'+s+t_0}(x)\\
&= \lim_{\ve \to 0}\sum_{\beta  ,\bar j}\int_{\bRp^2}ds'dt'  \; \frac{(t's')^{n-1} e^{-z(t'+s')}}{[(n-1)!]^2} \int_{M}   \langle \omega_{\beta,\bar j}, \cL_{t'+s'+t+s+2t_0}^{(\ell)}j_{\ve, \beta,\bar j, x}\rangle,
\end{split}
\]
where, in the last line, we have changed variables and used \eqref{eq:normalization1}. Next, after the change of variables $v=t'+s'$, $u=t'$, we integrate in $u$ (the integral is given by the $\beta$-function)  and recall \eqref{traceoriginal} to obtain the statement of the Lemma.
\end{proof}
At last we can start harvesting the benefits of the previous results. For fixed $p,q>0$, $p+q<r-1$, by Lemma \ref{lem:quasicompact2} we have
$R^{(\ell)}_{2}(z) =  {P}_2^{(\ell)}(z)+ {U}_2^{(\ell)}(z)$ where ${P}_2^{(\ell)}(z)$ is a finite
 rank operator and the spectral radius of ${U}_2^{(\ell)}(z)$ is bounded by  $(\Re(z)-\sigma_{\ell}+\frac{\sigma_{p,q}}2)^{-2}$.  Recall that by Lemma \ref{lem:quasicompactness} we have $R^{(\ell)}(z) =  P^{(\ell)}(z)+ U^{(\ell)}(z)$ where $P^{(\ell)}(z)$ is a finite rank operator and the spectral radius of $U^{(\ell)}(z)$ is bounded by  $(\Re(z)-\sigma_{\ell}+\sigma_{p,q})^{-1}$. In addition,
\[
P^{\left(\ell\right)}(z) U^{\left(\ell\right)}(z) = 
U^{\left(\ell\right)} (z)P^{\left(\ell\right)} (z)= 0\;;\quad
{P}_2^{\left(\ell\right)}(z) {U}_2^{\left(\ell\right)}(z) = 
{U}_2^{\left(\ell\right)} (z){P}_2^{\left(\ell\right)} (z)= 0. 
\]

\begin{lem} \label{splitR2} 
There exists $\Cnu>0$ such that, for each $t,s\in\bRp$,  $n\in\bN$ and $z\in \bC$, $a=\Re(z)\geq \Cnu$, we have that
\begin{equation}
\delta_2^{(\ell)} 
\left( P_2^{(\ell)}(z)^n \cL_{t,s}^{(\ell)} \bar{\delta}_2^{(\ell)}(z) \right) = 
\tr \left({P}^{(\ell)}(z)^{2n}\cL^{(\ell)}_{t+s+2\vuo}\right).
\end{equation}
\end{lem}
\begin{proof}
Recalling  \eqref{eq:wedgetransfer} and \eqref{hodgedualform} we can write, for $\Re(z)$ sufficiently large,
\begin{equation}  \label{lastchange1}
\begin{split}
&\langle R^{(\ell)}_2(z)^{n} \cL^{(\ell)}_{t+\vuo,s+\vuo}J_\ve ,  \pi_1^*(f) \wedge \pi_2^*(g)  \rangle_{\Omega^{(2\ell)}} \\
&=  \int_{\bRp^2\times M^2}\hskip-1cm ds'\,dt'\; \frac{(t's')^{n - 1}}{(n-1)!^2e^{z(t'+s')}}  \langle   J_\ve  , \pi_1^*(* \phi_{t'+t+\vuo}^* \! * \! f) \wedge \pi_2^*(\phi_{-s-s'-\vuo}^*g)\rangle .
\end{split}
\end{equation}
By taking the limit $\ve\to 0$, since the integrals are uniformly convergent with respect to time (as in equation \eqref{eq:trace1}, by a rough bound on the growth of the number of orbits), using Lemma \ref{lem:leftje} we obtain
\begin{equation}\label{eq:double}
\begin{split}
&R^{(\ell)}_2(z)^{n}\cL^{(\ell)}_{t,s}\bar\delta_2^{\ell}(\pi_1^*(f) \wedge \pi_2^*(g))=\int_{\bRp^2\times M}\hskip-.8cm  \frac{(t's')^{n - 1}\langle *\phi_{t'+t+\vuo}^**f , \phi_{-s'-s-\vuo}^*g\rangle}{(n-1)!^2e^{z(t'+s')}}  \\
&=(-1)^{\ell(d-\ell)}\int_{\bRp^2}ds\,dt\; \frac{(t's')^{n - 1}e^{-z(t'+s')}}{(n-1)!^2} \int_{M} \langle   f , \phi_{-s-t-s'-t'-2\vuo}^*g\rangle\\
&=(-1)^{\ell(d-\ell)}\int_{\bRp}d\tau\int_0^1 d\eta \frac{\tau^{2n-1}\eta^{n-1}(1-\eta)^{n - 1}}{(n-1)!^2e^{z\tau}} \int_{M} \langle   f , \phi_{-\tau-t-s-2\vuo}^*g\rangle\\
&=(-1)^{\ell(d-\ell)}\int_{\bRp\times M}\hskip-.4cm d\tau\; \frac{\tau^{2n-1} \langle   f , \phi_{-\tau-t-s-2\vuo}^*g\rangle}{(2n-1)!e^{z\tau}} \\
&=(-1)^{\ell(d-\ell)} \langle   f , R(z)^{2n}\cL^{(\ell)}_{t+s+2\vuo}g\rangle_{\Omega^\ell},
\end{split}
\end{equation}
since the $d\eta$ integral in the third line is given by the $\beta$-function.
By the von Neumann expansion, for $\Re(z)$ and $\Re(\xi)$ large enough, we have
\[
(\xi\Id-R^{(\ell)}_2(z))^{-1}\cL_{t,s}^{(\ell)}\bar\delta_2^{\ell}(\pi_1^*(f) \wedge \pi_2^*(g))=(-1)^{\ell(d-\ell)}\langle f, (\xi\Id-(R(z)^{(\ell)})^2)^{-1}\cL^{(\ell)}_{t+s+2\vuo}g\rangle_{\Omega^\ell}.
\]
Since both expression are meromorphic in the region $\{ |\xi| >(a-\sigma_\ell+\frac{\sigma_{p,q}}2)^{-2} \}$, it follows that they must agree on such a region. Given a curve $\gamma$ surrounding the region $\{|\xi | \leq (a-\sigma_\ell+\frac{\sigma_{p,q}}2)^{-2} \}$, we can use standard functional analytic calculus (e.g., \cite{Kato}) and recall Lemma \ref{lem:immersion2} to obtain, for $h=\pi^*_1f\wedge\pi^*_2g$,
\[
\begin{split}
&\left\{\left[ {U}^{(\ell)}_2(z)\right]^m \cL^{(\ell)}_{t,s}\bar\delta_2^{\ell}\right\}( h )=   
\frac{1}{2\pi i} \int_{\gamma}\xi^m\left(\xi\mathds{1} - R_2^{(\ell)}(z)  \right)^{-1}\cL^{(\ell)}_{t,s}\bar\delta_2^{\ell}( h) d\xi \\
&= \frac{(-1)^{\ell(d-\ell)}}{2\pi i} \int_{\gamma}\xi^m  \langle f, \left( \xi \mathds{1} - R^{(\ell)}(z)^{2}  \right)^{-1}\cL^{(\ell)}_{t+s+2\vuo}g 
\rangle_{\Omega^\ell}    d\xi  \\
& = \langle f,{U}^{(\ell)}(z)^{2m}\cL^{(\ell)}_{t+s+2\vuo} g \rangle_{\Omega^\ell}.
\end{split} 
\]
Hence, we have
\begin{equation}\label{eq:proj-eq}
\left\{ \left[ {P}^{(\ell)}_2(z)\right]^{m}\cL^{(\ell)}_{t,s} \bar\delta_2^{\ell}\right\}( h )= \langle f,{P}^{(\ell)}(z)^{2m}\cL^{(\ell)}_{t+s+2\vuo} g \rangle_{\Omega^\ell}.
 \end{equation}
 Since ${P}^{(\ell)}(z)^{2m}$ is a finite rank operator, it follows that ${P}^{(\ell)}(z)^{2m}\cL^{(\ell)}_{t+s+2\vuo}=\cL^{(\ell)}_{\vuo}{P}^{(\ell)}(z)^{2m}\cL^{(\ell)}_{t+s+\vuo}$ is also finite rank on $\cB_{p,q,\ell}$. Thus  there exist $u_{m,k}\in\cB_{-,p,q,\ell}', v_{m,k}\in \cB_{+,p,q,\ell}$, $k\in\{1,\dots, L\}$ (see Lemma \ref{lem:LY}), such that
 \[
 \langle f,{P}^{(\ell)}(z)^{2m}\cL^{(\ell)}_{t+s+2\vuo} g \rangle_{\Omega^\ell}=\sum_{k=1}^L\langle f, v_{m,k} \rangle_{\Omega^\ell}u_{m,k}(g).
 \]
 Next, we define the mollification of an element $h\in\Omega^\ell_{2,s}$ by
 \begin{equation}\label{eq:molli2-def}
 \begin{split}
 \bM_{2,\ve}&(h)(x,y)=\sum_{\alpha,\beta}\sum_{\bar i,\bar j}\psi_{\alpha}(x,y)\psi_{\beta}(x,y) \omega_{\alpha,\bar i}(x)\wedge \omega_{\beta,\bar j}(y)\\
&\times  \ve^{-2d}\int_{\bR^{2d}} \kappa_\ve(\Theta_\alpha(x)-\xi) \kappa_\ve(\Theta_\beta(y)-\eta)
h^{\alpha,\beta}_{\bar i,\bar j}(\Theta_\alpha^{-1}(\xi),\Theta_\beta^{-1}(\eta)).
\end{split}
 \end{equation}
By duality, we can define the mollificator $\bM_{2,\ve}'$ on the currents.
Following the same reasoning as in  Lemma \ref{lem:molli} we have that $\bM_{2,\ve}'$, restricted to $\cB^{p,q,\ell}_2$, is a bounded operator which converges, in the sense of  Lemma \ref{lem:molli},  to the identity when $\ve\to 0$.

A direct computation shows that $\bM_{2,\ve}(\pi_1^*(f)\wedge\pi_2^*(g))=\pi_1^*(\bM_\ve f)\wedge\pi_2^*(\bM_\ve g)$, hence \eqref{eq:proj-eq} implies
\[
\bM_{2,\ve}'\left\{\left[ {P}^{(\ell)}_2(z)\right]^{m}\cL^{(\ell)}_{t,s} \bar\delta_2^{\ell}\right\}( h )= \langle \bM_\ve f,{P}^{(\ell)}(z)^{2m}\cL^{(\ell)}_{t+s+2\vuo} \bM_\ve g \rangle_{\Omega^\ell}.
\]
Note that the  $\bM_{2,\ve}'\left\{\left[ {P}^{(\ell)}_2(z)\right]^{m}\cL^{(\ell)}_{t,s} \bar\delta_2^{\ell}\right\}\in\cE^{2\ell}_s$ for $s\leq p+q$. Thus its value on the product sections  determines it uniquely as a $\cE^{2\ell}_{p+q}$ current (see footnote \ref{foo:little-holder}). Accordingly, since 
$\bM_{2,\ve}' [ {P}^{(\ell)}_2(z) ]^{m}\cL^{(\ell)}_{t,s} \bar\delta_2^{\ell}$ and $\sum_{k+1}^L \pi_1^*(\bM_\ve' v_{m,k})\wedge\pi_2^*(\bM_\ve' u_{m,k})$
agree on each element of the type $\pi_1^*(f) \wedge \pi_2^*(g)$, they agree as currents. Thus, by Lemma \ref{lem:immersion2}, they are the same element of $\cB^{p,q,\ell}_{2}$. Accordingly, by Lemmata \ref{lem:molli}, \ref{lem:delta2ext} and equation \eqref{eq:dscalar}, we finally obtain that
\[ 
\begin{split}
\delta_2\left({P}^{(\ell)}_2(z)^{m}\cL^{(\ell)}_{t,s} \bar\delta_2^{\ell}\right)&=\lim_{\ve\to 0}\delta_2\left(\bM_{2,\ve}'{P}^{(\ell)}_2(z)^{m}\cL^{(\ell)}_{t,s} \bar\delta_2^{\ell}\right)\\
&=\lim_{\ve\to 0}\sum_{k=1}^L\delta_2\left( \pi_1^*(\bM_\ve' v_{m,k})\wedge\pi_2^*(\bM_\ve' u_{m,k})\right)\\
&=\lim_{\ve\to 0}\sum_{k=1}^L\langle \bM_\ve' v_{m,k},\bM_\ve' u_{m,k}\rangle=\sum_{k=1}^L u_{m,k}(v_{m,k}).\hskip2cm \qedhere
\end{split} 
\]
\end{proof}

Thus we have reduced the computation of certain ``flat trace" to that of a trace of a matrix. This suffices to obtain the same type of result for the operators we are interested in, at last.
\begin{lem}\label{lem:tensor-trace}
For each  $p+q<r-1$, $z=a+ib\in \bC$ with $a>\Cnu$,  $\mu>(a- \sigma_{\ell}+\frac{\sigma_{p,q}}2)^{-1}$,  we have,  for each $n\in\bN$, that 
\begin{equation}
\begin{split}
\left|\trf(R^{(\ell)}(z)^n)- \operatorname{tr}(P^{(\ell)}(z)^n) \right|\leq C_{p,q,z,\mu}\mu^n \,.
\end{split}
\end{equation}
\end{lem}
\begin{proof}
Lemma \ref{splitR1}  yields, for all $n,m, k\in\bN$, $|m-k|\leq C_\#$, the following ugly, but useful, formula
\begin{equation}\label{eq:hopefully}
\begin{split}
&\int_\vuo^\infty\hskip-.3cm  dt\int_\vuo^\infty \hskip-.3cm ds \;\frac{s^{m-1}t^{k-1}e^{-z(t+s)}}{(m-1)!(k-1)!} \delta_2^{(\ell)} \left( R_{2}^{(\ell)}(z)^{n}\cL^{(\ell)}_{t-\vuo,s-\vuo}\bar{\delta}_2^{(\ell)}  \right)\\
&=\int_{2\vuo}^\infty dt\int_{\vuo}^{t-\vuo}\hskip-.4cm ds \;\frac{s^{m-1}(t-s)^{k-1}e^{-zt}}{(m-1)!(k-1)!}\trf(R^{(\ell)}(z)^{2n}\cL^{(\ell)}_{t})\\
&=\trf(R^{(\ell)}(z)^{2n+m+k})-\int_0^{2\vuo}\hskip-.3cm dt\;\frac{t^{m+k-1}}{(m+k-1)!} e^{-z t}\trf(R^{(\ell)}(z)^{2n}\cL^{(\ell)}_{t})\\
&\quad -2\int_{2\vuo}^\infty dt\int_{0}^{\vuo}\hskip-.4cm ds \;\frac{s^{m-1}(t-s)^{k-1}e^{-zt}}{(m-1)!(k-1)!}\trf(R^{(\ell)}(z)^{2n}\cL^{(\ell)}_{t}).
\end{split}
\end{equation}
On the other hand by the spectral decomposition of $R_2$, as in Lemma \ref{lem:quasicompact2}, and Lemma \ref{splitR2} it  follows that, for each $\mu>\mu_1>(a-\sigma_{\ell}+\frac{\sigma_{p,q}}2)^{-1}$,
\[
\begin{split}
&\int_\vuo^\infty\hskip-.3cm  dt\int_\vuo^\infty \hskip-.3cm ds \;\frac{s^{m-1}t^{k-1}e^{-z(t+s)}}{(m-1)!(k-1)!} \delta_2^{(\ell)} \left( R_{2}^{(\ell)}(z)^{n}\cL^{(\ell)}_{t-\vuo,s-\vuo}\bar{\delta}_2^{(\ell)}  \right)\\
&= \int_\vuo^\infty\hskip-.3cm  dt\int_\vuo^\infty \hskip-.3cm ds \;\frac{s^{m-1}t^{k-1}e^{-z(t+s)}}{(m-1)!(k-1)!} \delta_2^{(\ell)} \left( P_{2}^{(\ell)}(z)^{n}\cL^{(\ell)}_{t-\vuo,s-\vuo}\bar{\delta}_2^{(\ell)}  \right)+\cO\left(\frac{C_{p,q,z,\mu_1}\mu_1^{2n}}{(a-\sigma_{\ell})^{m+k}}\right)\\
&= \int_\vuo^\infty\hskip-.3cm  dt\int_\vuo^\infty \hskip-.3cm ds \;\frac{s^{m-1}t^{k-1}e^{-z(t+s)}}{(m-1)!(k-1)!} \tr \left( P^{(\ell)}(z)^{2n}\cL^{(\ell)}_{t+s} \right)+\cO\left(\frac{C_{p,q,z,\mu_1}\mu_1^{2n}}{(a-\sigma_{\ell})^{m+k}}\right)
\end{split}
\]
Finally, recall again that the number of periodic orbits of length $t$ grows at most exponentially. Hence, by Lemma \ref{lem:traceorbits},  there exists $A\geq \sigma_\ell$ such that $|\trf(R^{(\ell)}(z)^{2n}\cL^{(\ell)}_{t})|\leq C_\# (a-A)^{-2n} e^{At}$ while $|\tr \left( P^{(\ell)}(z)^{2n}\cL^{(\ell)}_{t+s} \right)|\leq C_\# (a-\sigma_\ell)^{-2n} e^{\sigma_\ell t}$. We can use such estimates to evaluate the integrals in \eqref{eq:hopefully} and obtain,
\[
\begin{split}
\left|\trf(R^{(\ell)}(z)^{2n+m+k})-\tr \left({P}^{(\ell)}(z)^{2n+m+k}\right)\right|\leq& \frac{C_{p,q,z,\mu_1}\mu_1^{2n}}{(a-\sigma_{\ell})^{m+k}}
+\frac{C_\#\vuo^m}{(a-A)^{2n+k}m!}\\
&+\frac{C_\#(2\vuo)^{m+k}}{(a-A)^{2n}(m+k)!}.
\end{split}
\]
To conclude we choose $m=\varsigma n$, $|m-k|\leq 1$. Note that the last two terms are smaller than $\mu^{2n+m+k}$ provided $n\geq \frac{C_\#e\vuo}{\varsigma(a-A)^{1+\frac 2\varsigma}\mu^{2+\frac 2\varsigma}}$, while the first term on the left is bounded by $C_\#\mu^{2n+m+k}$ provided we choose $\varsigma= \frac{\ln \mu\mu_1^{-1}}{\mu^{-1}(a-\sigma_{\ell}^{-1})}$.
\end{proof}

\section{Contact flows}\label{sec:contact}
The results of the previous sections suffice to prove that the $\zeta_{\text{Ruelle}}$ is meromorphic, yet they provide very little information on the location of its zeroes and poles. Such a knowledge is fundamental to extract information from the $\zeta$ functions (e.g. counting results or statistical properties of the flow). In section \ref{sec:zeta-growth} we provide an approach to gain such information, partially inspired by \cite{Liverani05}, based on a Dolgopyat type estimate on the norm of the resolvent. We are not aware of a general approach to gain such estimates apart for the case of $\Cs^1$ or Lipschitz foliations \cite{Dolgopyat98a, Stoyanov11a, Stoyanov11b} and the case of {\em contact flows} (\cite{Liverani95, Tsujii10}). In the following we will restrict to the latter since it covers the geometrically relevant case of geodesic flows in negative curvature. Our approach follows roughly  \cite[Section 6]{Liverani95} but employs several simplifying ideas, some from \cite{BaladiLiverani11}, some new.

In the case of contact flows $d$ must be odd, $d_u=d_s$ and we can, and will, assume that the contact form $\overline\alpha$ in coordinates reads (see \cite[section 3.2, Appendix A]{BaladiLiverani11} for details)
\[
(\Theta_\beta^{-1})^*\overline\alpha=dx_d-\langle x^s,dx^u\rangle\doteq\alpha_0
\]
where $(x^s,x^u,x_d)\in \bR^{(d_s)}\times\bR^{(d_s)}\times \bR=\bR^d$ is a point in the chart.

\begin{rmk} The extra information that we need, and that can be gained in the contact flow case, are bounds on the size of the resolvent $R^{(d_s)}(z)$ for $\Im(z)$ large. See Lemma \ref{lem:ruellebound} for an explanation of why $\ell\neq d_s$ need not be studied.
Since  in the rest  of this section we discuss only the case $\ell=d_s$ we will often drop the scripts $d_s$ in the relevant objects.
\end{rmk}

\subsection{Dolgopyat's estimate}\label{sec:dolgo}
Let $a=\Re(z)>\sigma_{d_s}$ be fixed once and for all.
For a fixed $\Cnc\in (0, e^{-4})$, define $c_a\doteq \Cnc(a-\sigma_{d_s})^{-1}e^{-1}$ and
\[
{\widehat R}_n(z)h=\int_{c_a n}^\infty dt\; e^{-zt}\frac{t^{n-1}}{(n-1)!}\cL^{(d_s)}_{t} h.
\]
Then, bounding $e^{-\sigma_{d_s}t}\spqnorm[0,q,d_s]{\cL_t}$ by one, integrating and using Stirling formula we have, for each $q>0$, 
\begin{equation}\label{eq:do-2}
\spqnorm[0,q,d_s]{(R^{(d_s)}(z)^n-{\widehat R}_n(z))h}\leq C_\# (a-\sigma_{d_s})^{-n} \Cnc^{n}\spqnorm[0,q,d_s]{h}.
\end{equation}
Moreover, for $c_a n>\vuo$, by Lemma \ref{lem:LY} it follows that
\begin{equation}\label{eq:strongtoweak}
\spqnorm[0,q,d_s]{{\widehat R}_n(z)h}\leq C_\#(a-\sigma_{d_s})^{-n}\pqnorm[0,q,d_s]{h}.
\end{equation}
Thus, it is natural to start estimating the latter norm.

In the following arguments it turns out to be convenient to introduce norms similar to the one used in \cite{Liverani95}. We did not use them in the previous sections since they do not allow to keep track of the higher regularity of the flow.

\begin{defn}\label{def:stable-norm} We define norms $\pqnorm[q]{\cdot}^s$ as with the norms $\pqnorm[0,q,d_s]{\cdot}$ with the only difference that the set $\Sigma$ is replaced by the set $\Sigma^s$, defined as the elements of $\Sigma$ which are subsets of a strong stable manifold. Also,  let $\cV^u$ be the set of continuous vector fields tangent to the strong unstable direction, of $\Cs^0$-norm one, and $\Cs^1$ when restricted to any unstable manifold.\footnote{Note that such vector fields determine a unique flow on each unstable manifold, hence a unique global flow, even if they are only continuous on the manifold. Thus the corresponding Lie derivative is well defined.}  We then define $\pqnorm[]{h}^u\doteq\sup_{v\in \cV^u}\pqnorm[0,0,d_s]{L_v h}$. To conclude we define the norm $\pqnorm[q]{\cdot}^*\doteq \pqnorm[q]{\cdot}^s+\pqnorm[]{\cdot}^u$.
\end{defn}
Note that the unstable foliation can be trivialized by a  H\"older continuous change of variables that is $C^1$ when restricted to any leaf.
In such coordinates one can easily verify that \eqref{eq:extensionproperty} holds also in the present case by arguing as in Lemma \ref{lem:extension}.

Before starting the real work let us fix some notation and recall some facts.
Let $\varpi$ be the H\"older regularity of the strong foliations and $\varpi^*$ the  H\"older regularity of the Jacobian of the associated holonomy. Then, in the contact flow case, $\varpi\geq \frac {2\olambda}{\lambda_+}$ and $\varpi^*\geq \frac {\olambda}{\lambda_+}$ (see Appendix \ref{app:holo-est} for more details). 

For a manifold $W\in \Sigma$ let $\widetilde W=\cup_{|t|\leq \delta}\phi_tW$, $\widetilde W_+=\cup_{|t|\leq 2\delta}\phi_tW_+$. For each two sufficiently nearby manifolds $W,W'\in \Sigma$, let $H_{W,W'}:\widetilde W_+\to \widetilde W'_+$ be the strong unstable holonomy, then $\|J H_{W, W'}\|_{\Cs^{\varpi^*}}+\| H_{W,W'}\|_{\Cs^\varpi}\leq C_\#$. We set $\hat \varpi =\frac {2\olambda}{\lambda_+}$, $\varpi'=\min\{1,\hat \varpi\}$
and $\varpi_*=\min\{\olambda/\lambda_+, \hat\varpi^2\}$.

\begin{lem}\label{sublem:holo} For each $q\geq 1$, $\alpha\in\cA$, $W,W'\in\Sigma_\alpha$ such that $H_{W,W'}(\widetilde W)\subset\widetilde W_+'$, $g\in \Gamma^{d_s, q}_c(\widetilde W')$, there exists $\bar g\in \Gamma^{d_s,\varpi_*}_c(\widetilde W'_+)$, $\|\bar g\|_{\Gamma^{d_s,\varpi_*}_c(\widetilde W')}\leq C_\#\|g\|_{\Gamma^{d_s,q}_c(\widetilde W)}$, such that for all $h\in\Omega_r^{d_s}$ we have
\[
\left|\int_{\widetilde W}\langle g, h\rangle-\int_{\widetilde W_+'}\langle \bar g,h\rangle\right|\leq C_\# d(W,W') \pqnorm[]{h}^u\|g\|_{\Gamma^{d_s,q}_c(\widetilde W)}.
\]
\end{lem}
\begin{proof}
Let $\{W^*_\tau\}_{\tau\in [0,1]}$ be a smooth foliation interpolating between $\widetilde W$ and $\widetilde W'$. Also let $H_\tau$ be the strong unstable holonomy from $\widetilde W$ to $W^*_\tau$. Note that, by performing the interpolation in the chart $\Theta_\alpha$, one can ensure that there exists $W_\tau\in \Sigma_\alpha$ such that $W^*_\tau\subset\widetilde W_{\tau,+}$. Next, let $\vf_\tau$ be a one parameter family of forms, then
\[
\int_{W^*_\tau}\langle \vf_\tau, h\rangle=\int_{\widetilde W}\langle \vf_\tau, h\rangle\circ H_\tau JH_\tau=\sum_{\bar i}\int_{\widetilde W}\overline\vf_{\tau,\bar i}  h_{\bar i}\circ H_\tau.
\]
where $\overline\vf_{\tau,\bar i}=\vf_{\tau,\bar i}\circ H_\tau JH_\tau$. By the implicit function theorem it follows that $\partial_\tau h_{\bar i}\circ H_\tau=(L_v h_{\bar i})\circ H_\tau$ for some $v\in\cV^u$, $\|v\|_\infty\leq C_\# d(W,W')$. In addition,
\[
(L_v h)_{\bar i}=L_v h_{\bar j}+\sum_{\bar j} \langle \omega_{\alpha,\bar i},L_v\omega_{\alpha,\bar j}\rangle h_{\bar j}.
\]
Let $A_{\bar i,\bar j}(\tau)=\langle \omega_{\alpha,\bar i},L_v\omega_{\alpha,\bar j}\rangle\circ H_\tau$ and note that $A$ is $\Cs^r$ in $\tau$.
\[
\partial_\tau\int_{W^*_\tau}\langle \vf_\tau, h\rangle=\sum_{\bar i}\int_{\widetilde W}\left\{\partial_\tau \overline\vf_{\tau,\bar i} -\sum_{\bar j} A_{\bar j,\bar i}\overline \vf_{\tau,\bar j}\right\} h_{\bar i}\circ H_\tau+\overline\vf_{\tau,\bar i}(L_vh)_{\bar i}\circ  H_\tau.
\]
We can then choose $\overline\vf_{\tau,\bar i}$ as the solution of the ordinary differential equation (the variables on $\widetilde W$ are treated as parameters)
\[
\begin{split}
& \frac{d}{d\tau} \overline\vf_{\tau,\bar i} =\sum_{\bar j} A_{\bar j,\bar i}(\tau)\overline \vf_{\tau,\bar j}\\
& \overline\vf_{0,\bar i}=g_{\bar i}.
\end{split}
\]
By the dependence on the parameters of the solution of an ODE we have, for each $\tau\in[0,1]$, $\|\overline \vf_{\tau,\bar i}\|_{\Cs^{\varpi'}}\leq C_\#\|g\|_{\Gamma^{q}_c}$. Hence, $\|\vf_{\tau}\|_{\Gamma^{\varpi_*}_c(W^*_\tau)}\leq C_\#\|g\|_{\Gamma^{q}_c(\widetilde W)}$.
\begin{equation}\label{eq:sup-flow-est}
\begin{split}
&\left|\int_{\widetilde W'}\langle \vf_1, h\rangle-\int_{\widetilde W}\langle g,h\rangle\right|\leq\int_0^1d\tau \left|\sum_{\bar i}\int_{\widetilde W}\overline\vf_{\tau,\bar i}(L_vh)_{\bar i}\circ  H_\tau\right|.\\
&= \int_0^1d\tau \left|\int_{W^*_\tau}\langle \vf_\tau, L_v h\rangle\right|\leq C_\# d(W,W') \pqnorm[]{h}^u \|g\|_{\Gamma^{d_s,q}_c(\widetilde W)}.
\end{split}
\end{equation}
\end{proof}

\begin{lem}\label{lem:norm-equiv}
For each $\eta\in (0,\varpi_*)$, $t\in(\vuo,\infty)$ and $h\in\Omega_r^{d_s}$ we have
\[
\begin{split}
&\pqnorm[0,\eta,d_s]{h}\geq \pqnorm[\eta]{h}^s\,,\\
&\pqnorm[0,1+\eta,d_s]{\widehat R_n(z)h}\leq \frac {C_\#}{a-\sigma_{d_s}} \pqnorm[\eta]{h}^*.
\end{split}
\]
\end{lem}
\begin{proof}
The first inequality is obvious since the sup in the definition of the norm is taken on a larger set. To prove the second note that
\[
\int_{W_{\alpha,G}}\langle g,\widehat R_n(z)h\rangle=\int_{c_an}^\infty dt\int_{W_{\alpha,G}} \frac{t^{n-1}e^{-zt}}{(n-1)!}\langle g,\cL_t^{d_s}h\rangle
\]
which, by introducing a partition of unity in time as in \eqref{eq:time-partitioning} and changing variables as in \eqref{eq:mani-split}, can be seen as an integral on $d_s+1$ dimensional manifolds close to the weak stable foliation. If we consider the integral on a strong stable leaf close to $W_{\alpha,G}$ we can split it similarly and apply Lemma \ref{sublem:holo} to compare the integral on each manifold.
The Lemma readily follows.
\end{proof}

We can now state the main estimate of this section. As the proof is a bit involved we postpone it at the end of the section so as not to break the flow of the argument.

\begin{prop}[Dolgopyat type estimate]\label{lem:dolgo-est0} Let $\phi_t$ be a contact flow such that $\varpi'>\frac 23$. Then for each $\eta\in (0,\varpi_*)$,  there exist constants $\Cnd, a_0,b_0\geq 1$, $\lambda  \Cnd\Cnc>2 ea_0$, and $\gamma_0\in (0,1)$ such that, for each $h\in\Omega^{d_s}_{0,1}(M)$, $z=a+ib$ with $2a_0+\sigma_{d_s}\geq a\geq a_0+\sigma_{d_s}$, $|b|\geq b_0$ and $n\geq \Cnd \ln |b|$, we have
\[
\begin{split}
&\pqnorm[\eta]{{\widehat R}_n(z) h}^*\leq \frac{C_\eta}{(a-\sigma_{d_s})^{n}}\pqnorm[\eta]{h}^*,\\
&\pqnorm[\eta]{{\widehat R}_n(z)^3 h}^*\leq \frac{C_\eta}{(a-\sigma_{d_s})^{3n}} |b|^{-3\gamma_0}\pqnorm[\eta]{h}^* .
\end{split}
\]
\end{prop}
\begin{rmk}
The bunching condition in the statement of Proposition \ref{lem:dolgo-est0} arises as follows: in the Dolgopyat estimate one reduces the problem to comparing certain integrals over a multitude of nearby manifolds. They are usually broken into two sets: the one that are sufficiently far apart to allow for cancellations due to the non integrability of the foliations (this is the basic Dolgopyat cancellation mechanism) and the ones that are too close for the cancellation mechanism to be effective. The latter are shown to be too few to contribute to the total estimate. In Lemma \ref{sublem:doest-3} we give explicit meaning to ``sufficiently far apart", we believe this to be optimal. The remaining ``too close" manifolds are handled in Lemma \ref{sublem:doest-2}. The estimate is reduced to the growth rate of certain unstable disks that we can handle only if ``too close" is small enough. This is due to the possibility for a disk to grow in an extremely elongated ellipses because of different expansion rates in the 
unstable directions. The two different meaning of ``too close" are consistent only under the above mentioned bunching conditions. This is the only place where such a condition is used.  It might be possible to improve Lemma \ref{sublem:doest-2}, and hence remove the bunching condition, by cleverly using some distortion estimate but at the moment we do not see how.
\end{rmk}
\begin{rmk}
The choice of stating the Lemma in terms of the norm $\pqnorm[\eta]{\cdot}^*$ is a bit arbitrary but very convenient in the present context.  The estimate is largely norm-independent as better shown in \cite{BaladiLiverani11}.
\end{rmk}
The goal of this section easily follows.

\begin{lem}\label{lem:dolgo-est} If $\phi_t$ is a contact Anosov flow with $\varpi'>\frac 23$, then for each $\eta\in(0,\varpi_*)$, there exist constants $\Cnd, a_0,b_0\geq 1$, $\lambda  \Cnd\Cnc>2 ea_0$, such that  for each  $h\in\cB^{1,\eta,d_s}$, $z=a+ib$ with $2a_0+\sigma_{d_s}> a\geq a_0+\sigma_{d_s}$, $|b|\geq b_0$, and $n\geq C_\#\Cnd\ln |b|$, holds true
\[
\spqnorm[1,\eta,d_s]{R(z)^{n} h}\leq C_{\eta}(a-\sigma_{d_s})^{-n} b^{-\gamma_0}\spqnorm[1,\eta,d_s]{h}.
\]
\end{lem}
\begin{proof}
By Lemma \ref{lem:quasicompactness} we have, for $m=\lceil \Cnd\ln |b|\rceil$ where $\Cnd$ is given by Proposition \ref{lem:dolgo-est0},
\[
\begin{split}
\spqnorm[1,\eta,d_s]{ R(z)^{(k+2)m} h}&\leq \frac{C_\#}{(a-\sigma_{d_s})^{(k+1)m}(a-\sigma_{d_s}+\lambda)^{m}}\spqnorm[1,\eta,d_s]{h}\\
&\quad+\frac{C_\#|b|}{(a-\sigma_{d_s})^{m}}\spqnorm[0,1+\eta, d_s]{ R(z)^{(k+1)m}h}.
\end{split}
\]
The first term is already small enough.
By equations \eqref{eq:do-2}, \eqref{eq:strongtoweak}
\[
\spqnorm[0,1+\eta, d_s]{ R(z)^{(k+1)m}h}\leq \frac{C_\# k\Cnc^m \spqnorm[0,1+\eta, d_s]{h}}{ (a-\sigma_{d_s})^{(k+1)m}}+\frac{C_\#}{a-\sigma_{d_s}}\pqnorm[0,1+\eta, d_s]{{\widehat R}_m(z)^{k}h}.
\]
The first term is small enough provided $\gamma_0\leq\Cnd\ln \Cnc^{-1}-1-\frac{\ln k}{\ln b_0}$. 
Accordingly, to establish the Lemma it suffices to prove, for some $k\in\bN$,
\begin{equation}\label{eq:enough-do}
\pqnorm[0,1+\eta,d_s]{{\widehat R}_m(z)^{k} h}\leq C_{\eta}(a-\sigma_{d_s})^{-km} b^{-1-\gamma_0}\pqnorm[1,\eta,d_s]{h}.
\end{equation}
Note that $\pqnorm[+,1,\eta,d_s]{{\widehat R}_m(z) h}\leq \frac {C_\eta}{a-\sigma_{d_s}}\pqnorm[-,1,\eta,d_s]{h}$. Hence,
by Lemmata  \ref{lem:molli}, \ref{lem:norm-equiv} and Proposition \ref{lem:dolgo-est0} we have
\[
\begin{split}
&\pqnorm[0,1+\eta, d_s]{{\widehat R}_m(z)^{k}h}\leq \frac{C_\#\ve \pqnorm[1,\eta,d_s]{h}}{(a-\sigma_{d_s})^{kn}}+C_\#\pqnorm[0,1+\eta, d_s]{{\widehat R}_m(z)^{k-1}\bM_\ve {\widehat R}_m(z)h}\\
&\leq\frac{C_\#\ve \pqnorm[1,\eta,d_s]{h}}{(a-\sigma_{d_s})^{km}}+\frac{C_\#}{a-\sigma_{d_s}}\pqnorm[\eta]{{\widehat R}_m(z)^{k-2}\bM_\ve{\widehat R}_n(z) h}^*\\
&\leq  \frac{C_\#\ve \pqnorm[1,\eta,d_s]{h}}{(a-\sigma_{d_s})^{km}}+C_\#(a-\sigma_{d_s})^{-(k-1)m}|b|^{-(k-1)\gamma_0}\pqnorm[\eta]{\bM_\ve {\widehat R}_m(z) h}^*\\
&\leq  \frac{C_\#\ve \pqnorm[1,\eta,d_s]{h}}{(a-\sigma_{d_s})^{kn}}+C_\#(a-\sigma_{d_s})^{-km}|b|^{-(k-1)\gamma_0}\ve^{-1-\eta-d_s}\pqnorm[1,\eta,d_s]{h},
\end{split}
\]
which, after choosing $\ve=|b|^{-1-\gamma_0}$ and $k=C_\#\gamma_0^{-1}$, proves equation \eqref{eq:enough-do} and hence the Lemma. Indeed, the condition $\gamma_0\leq\Cnd\ln \Cnc^{-1}-1-\frac{\ln k}{\ln b_0}$ can be satisfied by choosing $b_0$ large enough.
\end{proof}

We conclude with the missing proof.
\begin{proof}[{\bf Proof of Proposition \ref{lem:dolgo-est0}}]
First note that, for each $v\in\cV^u$, $L_v\phi_{-t}^*h=\phi_{-t}^*L_{(\phi_{-t})_*v}h$ and, by the Anosov property, $\|(\phi_{-t})_*v\|_{\Cs^0}\leq C_\# e^{-\olambda t}\|v\|_{\Cs^0}$. 
\begin{equation}\label{eq:u-normL}
\pqnorm[]{\cL_t h}^u\leq C_\#e^{-\olambda t} \sup_{v\in \cV^u}\pqnorm[0,0,d_s]{\cL_t (L_vh)}\leq C_\#e^{(\htop-\olambda) t}\pqnorm[]{h}^u.
\end{equation}
Thus
\begin{equation}\label{eq:u-normR}
\pqnorm[]{{\widehat R}_n(z) h}^u\leq 
C_\#(a-\sigma_{d_s}+\olambda)^{-n}\pqnorm[]{h}^u.
\end{equation}
To continue note that the proof of Lemma \ref{lem:LY} holds for the set of manifolds $\Sigma^s$, indeed the only required property of the set of manifolds is the fact that the image under the flow can be covered by manifolds in the set. Thus the first two inequalities of Lemma \ref{lem:LY} hold for the norms $\pqnorm[\eta]{\cdot}^s$. By the first
\[
\pqnorm[\eta]{{\widehat R}_n(z) h}^s\leq 
C_\#(a-\sigma_{d_s})^{-n}\pqnorm[\eta]{h}^s,
\]
which suffices to prove the first inequality of the Lemma.
To prove the second is much harder. The second inequality in Lemma \ref{lem:LY} can be used, as in Lemma \ref{lem:quasicompactness}, to yield
\begin{equation}\label{eq:reg-up0}
\pqnorm[\eta]{{\widehat R}_n(z) h}^s\leq \frac{C_\eta}{(a-\sigma_{d_s}+\lambda\eta)^n} \pqnorm[\eta]{h}^s+\frac{C_\eta}{(a-\sigma_{d_s})^n} \pqnorm[1+\eta]{h}^s.
\end{equation}
It is then sufficient to consider the case $h\in\Omega_{0,r}^{d_s}$, $g\in\Gamma^{\ell, 1+\eta}_c$.
In the following we will set $h_{s}= \cL_s^{(\ell)}h$. It is convenient to proceed by small time steps of size $r>0$, to be fixed later. Let $p(t)=\tilde p(\delta t)$ where $\tilde p$ is as in the proof of Lemma \ref{lem:delta-app}. 
For $n\in\bN$ large enough, we must estimate (see \eqref{eq:mani-split})
\begin{equation}\label{eq:dolgo-step0}
\begin{split}
\int_{c_an}^\infty dt\int_{W_{\alpha,G}} &\frac{t^{n-1}e^{-zt}}{(n-1)!}\langle g,\phi_{-t}^* h_{s}\rangle=\sum_{k\in\bN}\int_{c_an}^\infty dt\; \frac{t^{n-1}e^{-zt}p(k-tr^{-1})}{(-1)^{\ell(d-\ell)}(n-1)!}\\
&\times\int_{\phi_{-kr}W_{\alpha,G}}\hskip-.5cm\langle *\phi_t^**g, h_{s}\rangle\circ\phi_{-t+kr} J_W\phi_{kr}J\phi_{-t}\circ\phi_{kr}.
\end{split}\end{equation}

Next, it is convenient to localize in space as well. To this end we need to define a sequence of smooth partitions
of unity. 

For each $\bi\in \bZ^d$, let us define $x^{\bi}=r d^{-\frac 12}\bi$. We then introduce  the partition of unity $\Phi_{r,\bi}(x)=\prod_{l=1}^d p(d^{\frac 12}r^{-1} x_{i_l}-\bi_l)$  (limited to the set $B_d(0,30\delta)$). Note that it enjoys the following properties
\begin{enumerate}[\bf (i)]
\item $\Phi_{r,\bi}(z)=0$ for all $z\not\in B_d(x_\bi, r)$;
\item  the collection $\{B_d(x^{\bi}, r)\}_{\bi\in B_d(0,10\delta)}$ covers $B_d(0,30\delta)$ with a uniformly bounded number of overlaps;
\item for each $r,\bi$ we have $\|\nabla\Phi_{r,\bi}\|_{L^{\infty}}\leq C_\# r^{-1}$. 
\end{enumerate}
Note that the collection in {\bf (ii)} has a number of elements bounded by $C_\# r^{-d}$.

For each $k\in\bN$ let $\cW_{\alpha, G, \beta, k r}$ be the family of manifolds defined in \eqref{eq:def-family}, with $t=kr$, and set $\cW_{\alpha, G, \beta, k r, \bi}=\{W\in \cW_{\alpha, G, \beta, k r}\;:\; \Theta_\beta(W)\cap B_d(x^{\bi},r)\neq\emptyset\}$. For simplicity, let us adopt the notation $\cW_{k, \beta, \bi}\doteq \cW_{\alpha, G,\beta, k r, \bi}$. 
For $W\in \cW_{k, \beta, \bi}$ let $\tau_W:\widetilde W\doteq \cup_{t\in[-2r,2r]}\phi_t W\to \bR$ be defined by $\phi_{\tau_W(x)}(x)\in W$. Also set 
\begin{equation}\label{eq:g-def1}
\varphi_{k, \beta, \bi}(x)=\psi_\beta(x)\Phi_{r, \bi}(\Theta_\beta(x))p(r^{-1}\tau_W(x))\|V(x)\|^{-1}.
\end{equation}
Letting 
\begin{equation}\label{eq:g-def2}
\hat g_{k, \beta, \bi}=\varphi_{k, \beta, \bi}\frac{(k r+\tau_W)^{n-1}J_W\phi_{k r}\circ\phi_{\tau_W}}{(-1)^{\ell(d-\ell)}e^{-z(k r+\tau_W)}J\phi_{k r+\tau_W}(n-1)!}  *\phi_{k r+\tau_W}^**g,
\end{equation}
we can rewrite \eqref{eq:dolgo-step0} as\footnote{ First, for each $k\in\bN$, perform the change  change of variables $L_k:\bR\times W\to\widetilde W$ defined by $L_k(t, y)=\phi_{-t+kr}(y)$ and then introduce the partition of unity in space.  Note that this implies $t=kr-\tau_W(x)$. Also, note that the same manifold $\widetilde W$ appears several times. Yet, given the support of the functions, the value of the integral is the correct one.}
\begin{equation}\label{eq:dolgo-step1}
\begin{split}
\int_{cn}^\infty\hskip-.1cm dt\int_{W_{\alpha,G}} \frac{t^{n-1}\langle g,\phi_{-t}^* h_{s}\rangle}{e^{zt}(n-1)!}=\sum_{k, \beta, \bi}\;\sum_{W\in \cW_{k, \beta, \bi}}\int_{\widetilde W}\langle\hat g_{k, \beta, \bi}, h_{s}\rangle.
\end{split}
\end{equation}
Recalling  \eqref{eq:q-test} and since the derivative of $\hat g_{k,\beta,\bi}$ in the flow direction is uniformly bounded, we have 
\begin{equation}\label{eq:gb-est}
\left\| J_W\phi_{k r}\circ\phi_{\tau_W} \,(J\phi_{k r+\tau_W})^{-1}  *\phi_{k r-\tau_W}^**g\right\|_{\Gamma^{1+\eta}_c(\widetilde W)}\leq C_\# .
\end{equation}
For each $\beta, \bi$ we choose the reference manifold $W^s_\delta(\Theta_\beta^{-1}(x^{\bi}))\doteq \overline{W}_{\beta, \bi}\in\Sigma^s$, and define $W_{\beta, \bi}=\overline{W}_{\beta, \bi}\cap \Theta_\beta^{-1}(B_d(x^{\bi},2r))$ and $\widetilde W_{\beta, \bi}=\cup_{t\in[-4r,4r]}\phi_t(W_{\beta, \bi})$. Let us set $H_{\widetilde W_{\beta, \bi}, \widetilde W}\doteq H_{\beta, \bi, W}$. We can then rewrite the right hand side of \eqref{eq:dolgo-step1} multiplied by $(-1)^{\ell(d-\ell)}$ as 
\begin{equation}\label{eq:dolgo-step2}
\sum_{k, \beta, \bi}\;\sum_{W\in \cW_{k, \beta, \bi}}\int_{\widetilde W_{\beta, \bi}}\langle \hat g_{k, \beta, \bi, W}, h_{s}
\rangle+\cO\left(\frac{ r^{1-\eta} (kr)^{n-1}e^{(\sigma_{d_s}-\lambda) s}\pqnorm[]{h}^u}{(n-1)!e^{(a-\sigma_{d_s})kr}}\right),
\end{equation}
where $\hat g_{k, \beta, \bi, W}$ is given by Lemma \ref{sublem:holo} and we have used \eqref{eq:u-normL}.\footnote{ More precisely, rather than Lemma \ref{sublem:holo}, we have used equation \eqref{eq:sup-flow-est} and the fact that, in the present case, the support of $\vf_\tau$, in the flow direction, is contained in a strip of size $C_\# r$, which provides the extra factor $r$.} Then Lemma \ref{sublem:holo} and equations \eqref{eq:g-def1}, \eqref{eq:g-def2}, \eqref{eq:gb-est} imply
\begin{equation}\label{eq:g-est-one}
\|\hat g_{k, \beta, \bi, W}\|_{\Gamma^{\eta}_c(\widetilde W_{\beta,\bi})}\leq C_\# r^{-\eta}\frac{(k r)^{n-1}e^{-a k r}}{(n-1)!}.
\end{equation}
It is then natural to define 
\begin{equation}\label{eq:final-testf}
\mathfrak g_{k, \beta, \bi}\doteq \sum_{W\in \cW_{k, \beta, \bi}}\hat g_{k, \beta, \bi, W}.
\end{equation}
To conclude we need a sharp estimate for the sup norm of $\mathfrak g_{k, \beta, \bi}$, this will follow from an $L^2$ estimate and the inequality\footnote{ To prove the inequality note that if $f:\bR^\nu\to \bRp$ is a $\eta$-H\"older function and $x_0$ satisfies  $f(x_0)=\|f\|_{\Cs^0}$, then $f(x)\geq \|f\|_{\Cs^0}-\|f\|_{\Cs^\eta}|x_0-x|^\eta$, thus $f(x)\geq \frac 12 \|f\|_\infty$ for all $x\in B_\nu(x_0,\rho)$ where $\rho=[\|f\|_{\Cs^0}(2\|f\|_{\Cs^\eta})^{-1}]^\frac 1\eta$.}
\[
\int_{\widetilde W_{\beta, \bi}}\|\mathfrak g_{k, \beta, \bi}\|^2\geq  C_\#\|\mathfrak g_{k, \beta, \bi}\|_{\Gamma_c^0}^{2+\frac{d_s+1}{\eta}} \|\mathfrak g_{k, \beta, \bi}\|_{\Gamma^\eta_c}^{-\frac{d_s+1}\eta}.
\]
Hence
\begin{equation}\label{eq:final-testf2}
\|\mathfrak g_{k, \beta, \bi}\|_{\Gamma^0_c}\leq C_\#  \|\mathfrak g_{k, \beta, \bi}\|_{L^2(\widetilde W^s_{\beta, \bi})}^{\mathfrak p} \|\mathfrak g_{k, \beta, \bi}\|_{\Gamma^\eta_c}^{1-\mathfrak p}.
\end{equation}
where we have set $\mathfrak p=\frac{2\eta}{2\eta+d_s+1}$.
To compute the $L^2$ norm it is convenient to proceed as follows. Given $W\in\cW_{k, \beta, \bi}$ let $\{x_W\}\doteq \widetilde W\cap \Theta_\beta^{-1}(\{x^\bi+(0,\eta)\}_{\eta\in B_{d_u}(0,r)})$ be its ``central" point. Also, for each $\rho>0$, consider the disks $\widetilde D^u_\rho(W)=\{\Theta_\beta(x_W)+(0,u,0)\}_{u\in B_{d_u}(0, \rho)}$. Then, for a fixed $\varrho\in (0,r)$ to be chosen later, let $A_{k, \beta, \bi}(W)=\{W'\in \cW_{k, \beta, \bi}\;:\;  \Theta_\beta(\widetilde W')\cap \widetilde D^u_\varrho(W)\neq\emptyset\}$ and $B_{k, \beta, \bi}(W)=\cW_{k, \beta, \bi}\setminus A_{k, \beta, \bi}(W)$ and set
\[
\begin{split}
&\mathfrak G_{k, \beta, \bi, A}\doteq \sum_{W\in \cW_{k, \beta, \bi}}\sum_{W'\in A_{k, \beta, \bi}(W)}\langle\hat g_{k, \beta, \bi, W}, \hat g_{k, \beta, \bi, W'}\rangle\\
&\mathfrak G_{k, \beta, \bi, B}\doteq \sum_{W\in \cW_{k, \beta, \bi}}\sum_{W'\in B_{k, \beta, \bi}(W)}\langle\hat g_{k, \beta, \bi, W},\hat g_{k, \beta, \bi, W'}\rangle,
\end{split}
\] 
note that $\langle\mathfrak g_{k, \beta, \bi},\mathfrak g_{k, \beta, \bi}\rangle=\mathfrak G_{k, \beta, \bi, A}+\mathfrak G_{k, \beta, \bi, B}$.
We will estimate the two terms by separate arguments, the first being similar to \cite[Lemma 6.2]{Liverani04}, the second being the equivalent of \cite[Lemma 6.3]{Liverani04}. 

To simplify  the notation we set: $D_{k,\beta,\bi} = \frac{(k r)^{n-1}e^{-a rk}}{(n-1)!}\#\cW_{k, \beta, \bi}$, where by $\#A$ we mean the cardinality of the set $A$. Also, we assume
\begin{equation}\label{eq:n-choice}
n\geq \Cnd\ln |b|
\end{equation}
where $\Cnd$ is a constant to be chosen large enough.

\begin{lem}\label{sublem:doest-2} 
If $C_\# |b|^{-\frac{\bar\lambda\Cnd\Cnc}{2ea_0} }\leq \varrho < \delta r^{(1+\varsigma)\frac 2{\hat\varpi}}$, for some $\varsigma>0$, then there exists $\varsigma_0>0$ s.t.
\begin{equation}\label{eq:doest-2}
\begin{split}
\|\mathfrak G_{k, \beta, \bi, A}\|_\infty\leq& C_\# D_{k,\beta,\bi}^2 r^{\varsigma_0}.
\end{split}
\end{equation}
\end{lem} 
\begin{proof}
Note that the number of elements in  $A_{k, \beta, \bi }(W)$ must correspond to the number of intersections between $\cup_{t\in[-3\delta,3\delta]}W_{\alpha,G}$  and $\phi_{k r}(D_\varrho^u(W))$. Since each intersection has a $\delta$-neighborhood in $\phi_{k r}(D_\varrho^u(W))$ that cannot contain any other intersection we have $\# A_{k, \beta, \bi }(W)\leq C_\#\text{vol}(\phi_{k r}(D_{2\varrho}^u(W)))$, provided that each point in $\phi_{k r}(D_\varrho^u(W))$ has a $\delta$ neighbourhood contained in  $\phi_{k r}(D_{2\varrho}^u(W))$. The latter condition is satisfied provided 
$C_\#|b|^{\bar\lambda\frac{ \Cnc\Cnd}{2 ea_0}}\varrho\geq \delta$, which is implied by the first hypothesis in the Lemma. On the other hand, by the mixing property,\footnote{Recall that a contact flow on a connected manifold is mixing, \cite{Katok94}.} a disk in $\phi_{k r}(D_r^u(W))$ with diameter larger than $C_\#$ must intersect $\cup_{t\in[-3\delta,3\delta]}W_{\alpha,G}$, thus $\#\cW_{k, \beta, \bi}\geq C_\#\text{vol}(\phi_{k r}(D_{r/2}^u(W)))$.  We are thus reduced to estimating the ratio of two volumes.

The simplest possible estimate is as follows:  let $\bar t\in\bRp$ be such that $\Cnz e^{\olambda \bar t}r= r^{-\varsigma}$, now assume that $2e^{\lambda_+ \bar t}\varrho\leq \delta$. It follows that at  time $\bar t$ the image of $D_{2\varrho}^u(W)$ is contained in a disk of radius $\delta$ while the image of $D_{r/2}^u(W)$ contains at least $C_\# r^{-\varsigma d_u}$ such disks. By Lemmata \ref{lem:h1} and \ref{lem:h3} such a ratio will persist at later times.
Note that the above assumptions are verified only if $\varrho < \delta r^{(1+\varsigma)\frac 2{\hat \varpi}}$.
\end{proof}

\begin{lem}\label{sublem:doest-3}
Setting $|b|=\varrho^{-2+\varpi'-2\varsigma}$ and $ r= \varrho^{1-\varpi'+\varsigma}$, for $\varsigma> 0$, there exists $\gamma>0$ such that
\begin{equation}\label{eq:doest-3}
\left|\int _{\widetilde W_{\beta, \bi}} \mathfrak G_{k, \beta, \bi, B}\right|\leq C_\#  b^{-6\gamma}r^{d_s+1}D_{k,\beta,\bi}^2.
\end{equation}
\end{lem}
The proof of the above Lemma is postponed to Subsection \ref{subsec:dolgoest}. 
\begin{rmk} Note that the conditions of the two Lemmata can be simultaneously satisfied, by choosing $\Cnd$ large and $\varsigma$ small enough,  only if $\varpi'>\frac 23$. In addition, since $\varpi'\leq 1$, $\varsigma$ must always be smaller than $\frac{\sqrt 3-1}2$.
\end{rmk}

Substituting in \eqref{eq:final-testf2} the estimates given by \eqref{eq:g-est-one},  \eqref{eq:doest-2} and  \eqref{eq:doest-3}, we obtain that there exists $\gamma_0\in (0,\gamma]$ such that
\begin{equation}\label{eq:final-testf3}
\|\mathfrak g_{k, \beta, \bi}\|_{\Gamma^0_c}\leq   C_\# D_{k,\beta,\bi} |b|^{-3\gamma_0} .
\end{equation}
Next, we assume $s\geq c_an$.
Then, by \eqref{eq:g-est-one} and \eqref{eq:final-testf3} and arguing as in the proof of \eqref{eq:zero-bis-norm},\footnote{ Namely, we use the first line of the equation before \eqref{eq:zero-bis-norm} with $g-g_\ve={\mathfrak g}_{k, \beta, \bi}$, $\ve=1$ and $q=\eta$.} we can write the integral in \eqref{eq:dolgo-step2} as\footnote{ Note that, since  $\mathfrak g_{k, \beta, \bi}$ is supported in $\widetilde W_{\beta, \bi}$, the integral can also be seen as an integral on $\overline{W}_{\beta,\bi}\in\Sigma^s$.}
\[
\left|\sum_{k, \beta, \bi}\;\int_{\widetilde W_{\beta, \bi}}\langle \mathfrak g_{k, \beta, \bi}, h_{s}\rangle\right|\leq C_\# \sum_{k, \beta, \bi}D_{k,\beta,\bi} (b^{-3\gamma_0}+e^{-\lambda\eta s}r^{-\eta}) \vol\!\left(\phi_{-s} \widetilde W_{\beta, \bi}\right) \pqnorm[\eta]{h}^s.
\]
Note that $\sum_{\bi}\#\cW_{k, \beta, \bi}|\phi_{-s} \widetilde W_{\beta, \bi}|\leq C_\# r\sum_{W\in \cW_{\alpha, G,\beta, kr}} |\phi_{-s} W|$. Thus, by Lemma \ref{lem:h3} and Remark \ref{rem:h0},

\begin{equation}\label{eq:dolgo-step3}
\begin{split}
& \left|\sum_{k, \beta, \bi}\;\int_{\widetilde W_{\beta, \bi}}\langle \mathfrak g_{k, \beta, \bi}, h_{s}\rangle\right|\leq C_\#\sum_k  \frac{(k r)^{n-1}(|b|^{-3\gamma_0}+e^{-\lambda\eta s}r^{-\eta})}{(n-1)!e^{a rk-\sigma_{d_s} (kr+s)}} r \pqnorm[\eta]{h}^s\\
&\leq C_\# (a-\sigma_{d_s})^{-n} (|b|^{-3\gamma_0}+e^{-\lambda\eta s}r^{-\eta})  e^{\sigma_{d_s} s}\pqnorm[\eta]{h}^s.
\end{split}
\end{equation}
Thus, by \eqref{eq:dolgo-step1}, \eqref{eq:dolgo-step2}, \eqref{eq:dolgo-step3},
\begin{equation}\label{eq:dolgo0}
\begin{split}
&\pqnorm[1+\eta]{{\widehat R}_n(z)^2 h}^s\leq\frac{C_\#  }{(a-\sigma_{d_s})^{2n}}\left[\frac{r^{-\eta}\pqnorm[]{h}^u(a-\sigma_{d_s})^{n}}{(a-\sigma_{d_s}+\lambda\eta)^{n}} +|b|^{-3\gamma_0}\pqnorm[\eta]{ h}^s\right].
\end{split}
\end{equation}
Next, by \eqref{eq:reg-up0},  \eqref{eq:dolgo0} and Lemma \ref{sublem:doest-3},
\begin{equation}\label{eq:dolgo1}
\begin{split}
&\pqnorm[\eta]{{\widehat R}_n(z)^3 h}^s\leq \frac{C_\eta  \pqnorm[\eta]{{\widehat R}_n(z)^2 h}^s}{(a-\sigma_{d_s}+\lambda\eta)^{n}}+\frac{C_\eta }{(a-\sigma_{d_s})^{n}} \pqnorm[1+\eta]{{\widehat R}_n(z)^2 h}^s\\
&\leq\frac{C_\eta  }{(a-\sigma_{d_s})^{3n}}\left[\frac{(a-\sigma_{d_s})^{n}|b|^{\frac \eta 2}}{(a-\sigma_{d_s}+\lambda\eta)^{n}} +|b|^{-3\gamma_0}\right] \pqnorm[]{h}^*.
\end{split}
\end{equation}
The above, and \eqref{eq:u-normR}, \eqref{eq:n-choice} imply the Lemma provided $\gamma_0$ has been chosen smaller than $40 \eta $.
\end{proof}

\subsection{A key (but technical) inequality}\label{subsec:dolgoest}
We are left with the task of proving \eqref{eq:doest-3}.

\begin{proof}[{\bf Proof of Lemma \ref{sublem:doest-3}}]
This is the heart of Dolgopyat's estimate. 
Given $W\in\cW_{k, \beta, \bi}$ and $W'\in B_{k, \beta, \bi}(W)$, we must estimate
\begin{equation}\label{eq:doest-4}
\begin{split}
&\int_{\widetilde W_{\beta, \bi}}\hskip-.2cm\langle \hat g_{k, \beta, \bi,W},\hat g_{k, \beta, \bi,W'}\rangle=\int_{\widetilde W_{\beta, \bi}}\hskip-.3cm JH_{\beta, \bi, W}\cdot JH_{\beta, \bi, W'} \cdot \overline{\hat g_{k, \beta, \bi,\bar i}}\circ H_{\beta, \bi, W}\\
&\hskip6cm\times \hat g_{k, \beta, \bi, \bar i}\circ H_{\beta, \bi, W'}\, .
\end{split}
\end{equation}
It turns out to be convenient to write the above as an integral over $\widetilde W$. More precisely, let $\tilde \vartheta_{W}:B_{d}(0,\delta)\to\bR^d$ be a flow box coordinate change preserving the contact form $\alpha_0$ and such that, setting $\vartheta_{W}=\Theta_\beta^{-1}\circ\tilde \vartheta_{W}$, has the property that $\vartheta_{W}(\{(\xi,0)\}_{\xi\in B_{d_s}(0,C_\#r)})\subset W$ and, moreover, $\widetilde W$ contains the support (projected on via the unstable Holonomy) of the integrand.\footnote{ Such a coordinate change always exists, see \cite[Lemma A.4]{BaladiLiverani11}.} 
We can then rewrite \eqref{eq:doest-4} as\footnote{ By Appendix \ref{app:holo-est} it follows that the image of $W'$ on $\widetilde W$ by the unstable holonomy is strictly contained in a ball of radius $2r$, provided $r$ is small enough.}
\begin{equation}\label{eq:doest-5}
\begin{split}
&\int_{\widetilde W_{\beta, \bi}}\hskip-.2cm\langle \hat g_{k, \beta, \bi,W},\hat g_{k, \beta, \bi,W'}\rangle=\int_{B_{d_s+1}(0,2r)} \hskip-1.5cm d\xi\,e^{-z(\tau_{W'}\circ \widetilde H_{W'}(\xi)-\xi_d)}\cG_{k, \beta, \bi, W,W', \bar i}(\xi) ,
\end{split}
\end{equation}
where $\xi=(\tilde\xi,\xi_d)$, $\widetilde H_{W'}=H_{\widetilde W, \widetilde W'}\circ\vartheta_W$,  and
\[
\begin{split}
\cG_{k, \beta, \bi, W,W', \bar i}\doteq &\left\{\overline{\varphi_{k, \beta}\frac{(k r+\tau_W)^{n-1}e^{-a k r}}{(n-1)!} \frac{J_W\phi_{k r}\circ\phi_{\tau_W}}{J\phi_{k r+\tau_{W}}} (*\phi_{k r+\tau_W}^**g)_{\bar i}}\right\}\circ \vartheta_W\\
&\times\left\{\varphi_{k, \beta}\frac{(k r+\tau_{W'})^{n-1}e^{-a k r}}{(n-1)!} \frac{J_{W'}\phi_{k r}\circ\phi_{\tau_{W'}}}{J\phi_{k r+\tau_{W'}}} (*\phi_{k r+\tau_{W'}}^**g)_{\bar i}\right\}\circ\widetilde H_{W'}\\
&\times   J\widetilde H_{\beta, \bi, W'}.
\end{split}
\]
Let us define
\begin{equation}\label{eq:temp-def}
\Delta^*(\xi)\doteq\tau_{W'}\circ \widetilde H_{\beta, \bi, W'}(\xi)-\xi_d,
\end{equation}
and note that $\Delta^*(\xi+\zeta)-\Delta^*(\xi)$ is exactly the so called {\em temporal function}.\footnote{ In the language of  \cite{Liverani04} (in which the role of the stable and unstable manifolds are reversed) $\Delta^*(\xi+\zeta)-\Delta^*(\xi)=\Delta(y,y')$, where  $x=\vartheta_{W}(\xi)$, $y=\widetilde  H_{W'}(\xi)$ and $y'=\vartheta_W(\xi+(\tilde\zeta,0))$ with $\zeta=(\tilde\zeta,\zeta_d)$.} We let $\wz(\xi)=(\wz^s(\xi),\wz^u(\xi),\wz_d(\xi))=\vartheta_{W}^{-1}\circ\widetilde  H_{W'}(\xi)-(\xi,0)$. Then by Lemma \ref{lem:temporal} we have\footnote{ It may be possible to improve this estimate  by going forward or backward in time to the situation in which the two lengths are equal, but some non-obvious distortion estimate would have to play a role to ensure that the following ``preferred" direction does not changes too wildly.}
\begin{equation}\label{eq:temporal}
\begin{split}
\left|\Delta^*(\xi+\zeta)-\Delta^*(\xi)- d\alpha_0(\wz(\xi), \zeta)\right|\leq &C_\#\|\wz\|^2\|\tilde\zeta\|^{\varpi'}+C_\#\|\wz\|^{\varpi'}\|\tilde\zeta\|^2\\
&\quad +C_\#\|\wz\|^{1+\varpi'}\| \tilde\zeta\|^{1+\varpi'}.
\end{split}
\end{equation}
We are left with the task of performing the integral in the lefthand side of \eqref{eq:doest-5}. 
The basic idea (introduced in \cite{Liverani04}) is to integrate first in the direction $y_\bi=\|\wz^u(0)\|^{-1}(\wz^u(0),0)$. Note that $d\alpha_0(\wz(0),y_{\bi})=\|\wz^u(0)\|$ and, by the cone condition, $\|\wz^s\|\leq C_\#\|\wz^u\|$.
Hence, by the H\"older continuity of the unstable foliation (detailed in Appendix \ref{app:holo-est}), for  $\xi\in B_{d_s+1}(0,r)$, 
\begin{equation}\label{eq:goodir}
d\alpha_0(\wz(\xi),y_{\bi})\geq C_\#\|\wz(\xi)\|\geq C_\#\varrho.
\end{equation}
Next, we impose the condition
\begin{equation}\label{eq:restriction0}
r=\varrho^{1-\varpi'+\varsigma}
\end{equation}
for some, very small, $\varsigma>0$.
By Lemma \ref{lem:holo-dist} we have that
\begin{equation}\label{eq:zeta-dist}
\|\wz(\xi)-\wz(\xi+\zeta)\|\leq C_\#(\|\wz(\xi)\|\|\tilde\zeta\|^{\varpi'}+\|\wz(\xi)\|^{\varpi'}\|\tilde\zeta\|)\leq C_\#\varrho.
\end{equation}
For each $\xi\in B_{d_s+1}(0, r)$, $\langle\xi,y_{\bi}\rangle=0$, we define $I_{\bi}(\xi)=\{s\in\bR\;:\; \xi+sy_{\bi}\in B_{d_s+1}(0, r)\}$ and consider the integral
\begin{equation}\label{eq:startingpoint}
\begin{split}
&\left|\int_{I_{\bi}(\xi)} ds\; e^{-z\Delta^*(\xi+sy_{\bi})}\cG_{k, \beta, \bi, W,W', \bar i}(\xi+sy_{\bi})\right|.
\end{split}
\end{equation}
It is then natural to divide the interval $I_{\bi}(\xi)$ in subintervals $\{s_l,s_{l+1}\}_{l\in\bZ}$ such that, 
\[
(s_{l+1}-s_l)d\alpha_0(\wz(\xi+s_l y_{\bi}),y_{\bi})=2\pi |b|^{-1}.
\] 
Let us set $\delta_l=s_{l+1}-s_l$ and $\wz_{l}=\wz(\xi+s_ly_{\bi})$, then, recalling \eqref{eq:goodir}, we have

\begin{equation}\label{eq:deltarange}
C_\# |b|^{-1}\leq\frac{\delta_l\|\wz_{l}\|}{2\pi} \leq  C_\#|b|^{-1}.
\end{equation}
Next, we must assume

\begin{equation}\label{eq:restriction1}
 |b|^{-1}\leq r\varrho,
\end{equation}
to insure that $I_{\bi}(\xi)$ contains at least one interval.
Note that Lemma \ref{lem:holo-dist} implies that $\|\widetilde H_{W'}(\xi+y_{\bi}s_{l+1})-\widetilde H_{W'}(\xi+y_{\bi}s_{l})\| \leq C_\#\max\{\delta_l, \|\wz_{l}\| \delta_l^{\varpi'}\}$.
Hence\footnote{ We remark that $\|J\widetilde H_{\beta,\bi,W'}(\xi)-J\widetilde H_{\beta,\bi,W'}(\xi+\zeta)\|\leq C_\# \|\wz_{l}\|\|\zeta\|^{\varpi_*}$, see Appendix \ref{app:holo-est}.}
\[
\begin{split}
&\sup_{|s-s_l|\leq \delta_i}|\cG_{k, \beta, \bi, W,W', \bar i}(\xi+sy_{\bi})-\cG_{k, \beta, \bi, W,W', \bar i}(\xi+s_ly_{\bi})|\leq
C_\# D_{n,k}^2\\
&\hskip7cm \times(r^{-1}\max\{\delta_l, \|\wz_{l}\| \delta_l^{\varpi'}\}+\delta_l^{\varpi_*}\|\wz_{l}\|)\\
&\text{where }D_{n,k}\doteq  \frac{(k r)^{n-1}e^{-a k r}}{(n-1)!}.
\end{split}
\]
We can then bound \eqref{eq:startingpoint} by
\[
\begin{split}
\sum_l&\left|\int_{s_l}^{s_{l+1}} \hskip-.4cm e^{-z [\Delta^*(\xi+sy_{\bi})-\Delta^*(\xi+s_ly_{\bi})]}\cG_{k, \beta, \bi, W,W', \bar i}(\xi+s_ly_{\bi})\right|\\
&+C_\# D_{n,k}^2\left( |b|^{-1}\|\wz_{l}\|^{-1}+ |b|^{-\varpi'}\|\wz_{l}\|^{1-\varpi'}+\frac{\|\wz_{l}\|^{1-\varpi_*} r}{|b|^{\varpi_*}}\right).
\end{split}
\]
Next, equations \eqref{eq:temporal}, \eqref{eq:deltarange} yield
\[
\begin{split}
|\Delta^*(\xi+sy_{\bi})-\Delta^*(\xi+s_ly_{\bi})-(s-s_l)&d\alpha_0(\wz_{l}, y_{\bi})|\leq  C_\# \frac{\|\wz_{l}\|^{2-\varpi'}}{ |b|^{\varpi'}}\\
&+C_\#\frac{\|\wz_{l}\|^{\varpi'-2}}{|b|^{2}}+C_\#\frac{1}{|b|^{1+\varpi'}}.
\end{split}
\]
Recalling  \eqref{eq:restriction0} and \eqref{eq:zeta-dist}, we can continue our estimate of \eqref{eq:startingpoint} as 
\[
\begin{split}
&C_\#\sum_l\left|\int_{s_l}^{s_{l+1}} ds\; e^{-z (s-s_l)d\alpha_0(\wz_{l}, y_\bi)}\cG_{k, \beta, \bi, W,W', \bar i}(\xi+s_ly_{\bi})\right|\\
&\quad+C_\# D_{n,k}^2 r\bigg( |b|^{-1} \varrho^{-1}r^{-1}+\frac{r^{1-\varpi_*}}{|b|^{\varpi_*}}+\frac{r^{2-\varpi'}}{ |b|^{\varpi'-1}}+\frac{\varrho^{\varpi'-2}}{|b|}+\frac{r^{-\varpi'}}{|b|^{\varpi'}}\bigg).
\end{split}
\]
We remark that the integrals in the first line are all bounded by $C_\# |b|^{-1}\delta_l$.
By the above equation, integrating in the remaining directions and letting\footnote{ The choice for $\varrho$ satisfies the constraints \eqref{eq:restriction1} .}
\[
D_{k, n, \beta, \bi}\doteq D_{k,n} \# \cW_{k, \beta, \bi} \hbox{ and } \quad |b|^{-1}=\varrho^{2-\varpi'+2\varsigma}
\]
equation \eqref{eq:doest-5} yields
\[
\left|\int _{\widetilde W_{\beta, \bi}} \mathfrak G_{k, \beta, \bi, B}\right|\leq  C_\# D_{k, n, \beta, \bi}^2\left[\varrho^{\varpi'\varsigma}+\frac{1}{|b|^{\frac{\varpi'}2}}\right] r^{d_s+1}\leq C_\# D_{k, n, \beta, \bi}^2 |b|^{-\frac {\varpi'\varsigma}{2-\varpi'+2\varsigma}}r^{d_s+1}.
\]
which gives the required inequality.
\end{proof}

\section{Growth of $\zeta$-functions}\label{sec:zeta-growth}

We start by showing that the estimates in section \ref{sec:dolgo} imply a bound on the growth of the traces. To this end note that the previous results show that $\trf(R^{(d_s)}(z)^n)$, which is well defined for $\Re(z)$ large enough, equals a function meromorphic in $\Re(z)>\sigma_{d_s} -\frac{\sigma_{p,q}}2$, for each $0<p+q<r-1$. Since  $\trf(R^{(d_s)}(z)^n)$ is independent on $p,q$, it can be extended as much as the condition $p+q<r-1$ allows. It is then natural to use $\trf(R^{(d_s)}(z)^n)$ to  also denote  such an extension.
Before getting to the point we need a refinement of equation \eqref{eq:trace-quant} which holds only for $\ell=d_s$.

\begin{lem} \label{lem:traceorbits-bis}
For each  $a=\Re(z)> \htop$ and $n \in \mathbb{N}$
\[
\begin{split}
\left|\int_{M} \sum_{\alpha,\bar{i}}   \langle \omega_{\alpha,\bar i},R^{(d_s)}(z)^n j_{\ve, \alpha,\bar i,x} \rangle_x -  
\sum_{\tau \in \cT}  \frac{ \chi_{d_s}(\tau)\lambda(\tau)^{n-1}\lambda(\tau)}{(n-1)!\mu(\tau)e^{z \lambda(\tau)}}\right|
&\leq \frac{C_\#\ve|z| n}{(a-\sigma_{d_s})^{n+1}}\\
&\hskip-.8cm+\frac{C_\#|z| n}{(a-\sigma_{d_s}+\lambda)^{n+1}}.
\end{split}
\]
\end{lem}
\begin{proof}
We start by equation \eqref{eq:finalizing0} with $s=0$ and $\ell=d_s$ and then we look more in depth at the derivative of $F_{\alpha,\bar i,z}$. By formulae \eqref{eq:ftrace-one} and \eqref{eq:trace-ok} we have
\[
\sum_{\bar i}F_{\alpha,\bar{i},z}(x,t)  =  \frac{t^{n-1}}{(n-1)!}e^{-zt} \psi_\alpha(x)  \textrm{tr}(\wedge^{d_s}( D_{\textrm{hyp}}\phi_{-\lambda(\tau)}(x) ).
\]
To compute the trace is it convenient to do a linear change of variables in the zero section such that the periodic orbit is at zero and its stable and unstable substances are spanned by $e_1,\dots, e_{d_s}$ and $e_{d_s+1},\dots, e_{d-1}$. Then, calling $\omega_i$ the dual base, we have that $\langle \omega_1\wedge\dots\wedge \omega_{d_s}, D_{\textrm{hyp}}\phi_{-\lambda(\tau)}\omega_1\wedge\dots\wedge \omega_{d_s}\rangle$ gives exactly the stable Jacobian $J^s_t$ of the map, while all the other terms in the trace are smaller than $C_\# J^s_te^{-\bar\lambda t}$. Note that this implies that $\sup_{\tau}\chi_{d_s}(\tau)\leq C_\#$. By the usual distortion estimates it follows that
\[
\|\partial_\zeta J^s_t\|_\infty \leq C_\# \| J^s_t\|_\infty\leq C_\# |\det(\Id-\Lambda_{\tau,\alpha,m}(0))| ,
\]
where we are using the coordinates defined in \eqref{eq:Xi-coordinates}. By the above inequality and Sub-Lemma \ref{sublem:change-smooth} we can obtain from equation \eqref{eq:finalizing0} the improved estimated (with respect to equation \eqref{eq:finalizing})
\begin{equation}\label{eq:finalizing-ds}
\begin{split}
\sum_{\alpha,\bar{i},m,\tau\in\cT_{\ve}}\int_{-\delta}^\delta&\frac{F_{\alpha,\bar{i},z} \circ \widetilde{\Theta}_{\alpha}^{-1} \circ \Xi_{\tau,\alpha,m}^{-1}(0,\xi_d , \lambda(\tau)) }{|\det(\Id-\Lambda_{\tau,\alpha,m}(0))|}\\
&+\frac{\lambda(\tau)^{n-1}}{(n-1)!}\cO\left(\ve |z|e^{-a\lambda(\tau)}+|z|e^{-(a+\bar\lambda)\lambda(\tau)}\right).
\end{split} 
\end{equation}
Using \cite[Theorem 18.5.7]{KH} we can bound the total error term in the above equation by
\[
 C_\# \ve |z|n(a-\htop)^{-n-1}+C_\#|z|n(a+\bar\lambda-\htop)^{-n-1}.
\]
The Lemma is then proven arguing exactly as at the end of Lemma \ref{lem:traceorbits}.
\end{proof}
\begin{lem}\label{lem:tracebound} There exists $a_0\geq 1$ and $\tau_*>0$ such that, for each $n\in\bN$,  $z\in \bC$, $\htop+2a_0\geq a=\Re(z)\geq \htop+a_0$, $b=\Im(z)\not\in [-1,1]$, 
\[
|\trf(R^{(d_s)}(z)^n)|\leq C_\# |b| (a-\sigma_{d_s}+\tau_*)^{-n}.
\]
\end{lem}
\begin{proof}
Choose $\eta\in (0,\varpi_*)$ and consider $R^{(d_s)}(z)\in L(\widetilde\cB^{1,\eta,d_s},\widetilde\cB^{1,\eta,d_s})$.
Next, we establish the analogue of \cite[Lemma 3.1]{Liverani05}.
By Lemmata \ref{lem:traceorbits} and \ref{lem:traceorbits-bis} it follows, for each $\ve< \ve_0$,
\[
\begin{split}
\left|\trf(R^{(d_s)}(z)^{n})\right|&\leq C_\# \spqnorm[1,\eta,d_s]{R^{(d_s)}(z)^{n}} \sup_{x\in M}\spqnorm[1,\eta,d_s]{j_{\epsilon, \alpha,\bar i,x}}+ \frac{C_\#|z|\ve n}{ (a-\sigma_{d_s})^{n+1}}\\
&\quad+\frac{C_\#|z|n}{(a+\lambda-\sigma_{d_s})^{n+1}}\\
&\leq C_\# \ve^{-C_\eta}\spqnorm[1,\eta,d_s]{R^{(d_s)}(z)^{n}} + \frac{C_\#|z|\ve n}{(a-\sigma_{d_s})^{n+1}}
+\frac{C_\#|z| n}{(a+\lambda-\sigma_{d_s})^{n+1}}
\end{split}
\]
where, in the first line, we have used the usual trick of seeing, locally, the integral over $M$ as an integral over a foliation made of elements of $\Sigma$.
Next, given $b\in \bR$, $|b|\geq B$, let $z=a+ib$, $a=a_0+\sigma_{d_s}$ and $c>0$. Also, let $n_b= \lceil \Cnd C_\# \ln |b|\rceil$ and write, for any $n\in\bN$, $n=k_n n_b+r_n$, $r_n<n_b$. Lemmata \ref{lem:dolgo-est}, \ref{lem:quasicompactness} imply
\begin{equation}\label{eq:rz-bound}
\begin{split}
\spqnorm[1,\eta,d_s]{R^{(d_s)}(z)^{n}}&\leq \frac{C_\eta^{k_n+1} |b|^{-\gamma_0 k_n}}{(a-\sigma_{d_s})^n}\leq \frac{ |b|^{\gamma_0}e^{\frac n{\Cnd C_\# \ln|b|}(\ln C_\eta-\gamma_0\ln |b|)}}{(a-\sigma_{d_s})^n}\\
&\leq \frac{C_\eta|b|^{\gamma_0} e^{-\frac{n\gamma_0}{2\Cnd C_\#} }}{ (a-\sigma_{d_s})^n},
\end{split}
\end{equation}
where we have assumed $|b|\geq C_\eta^{\frac 2\gamma_0}=:B$. Choosing $\ve$ depending on $n$ by requiring $ \ve^{-C_\eta} e^{-\frac{n\gamma_0}{2C_\#\Cnd} }= \ve $ yields
\[
\left|\trf(R^{(d_s)}(z)^{n})\right|\leq C_\#(a-\sigma_{d_s})^{-n}|b| e^{-\frac {n\gamma_0}{3\Cnd C_\#(1+C_\eta)} } +C_\#|b|n(a+\lambda-\sigma_{d_s})^{-n}
\]
implying the result for $|b|\geq B$. 
On the over hand, by Proposition \ref{lem:Nussbaum}, in the region $\{z\in\bC\;:\; |\Im(z)|\leq B, \,\htop-1\leq \Re(z)\leq\htop\}$ there are only finitely many eigenvalues of $X^{(d_s)}$ and $\htop$ is the only eigenvalue on the line $\Re(z)=\htop$. Accordingly, there exists $\tau_0>0$ such that $\sigma(X^{(d_s)})\cap \{z\in\bC\;:\; |\Im(z)|\leq B, \,\htop-\tau_0\leq \Re(z)<\htop\}=\emptyset$. By standard spectral theory it follows that, for $\tau_1<\min\{\tau_0, \frac 1{2a_0}\}$,
\[
\spqnorm[1,\eta,d_s]{R^{(d_s)}(z)^{n}}\leq C_\#(a+\tau_1-\sigma_{d_s})^{-n}.
\]
Using the above instead of \eqref{eq:rz-bound} yields the results in the remaining region.
\end{proof}

\begin{lem}\label{lem:ruellebound}
In the region $\{z\in\bC\;:\; \Re(z)\in[\htop-\tau_*,\infty)\}$ $\zeta_{\Ruelle}$ is non zero and has only a simple pole at $z=\htop$. Moreover, setting $D=\{z\in\bC\;:\; \Re(z)\in[\htop-\tau_*,A]\;; |\Im(z)|\geq 1\}$ we have
\[
\begin{split}
|\zeta_{\Ruelle}(z)|\leq C_\# e^{C_\# |z|}&\quad \text{for all } z\in D\\
\left|\frac{d}{dz}\ln\zeta_{\Ruelle}(z)\right|\leq C_\# |z| &\quad  \text{for all } z\in D.
\end{split}
\]
\end{lem}
\begin{proof}
If $\tau_*$ is chosen small enough, then Proposition \ref{main2} implies that the determinants $\mathfrak{D}_\ell$, $\ell\neq d_s$, are analytic, bounded and non zero in the regions considered  and hence, recalling  equation \eqref{eq:prod},  it suffices to study $\mathfrak{D}_{d_s}(z)$.
In view of Lemma \ref{damnedremark} we start by studying $\widetilde {\mathfrak{D}}_{d_s}(\xi-z,\xi)$. Fix $z_0\in D$ and consider $\xi=a_0+\htop+i\Im(z_0)$. Then, recalling  Lemma \ref{lem:traceorbits}
\[
\mathfrak{D}_{d_s}(z)=\mathfrak{D}_{d_s}(\xi)\tilde {\mathfrak{D}}_{d_s}(\xi-z,\xi)=\mathfrak{D}_{d_s}(\xi)\exp\left(-\sum_{n=1}^\infty\frac{(\xi-z)^n}n\trf( R(\xi)^n)\right).
\]
Note that $\mathfrak{D}_{d_s}(\xi)$ is uniformly bounded in $z_0$.
By Lemma \ref{lem:tracebound} it follows that 
\[
|\mathfrak{D}_{d_s}(z_0)|\leq C_\#\exp\left(\sum_{n=1}^\infty\frac{|\Im(z_0)|(a_0+\htop-\Re(z_0))^n}{n (a_0+\tau_*)^{n}}\right)
\]
which is convergent provided $\Re(z_0)> h_{\text{\tiny top}}(\phi_1)-\tau_*$.

Next, by the same argument as above,  the logarithmic derivative
\[
\frac d{dz}\ln \mathfrak{D}_{d_s}(z)=\sum_{n=1}^\infty(\xi-z)^{n-1}\trf( R(\xi)^n)
\]
 is bounded as claimed in the required  domain.
\end{proof}

\appendix
\appendixpage
\addappheadtotoc

\section{External Forms: a toolbox} \label{appendiceLie}

In this appendix we collect, for the reader's convenience, some useful formulae.
More details can be found in \cite[Section 2]{Jost08}).

Given a Riemannian $d$ dimensional manifold $M$, for each $x\in M$, $\ell\in\{0,\dots,d\}$ and $v_1,\dots, v_\ell, w_1,\dots, w_\ell\in T_x^*M$ we define\footnote{ By duality the scalar product in $T_*M$ induces a canonical scalar product in $T^*M$.}
\begin{equation}\label{eq:scalarproduct}
\langle v_1\wedge\cdots\wedge v_\ell,w_1\wedge\cdots \wedge w_\ell\rangle=\det (\langle v_i,w_j\rangle).
\end{equation}
Assuming  bilinearity, the above formula defines uniquely a scalar product among $\ell$-forms.

Let $\omega_M$ be the Riemannian volume form on $M$. We define a duality from $\ell$ to $d-\ell$ forms via (see \cite[(2.1.6)]{Jost08})
\begin{equation}\label{eq:hdgedef}
\langle v,w\rangle \omega_M=(-1)^{\ell(d-\ell)}v\wedge *w=(-1)^{\ell(d-\ell)}w\wedge *v=*v\wedge w.
\end{equation}
Since such a formula must hold for all $\ell$-forms, the $(d-\ell)$-forms $*w, *v$ are uniquely defined. The operator ``$*$'' is the so called {\em Hodge operator}. 

Note that if $\{dx_i\}$ is an orthogonal base of $T_x^*M$, then $\{dx_{\bar i}=dx_{i_1}\wedge\cdots \wedge dx_{i_\ell}\}$ is an orthonormal base of $\wedge^\ell T^*M$, where $\bar i=(i_1,\dots, i_l)$ is an ordered multi-index (i.e. $i_k<i_{k+1}$). Then
\[
*dx_{\bar i}=\ve(\bar i^c)dx_{\bar i^c}
\]
where $\bar i^c$ is the ordered $(d-\ell)$-multiindex such that $dx_{\bar i}\wedge dx_{\bar i^c}=\ve(\bar i)\omega_M$ and $\ve(\bar i)$ is the sign of the permutation $\pi(1,\dots, d)=(i_1,\dots,i_\ell, i^c_{1},\dots,i^c_{d-\ell})$.

The above definitions and considerations imply the following equalities 
\begin{equation}\label{eq:hodge-fact}
\langle v,*w\rangle=(-1)^{\ell(d-\ell)}\langle *v,w\rangle\;;\;\; **v=(-1)^{\ell(d-\ell)}v\;;\;\;\langle *v,*w\rangle=\langle v,w\rangle .
\end{equation}
Note also that, for any smooth function $f$, $*(f w)=f*w$.
It is also natural to define the scalar product
\[
\langle v, w\rangle_{\Omega_\ell} \doteq \int_M\langle v,w\rangle_x \;\omega_M(x).
\]
Next, consider a smooth diffeomorphism $F:M\to N$, for $M,N$ Riemannian manifolds. Let $\det(DF)$ be determined by
$F^*\omega_N=\det(DF)\omega_M$, then 
\begin{equation} \label{hodgedualform}
\langle F^* v,w\rangle=(-1)^{(d-\ell)\ell} \det(DF) \cdot F^*\langle v,*(F^{-1})^**w\rangle.
\end{equation}
To prove this, we just compute 
\[
\begin{split}
\langle F^* v,w\rangle\omega_M&=**F^*v\wedge*w=F^*(v\wedge **(F^{-1})^**w)\\
&=(-1)^{\ell(d-\ell)}F^*(\langle v,*(F^{-1})^**w\rangle\omega_N)\\
&=(-1)^{\ell(d-\ell)}\det(DF) \cdot \langle v,*(F^{-1})^**w\rangle\circ F \cdot \omega_M.
\end{split}
\]
In particular, letting $(F^*)'$ be defined by $\langle F^* f,g\rangle_{\Omega^\ell}=\langle  f, (F^*)'g\rangle_{\Omega^\ell}$, we have
\[
*(F^*)'= (F^{-1})^**.
\]
The above formulae yield a formula for the Lie derivative: let $Z$ be a vector field  and $F_t$ the flow generated by $Z$. Then,
by differentiating $F_t^*\langle v,w\rangle=\det(DF)^{-1}\langle F_t^**v,*F_t^*w\rangle$ with respect to $t$ at $t=0$. It follows that
\begin{equation}\label{eq:lie-scalar}
\begin{split}
L_Z\langle v, w\rangle&=-\langle v, w\rangle\operatorname{div} Z +\langle L_Z(* v), *w\rangle+\langle v,L_Z w\rangle\\
&=-\langle v, w\rangle\operatorname{div} Z +(-1)^{d(d-\ell)}\langle *L_Z(* v), w\rangle+\langle v,L_Z w\rangle.
\end{split}
\end{equation}
We will also be using the relation (see  \cite[Proposition
2.2.19]{AbrahamMarsden87} for the details)
\begin{equation} \label{eq:AM2219} 
L_Z F^* w=F^*L_{F_*Z}w.
\end{equation}

Let $\gamma V $ be a rescaling of the vector field $V$.
 Recall that, by definition, 
$X^{(\ell)}h \doteq \left. \frac{d}{dt} \phi_{-t}^*h \right|_{t=0} = - L_V h $ for $h \in \Omega_{s}^{\ell}$. 
Hence for all $t \in \bR$ and $h \in \Omega_{0,s}^{\ell}$, by the properties of Lie derivatives (for $i_V$ as in Remark \ref{rmk:oldbpq}),
\begin{equation} \label{LieFlow} 
\begin{split}
  L_{\gamma V}(\cL_t^{(\ell)} h)  &= \gamma L_V(\cL_t^{(\ell)} h) +
 d\gamma \wedge i_V (\cL_t^{(\ell)} h) \\ 
& = - \gamma \left. \frac{d}{ds} \phi_{-s}^*  
(\phi_{-t}^* h ) \right|_{s=0}   = - \gamma \cL_t^{(\ell)}(X^{(\ell)}h). 
 \end{split}  
 \end{equation}

\section{Orientability}\label{app:orientable}

Here we prove some simple facts about orientability of the invariant distributions for Anosov flows and discuss how to modify the arguments of this paper in the non orientable case. We believe some of the following results to be well known but we could not locate a simple reference, so we add them for the reader's convenience.

\begin{lem} \label{lem:geor} For each geodesic flow on an orientable compact Riemannian manifold $M_0$ of negative sectional curvature, the unstable and stable distribution are orientable.
\end{lem}
\begin{proof}
We present  the proof for the unstable distribution, the proof  for the stable distribution being exactly the same by reversing time.
Remember that the geodesic flow takes place in $M\doteq T_1M$, the unitary tangent bundle. If we  assume that $M_0$ is $d$ dimensional, then $M$ is $2d-1$ dimensional and the unstable space is $d-1$ dimensional.

Given a geodesic $\gamma$, let $(J,J')\in T(TM)$ be a Jacobi fields along it. We have then \cite[Chapter 5.2]{doCarmo92}
\[
\frac{D^2 J}{dt^2}+R(\gamma'(t),J(t)))\gamma'(t)=0
\]
where $\frac D{dt}$, is the covariant derivative and $R$ is the curvature tensor. Recall that $\frac D{dt}J=J'$ and that we can assume that both $J(t)$ and $J'(t)$ are perpendicular to $\frac{d}{dt} \gamma(t)$. If we then define the quadratic form $Q(J,J')=\langle J,J'\rangle$, we have \cite[chapter 3.3]{doCarmo92},\footnote{ We let $K$ denote  the sectional curvature.}
\[
\begin{split}
\frac d{dt}Q(J,J')&=\langle J',J'\rangle-\langle J, R(\gamma'(t),J(t))\gamma'(t)\rangle\\
&=\langle J',J'\rangle-K(\gamma'(t),J(t)) \,\|\gamma'(t)\wedge J(t)\|^2>0.
\end{split}
\]
This means that the cone $\{(J,J')\in T(TM)\;:\; Q(J,J')>0\}$ is invariant under the flow, thus the unstable vectors must belong to such a cone. 
Given a point $q\in M_0$ and $(q,p)\in M$, let us consider the tangent space $T_{(q,p)}M$. Let $(\delta q,\delta p)\in E^u(q,p)\subset T_{(p,q)}M$. By the above discussion we have $\langle \delta q, p\rangle=\langle \delta p, p\rangle=0$ and that $(0,\delta p)\in E^u$ implies $\delta p=0$. Finally, let $E^{cu}(q,p)=\operatorname{span}\{E^u(q,p), V\}$, where $V=(p, 0)$ is the vector field generating the geodesic flow. Accordingly, if we denote  the canonical projection $\pi:M\to M_0$ defined by $\pi(q,p)=q$, we have that $\pi_*: E^{cu}(q,p)\to T_qM$ is an isomorphism. Indeed, if $(\delta q, \delta p)\in E^{cu}$, then we can write $\delta q=\beta p+\xi$ and $\delta p=\eta$, where $\beta\in \bR$, and $(\xi,\eta)\in E^{u}$. Accordingly, if $\delta q=\pi_* (\delta q, \delta p)=0$, then $\beta=0$ and $\xi=0$, which in turns imply $\eta=0$.

To conclude note that since $M_0$ is orientable, then there exists a volume form $\omega$ on $M_0$. But then $\overline \omega=i_V(\pi^*\omega)$ is a volume form on each $E^{u}$, therefore the bundle is orientable.
\end{proof} 

The above arguments remain true for more general flows (see \cite{LW1, LW3} for a discussion of appropriate invariant cones associated to contact flows or \cite{Lw2} for a slightly more general setting) but  the result is false for some Anosov flows.\footnote{ For example consider the map $f:\bT^2\to\bT^2$ defined by $f(x,y)=-A(x,y) \mod 1$ where 
\[
A=\begin{pmatrix} 1&1\\1&2\end{pmatrix}.
\]
One can then define a ceiling function $\tau$ and the associated suspension, and clearly the resulting flows is Anosov but the invariant distribution are non orientable.} Yet, it is often possible to keep track of the orientation of an orbit by a simple multiplicative factor as follows. 

\begin{lem} \label{lem:or-homo} If the torsion of the first homology group of $M$ does not contain factors of the type $\bZ_{2^m}$, then there exists a smooth function $A:M\times \bRp\to \bR$ such that, for each closed orbit $\tau$ and $x\in\tau$,
\[
e^{\pi i A(x,\lambda(\tau))}=\epsilon(\tau) .
\]
Moreover, $A_t(\cdot)=A(\cdot,t)$ is a cocycle.
\end{lem}
\begin{proof}
Consider the line bundle $\pi:\cF\to M$ such that the fiber at $x\in M$ consists of the volume forms on $E^u(x)$. Let $\cF_0=\cF\setminus \{(x,0)\}_{x\in M}$, i.e. we have taken out the zero section. Hence, for each $x\in M$,  $\pi^{-1}(x)$ has two connected component. Thus, for each $x\in M$, there exists a neighborhood $\cF_0\supset U\ni \pi^{-1}x$ and a continuous map $F :U\to \pi(U)\times \bZ_2$ such that, for each $\bar x\in U$, $p\circ F(\bar x)=\pi(\bar x)$, where $p(x,i)=x$, and $q(F(U))=\bZ_2$, where $q(x,i)=i$. This construction defines a double covering $M_{or}$ of $M$. Given a cycle $\gamma:[0,1]\to M$ we can then consider any lift $\tilde \gamma$ to $ M_{or}$. Locally $\tilde \gamma $ will have the form $\tilde \gamma(t)=(\gamma(t), i(t))$, $i\in\bZ_2$. We define then the degree map $\bar d(\gamma)= i(0)+i(1)$ which is well known to be a homotopy invariant. Accordingly, $\bar d$ depends only on the homology class of $\gamma$. In other words $\bar d\in H^1(M,\bZ_2)$. We would like to show that $\bar d\in H^1(M,\bR)$. 

Since $H_1(M,\bZ)$ is finitely generated it is isomorphic to $\bZ^n\oplus \text{Tor}$. By the universal coefficient Theorem $H_1(M,\bR)$ is isomorphic to $\bR^n$ while $H_1(M,\bZ_2)$ is isomorphic to $\bZ_2^{n+m}$ where $m$ is the number of $\bZ_{2^k}$ that are present in the torsion part. Thus by hypothesis $m=0$. Next, it is known that the closed orbits of an Anosov map generate all $H_1(M,\bZ)$, \cite{Adachi87}. Consider a set of closed periodic orbits $\Gamma=\{\gamma_i\}_{i\in \bN}$ that generate $H_1(M,\bZ)$, hence they generate $H_1(M,\bZ_2)$. We can thus consider a base  $\Gamma_0\subset\Gamma$ of $H_1(M,\bZ_2)$. By the above discussion it is a base in $H_1(M,\bR)$ as well. We can then consider the dual base in  $H^1(M,\bR)$. By De Rham cohomology such a base can be represented in terms of $1$-forms. In conclusion, there exist closed 1-forms $\{w_\gamma\}_{\gamma\in\Gamma_0}$ such that, for each $\gamma, \gamma'\in\Gamma_0$, $\int_{0}^{\lambda(\gamma')} \phi_t^*w_\gamma(V)$ is one if $\gamma=\gamma'$ 
and zero otherwise. Finally, let
\[
A(x,t)=\sum_{\gamma\in\Gamma_0}\bar d(\gamma) \int_0^t\phi_{s}^* w_{\gamma}(V(x)) ds.
\]
Note that $A$ is an additive cocycle with respect to the flow $\phi_t$. Moreover, if $\tau$ is a closed orbit and $x\in\tau$, 
$A(x,\lambda(\tau))\mod 2= \bar d(\tau)$. The Lemma follows since $\epsilon(\tau)=e^{\pi i \bar d(\tau)}$.
\end{proof}

In the hypothesis of the above Lemma we can define, instead of \eqref{def:reducedoperator}, the operator
\begin{equation}\label{eq:transferop0}
\cL^{(\ell)}_t(h) = e^{\pi i A_t } \tphi_{-t}^*h.
\end{equation}
The study of such an operator can be carried out in complete  analogy with what we have done in sections \ref{secresolvent}, \ref{sec:flat}, \ref{sec:splitting} (at the price of slightly heavier notation).\footnote{ The computations are essentially the same. In Section \ref{secresolvent}, \ref{sec:splitting}, \ref{sec:contact} one obtains the same bounds for the spectrum, while in Section \ref{sec:flat} keeping track of the cocycle yields the required factor $\epsilon(\tau)$.}

If the hypothesis of Lemma \ref{lem:or-homo} does not apply, then one can still introduce an appropriate weight, but this creates some difficulties that can be solved only by introducing a more sophisticated Banach space. The basic idea is to consider the action of $\phi_t^*$ also on an orientation bundle. This would be simple enough, but in so doing the action of the flows on such a bundle produces a multiplicative factor proportional to the stable Jacobian. To eliminate such an unwanted multiplicative factor it is then necessary to consider a transfer operator with a weight given exactly by the inverse of the stable Jacobian.\footnote{ The interested reader can consult Fried \cite[Section 6]{Fried95} for a thorough explanation.} Yet, this solution produces another problem: in general such a weight is only H\"older, hence the resulting operator does not act properly on spaces of smooth functions, a property essential to obtain large domains of analyticity for the zeta function.  To overcome this last 
obstacle in the present setting one has to change the Banach space. This can be done as in \cite{GouezelLiverani08}.

Let $\cG$ be the Grassmannian of $d_s$ dimensional subspaces of the tangent bundle. Next, consider the fiber bundle $\cE:=\{(x, E, h,\omega)\;:\;  (x, E)\in\cG,\, h\in \bigwedge^\ell \cT_x^*M,\, \omega\in\bigwedge^{d_s}E'\}$, $\pi:\cE\to\cG$ given by $\pi(x,E,h,\omega)=(x,E)$. 
The fibers consist of an $\ell$-form times a volume form (hence an orientation) on the subspace. We can then consider the vector space $\cS$ of the $\Cs^{r-1}$ sections of the bundle $\cE$. This generalizes the space $\Omega^\ell_r$. Given $(h,\omega)\in\cS$, $g\in\Omega^{\ell}_r$ and $W_{\alpha,G}\in\Sigma$ we can then define the integral
\[
\int_{W_{\alpha, G}} \langle g,h\rangle \omega=\int_{\bR^{(d_s)}}\langle g(G_\alpha(\xi)), h(G_\alpha(\xi), E_{G_\alpha})\rangle\, G_\alpha^*\omega
\]
where $G_\alpha\doteq \Theta^{-1}_\alpha\circ G$ and $E_{G_\alpha}=DG_\alpha(\bR^{(d_s)})$. We can then use the above to define the analogue of the functional $J_{\alpha,G, g,v_1,\dots,v_p}$ in section \eqref{eq:ldefinition} and then construct the norms in the same way. The transfer operator is then the canonical action of the flow on $\cS$ with a weight given by the Jacobian of the flow restricted to the element of the Grassmanian. Note that one obtains an operator of the type considered in \cite[equation (3.1)]{GouezelLiverani08}. Accordingly, all the present arguments can be repeated in complete analogy, although one must compute in the more sophisticated way explained in \cite{GouezelLiverani08}.\footnote{ To be precise, in \cite{GouezelLiverani08} there are no $\ell$-forms, but the point of the present paper is exactly that functions and $\ell$-forms can be handled by essentially the same computations once the proper machinery has been set up.}

 As the present paper is already quite technical, we decided not to go into such a more sophisticated construction as, on the one hand, it does not add any new idea, and  on the other it would make the presentation much harder to follow.  Nevertheless, all the results presented here can be generalized by using such an approach.

\section{Topological Entropy and volume growth} \label{app:topent}

It is well known that there is a relation between asymptotic volume growth of manifolds and the topological entropy (see \cite{yomdin87, gelfert05, Newhouse90} and references therein). Unfortunately, here we need a uniform upper and lower bound for all times. We are not aware of such bounds for flows in the literature. For Axiom A diffeomorphisms they can be obtained in great generality from the results in \cite{GouezelLiverani08}.\footnote{Indeed, in \cite{GouezelLiverani08} it is shown in particular that, given a mixing Axiom A diffeomorphism $(M,f)$, the spectral radius of a transfer operator with potential $\bar \phi$ satisfies, for each manifold $W\in\Sigma$, $C_\# e^{P(\bar\phi)n}\leq\left|\int_W e^{\sum_{k=0}^{n-1} \bar \phi\circ f^k}\right|\leq C_\# e^{P(\bar\phi)n}$, where $P(\bar\phi)$ is the pressure associated to the potential, see \cite[equations (3.3), (4.10), Lemma 4.7, Theorem 5.1.]{GouezelLiverani08}.} It is clear that applying the same strategy to the present setting of flows similar 
results can be obtained. Yet, since we are interested only in the topological entropy and not in an arbitrary potential, a more elementary approach is available.

Let us proceed in slightly greater  generality than needed, since it can be done at no extra cost.
Let $M$ be a Riemannian manifold and $(M,f)$ a partially hyperbolic diffeomorphism. Call $E^s,E^u, E^c$ the stable, unstable and central distribution respectively. Let $d_s$ be the dimension of $E^s$. 

Assume that, for each $n\in\bN$, $ \|Df^n|_{E^c}\|+\|Df^{-n}|_{E^c}\|\leq C_\#$. 

A set $S\subset M$ is called $(\ve,n)$-separated if for each $x,y\in S$ we have that  $d_{f,n}  (x,y ) > \ve$ where $d_{f,n} (x ,y ) = \max_{0 \leq k \leq n} d( f^k (x), f^k (y))$. Let $\cS(f,\ve, n)$ be the set of all $(\ve,n)$-separated sets. We then define $N(f,\ve,n) = \sup\limits_{S\in\cS(f,\ve,n)} \#S $. It is well known that, \cite{KH},
\[
h_{\textrm{top}}(f)=\lim_{\ve\to 0}\lim_{n\to\infty} \frac{\ln N(f,\ve,n)}{n}.
\]
Moreover since if $\ve'>\ve$, then $N(f,\ve,n)\geq N(f,\ve',n)$. Thus, for each $\delta>0$,
\begin{equation}\label{eq:h0}
h_{\textrm{top}}(f)\geq \lim_{n\to\infty} \frac{\ln N(f,\delta,n)}{n}.
\end{equation}

Fix $\delta>0$ and let $\widehat\Sigma$ be the set of all $\Cs^1$, $d_s$-dimensional manifolds of radius $\delta$ and tangent space close to $E^s$. Let us define
\[
\begin{split}
\rho^+_n=\sup_{W\in\hat\Sigma}\vol(f^n W) \qquad ; \qquad \rho^-_n=\inf_{W\in\hat\Sigma}\vol(f^n W).
\end{split}
\]
\begin{lem}\label{lem:h1}
For all $n, m\in\bN$,
\[
\rho^-_{n+m}\geq C_\# \rho^-_n\rho^-_m;\quad \rho^+_{n+m}\leq C_\# \rho^+_n\rho^+_m .
\]
In addition, if $f$ is topologically transitive, 
\[
\rho^+_n\leq C_\#\rho^-_n.
\]
\end{lem}
\begin{proof}
To prove the first equation consider $f^n(W)$ and chose a set $\cW_n\doteq\{W_i\}\in\widehat\Sigma$ such that $W_i\subset f^n(W)$, $W_i\cap W_j=\emptyset$ and the cardinality of $\cW_n$ is maximal. Then $\#\cW_n\geq C_\# \vol(f^n(W))$. Thus
\[
\vol (f^{n+m}(W))\geq \sum_{W'\in\cW_n} \vol(f^m(W'))\geq C_\#\rho^-_n\rho^-_m
\]
and the result follows taking the inf on $W$. The second inequality is proven similarly by taking a minimal cover of $f^n(W)$.

Let us prove the last inequality. By topological transitivity for each $\ve>0$ there exists $n_\ve\in \bN$ such that given any two ball $B,B'$ of radius $\ve$, there exists $n\leq n_\ve$ such that $f^{n}(B)\cap B'\neq \emptyset$.\footnote{ Take $S\in \cS(f,0,\ve/4)$ with maximal cardinality. Then $\cup_{x\in S} B(x,\ve)\supset M$, where $B(x,\ve)$ is a ball (in the Riemannian metric) or radius $\ve$ and center $x$. Then for each $x,y\in S$ there exists $n_{x,y}\in\bN$ such that $f^{n_{xy}}(B(x,\ve/4))\cap B(y,\ve/4)\neq \emptyset$. The claim follows since any ball $B$ of radius $\ve$ must contain a point $x\in S$ such that $B(x,\ve/4)\subset B$.} Let $\bar n=n_\delta$.

Next, for each $n\in \bN$ there exists $W_n\in \widehat\Sigma$ such that $\vol(f^{n-\bar n}(W_n))\geq \frac 12\rho^+_{n-\bar n}$. Let $x\in W_n$ and $B$ an open ball centered at $x$ and of radius $c\delta$ for some fixed $c\in (0,1)$ to be chosen later. Given any $W\in \widehat\Sigma$ let $z\in W$ and $B'$ a ball of radius $c\delta$ centered at $z$. By the above arguments there exists $m\leq \bar n$, such that $f^{m}(B')\cap B \neq \emptyset$. Given that the stable and unstable manifolds are uniformly transversal, it is possible to choose $c$ such that, for each point in $y\in W_n$, $W^{cu}_\delta(y)\cap f^{m}W\neq \emptyset$.\footnote{ By $W^{cu}_\delta(y)$ we mean a disk of radius $\delta$ in the weak unstable manifold.}  By the hypothesis on the central bundle it follows that the point related by the holonomy will never have a distance larger than $\delta$ in the future, then $\vol( f^k(W_n))\leq C_\#\vol( f^{k+m}(W))$. Accordingly
\[
\vol( f^{n}(W))\geq C_\#\vol( f^{n-m}(W_n))\geq C_\# \rho^+_{n-\bar n}\geq C_\# \rho^+_{n},
\]
and the result follows again by taking the inf on the $W.$
\end{proof}
\begin{lem}\label{lem:h2}
If $(M,f)$ is topologically transitive, then, for each $n\in\bN$, $\ve>0$,
\[
C_\ve \rho^+_n\leq N(f,\delta,n)\leq C_\ve\rho^+_n.
\]
\end{lem}
\begin{proof}
Let $W\in\widehat\Sigma$ and consider a $\ve$ separated set $S_n$ on $f^n(W)$.\footnote{ That is the distance among the points is larger than $\ve$ when measured on the manifold $f^n(W)$.} Since the manifold contracts backward, points that are far away in $f^n(W)$ but close in $M$ will separate backward. Thus  $f^{-n}S_n\in\cS(f,\ve,n)$. Since $\# S_n\geq C_\# \vol(f^n(W))$, it follows that $N(f,\ve,n)\geq  C_\ve\rho^+_n$.

To prove the second inequality we use the notation introduced in the proof of Lemma \ref{lem:h1}. Fix $W\in\widehat\Sigma$. Let $S\in\cS(f,\ve, n)$ be a set with maximal cardinality. Cover $M$ with balls of radius $c\ve$. By transitivity, for each ball $B$ there exists a time $n_B\leq \bar n$ such that $f^{n_B}W\cap B\neq \emptyset$. Let $W_B\subset f^{n_B}W$, $W_B\in\widehat\Sigma$ be a manifold that intersects the ball $B$ and project all the point in $S\cap B$ to $W_B$ via the weak unstable holonomy. If we consider the images of such points we note that two points can be $(\ve,n)$ separated only if their corresponding points on $W_B$ are $(C_\#\ve, n)$ separated. Finally, note again that the $(C_\#\ve,n)$ separated points on $W_B$ must have a distance larger that $C_\# \ve$ in the manifold $f^{n}(W_B)$. Thus there can be at most $C_\ve\vol(f^{n}(W_B))$ such points. Accordingly,
\[
N(f,\ve,n)\leq\sum_B C_\ve\vol(f^{n+ n_B}(W))\leq C_\ve\rho^+_n.
\]
\end{proof}
Using the above results the required bound follows.
\begin{lem}\label{lem:h3} 
If $(M,f)$ is topologically transitive, then for all $n\in\bN$ we have
\[
e^{h_{\textrm{top}}(f)n} \leq \rho^+_n\leq C_\# e^{h_{\textrm{top}}(f) n}.
\]
\end{lem}
\begin{proof}
By Lemmata \ref{lem:h2}, \ref{lem:h1} 
\[
\frac{\ln N(f,\ve,kn)}{kn}\leq \frac{\ln C_\ve \rho^+_{kn}}{kn}\leq \frac{\ln C_\ve}{kn}+\frac{\ln \rho^+_n}{n}.
\]
Taking  first the limit for $k\to \infty$ and  then the limit $\ve\to 0$, the first inequality of the Lemma follows.

By Lemma \ref{lem:h1}, it follows that $\rho^+_{n+m}\geq C_\# \rho^+_n\rho^+_m$. Then
$\rho^+_{kn}\geq C_\#^k (\rho^+_{n})^k$ and
\[
\ln[C_\#\rho^+_{n}]\leq m\lim_{k\to\infty}\frac{ \ln\rho^+_{kn}}{km}\leq n\lim_{k\to\infty}\frac{ \ln N(f,\delta,kn)}{kn}\leq  n\htop
\]
where we have used Lemma \ref{lem:h2} and equation \eqref{eq:h0}.
\end{proof}

\begin{rmk}\label{rem:h0}
Note that if $f$ is the time one map of a hyperbolic flow, then by the spectral decomposition Theorem \cite[Section 1.1]{BDV}, the manifold can be decomposed into finitely many isolated topologically transitive sets. Hence applying Lemma \ref{lem:h3} to each isolating neighborhood we obtain $C_\# e^{\htop n}\leq \rho^+_n\leq C_\# e^{\htop n}$ for each Anosov flow.
\end{rmk}


\section{Averaging operators}\label{subsec:average}

The following operators are used extensively in this paper.
\begin{defn}\label{def:molli}
Let $\kappa\in\Cs^\infty(\bR^d,\bRp)$ such that $\int \kappa(x) dx =1$,
$\operatorname{supp} \,\kappa \subset \{x\in\bR^d\,:\,\|x\|\leq 1\}$ and let
$\kappa_\ve(x) \doteq \ve^{-d}\kappa(\ve^{-1}x)$.  For each $\alpha\in \cA$, $\ell\in\{0,\dots, d \}$, let $\Psi \in \Cs^\infty_0(\bR^{2d}, GL(\binom d \ell,\bR))$ with  $\operatorname{supp} \Psi \subset B_d(0,2\delta)^2$.  We define\footnote{ Obviously our definition depends on the choice of the function $\kappa$, and in the future some special choice of $\kappa$ will be made. However, we choose not to explicitly show the dependence as a subscript to simplify the notation.}
the operator $\bM_{\alpha,\Psi,\ve}:\Omega^\ell_r  \to\Omega^\ell_r $  by
\begin{equation*}
\bM_{\alpha, \Psi, \ve}(h)(x)\doteq \left\{
\begin{array}{lr} 
0 & \text{ if } x \not\in U_\alpha \phantom{.}\\
\sum_{\bar i, \bar j \in \cI_\ell}\big[\int_{\bR^{d}}\Psi(\Theta_\alpha(x),y)_{\bar i,\bar j}\kappa_\ve(\Theta_\alpha(x)-y) &  \\ 
\qquad \qquad  \qquad \times\langle \omega_{\alpha, \bar j},h\rangle_{\Theta_\alpha^{-1}(y)} dy\big]\omega_{\alpha,\bar{i}} & \text{ if } x \in U_\alpha. \\
\end{array} \right.
\end{equation*}
Given the particular  choice $\overline{\Psi}(x,y)_{\bar i,\bar j} \doteq \psi_\alpha(\Theta_{\alpha}^{-1}(x)) \delta_{\bar i,\bar j}$ we define
\[
\bM_\ve \doteq \sum_\alpha \bM_{\alpha, \ve} \doteq \sum_\alpha \bM_{\alpha,\overline{\Psi},\ve}.  
\]
Lastly, given the duality induced by \eqref{eq:scalar0}, we define the operators $\bM_{\alpha,\ve}', \bM_{\ve}'$, as those such that 
 for all $h,g\in\Omega^\ell_r$ we have\footnote{ By the duality relation $\bM'$ is then also defined on  currents. In addition, note that $\bM_{\alpha,\ve}\Omega^\ell_{0,r}\cup \bM_{\alpha,\ve}'\Omega^\ell_{0,r}\subset \Omega^\ell_{0,r}$.}
\[
\langle \bM_{\alpha,\ve}' h,g\rangle_{\Omega^\ell_r}=\langle h,\bM_{\alpha,\ve} g\rangle_{\Omega^\ell_r}\ \hbox{ and }\quad
\langle \bM_{\ve}' h,g\rangle_{\Omega^\ell_r}=\langle h,\bM_{\ve} g\rangle_{\Omega^\ell_r}.
\]
\end{defn}

\begin{lem}\label{lem:molli} There exists $\ve_0>0$ such that, for each 
 $\ve\in (0,\ve_0)$, $\ell\in\{1,\dots,d \}$, $q>0, p\in\bN^*$ and $h\in \Omega^\ell_{p+q}$,  we can bound
\[
\begin{split}
&\|\bM_\ve h\|_{-,p,q,\ell}+\|\bM_\ve' h\|_{-,p,q,\ell}\leq C_{p,q,\kappa} \|h\|_{+,p,q,\ell},\\
&\|\bM_\ve h\|_{\Cs^p(\Omega, M)}+\|\bM_\ve' h\|_{\Cs^p(\Omega,M)}\leq C_{p,q,\kappa} \ve^{-d_s-p-q}\|h\|_{-,p,q,\ell},\\
&\|h-\bM_\ve h\|_{-,p-1,q+1,\ell}+\|h-\bM_\ve' h\|_{-,p-1,q+1,\ell}\leq C_{p,q,\kappa} \ve\|h\|_{+,p,q,\ell}.
\end{split}
\]
\end{lem}
\begin{proof}
We will give a proof only for the operator $\bM_\ve$, the proof for $\bM_\ve'$ being similar. 
To prove the first inequality we begin by estimating the integral in  \eqref{eq:ldefinition} 
for the case $p = 0$. First of all note that, for  each $\alpha,\beta \in \cA, W_{\beta, G}\in \Sigma_\alpha(\rho_,L_0)$ and $ \Psi$, if
\begin{equation} \label{eq:leafover} 
\int_{W_{\beta,G}}\langle g, \bM_{\alpha,\Psi, \ve} h\rangle \volform \neq 0,
\end{equation}
then $W_{\beta,G}\cap \whU_\alpha\neq \emptyset$. Thus there exists $W_{\alpha, G'}\in\Sigma_{\alpha}(\rho_+,L_+)$ such that $W_{\beta,G}\cap \whU_\alpha\subset W_{\alpha, G'}$ (see Remark \ref{rem:rho}). On the other hand, by Definition \ref{def:molli}, the integrand is supported in $\whU_\alpha$, thus
\[
\int_{W_{\beta,G}}\langle g, \bM_{\alpha,\Psi, \ve} h \rangle \volform =\int_{W_{\alpha,G'}}\langle g, \bM_{\alpha,\Psi, \ve} h\rangle \volform.
\]
Thus, it suffices to compute
\[
\begin{split}
\int_{W_{\alpha,G_{x,F}}}\langle g, \bM_{\alpha,\Psi,\ve}h\rangle \, 
 \volform  &=\sum_{\bar i, \bar j\in\cI_\ell}
\int_{B_{d_s}(0,2\delta)} \hskip-.5cm d\xi\int_{\bR^d}dy \, \Psi(G_{x,F}(\xi),y)_{\bar i,\bar j}\\
&\times  g_{\alpha, \bar i}\circ\Theta_\alpha^{-1}\circ G_{x,F}(\xi)  h_{\alpha,\bar{j}}\circ \Theta_\alpha^{-1}(y) \kappa_\ve( G_{x,F}(\xi)-y) ,
\end{split}
\]
where $h_{\alpha,\bar{i}}(z)=\langle h,\omega_{\alpha, \bar i}\rangle_z$  and $g_{\alpha,\bar{i}}(z)=\langle g,\omega_{\alpha, \bar i}\rangle_z$ for each $z\in U_\alpha$. 

Next, we keep $\xi\in B_{d_s}(0,2\delta)$ fixed and we consider the change of variables $\zeta= G_{x,F}(\xi)-y$. By Fubini-Tonelli theorem and the definition of $G_{x ,F}$ we can exchange the integrals and obtain 
\begin{equation}\label{eq:aver-comp2}
\begin{split}
 \int_{W_{\alpha,G_{x,F}}}\hskip-.6cm\langle g, \bM_{\alpha,\Psi,\ve}h \rangle \volform
&= \sum_{\bar i, \bar j\in\cI_\ell}\int_{\bR^d} d\zeta \, \kappa_\ve(\zeta)  \int_{B_{d_s}(0,2\delta)}\hskip-1cm d\xi \, \Psi(G_{x,F} (\xi),G_{x-\zeta,F} (\xi))_{\bar i,\bar j}  \\
&   \times  g_{\alpha,\bar i}\circ\Theta_\alpha^{-1}\circ G_{x,F} (\xi)\, h_{\alpha,\bar j}\circ
\Theta_\alpha^{-1} \circ G_{x-\zeta,F} (\xi) .
\end{split}
\end{equation}
Now given our admissible test function $g$ and a leaf $W_{\alpha, G_{x,F}}$, we can define admissible elements $g_{\zeta,\Psi}$ and $W_{\alpha, G_{x_{\zeta},F}}$, as follows.
Let  $x_\zeta \doteq x-\zeta$, $g_{\zeta,\Psi,\bar j}\circ\Theta_\alpha^{-1}(z) \doteq \sum_{\bar i} \Psi(z+\zeta,z)_{\bar i,\bar j}g_{\alpha,\bar i}\circ \Theta_\alpha^{-1}(z+\zeta) $ and $g_{\zeta, \Psi}=\sum_{\bar i}g_{\zeta,\Psi,\bar i} \omega_{\alpha,\bar i}$, so that 
\[
\sum_{\bar i} \Psi(G_{x,F} (\xi),G_{x-\zeta,F} (\xi))_{\bar i,\bar j}g_{\alpha,\bar i}\circ\Theta_\alpha^{-1}\circ G_{x,F}(\xi)=g_{\zeta,\Psi, \bar j}\circ \Theta_\alpha^{-1}\circ G_{x_\zeta, F}(\xi).
\]
Then
\begin{equation}\label{eq:average-rep}
\int_{W_{\alpha,G_{x,F}}}\langle g,   \bM_{\alpha, \Psi,\ve}h \rangle \volform =\int_{\bR^d} d\zeta
\kappa_\ve(\zeta)\int_{W_{\alpha, G_{x_\zeta,F}}}\langle g_{\zeta,\Psi}, h \rangle
\volform . \end{equation}
Note that, since the cones are constant in the charts, $G_{x_{\zeta},F}\in \Sigma(\rho_+,L_+)$ for each $\zeta$ small enough.
Since,   $\|g_{\zeta,\Psi}\|_{\Gamma^{\ell,q}_c(\alpha, G_{x_\zeta,F})}\leq C_{q}\|\Psi\|_{\Cs^q}\|g\|_{\Gamma^{\ell,p+q}_c(\alpha, G_{x,F})}$ we have
\[
\|\bM_{\alpha,\Psi,\ve}h\|_{-,0,q}\leq C_{q}\|\Psi\|_{\Cs^q}\|h\|_{+,0,q}.
\]
The first inequality of the Lemma, for $p=0$, follows.

Next we treat the case $p>0$. Let $v\in\Cs^{p+q}$.
We need to compute $L_v \bM_{\alpha,\Psi,\ve}$. To do so note that $L_v (\omega_{\alpha,\bar{i}}) = \sum_{\bar{j}} \rho_{\bar{i},\bar{j}}(v) \omega_{\alpha,\bar{j}}$.
Moreover by equation \eqref{eq:AM2219} there exists $\Psi'_v$ such that we have
$L_v (\Psi(\Theta_\alpha(x),y))_{\bar{i},\bar{j}}  \doteq \Psi'_v( \Theta_\alpha(x),y))_{\bar{i},\bar{j}}$.  Next, we deal with $L_v ( \kappa_\ve(\Theta_\alpha(x)-y))$  using integration by parts with respect to $y$. Lastly, we apply \eqref{eq:lie-scalar} to  $L_v  \langle \omega_{\alpha, \bar j},h\rangle_{\Theta_\alpha^{-1}(y) }$.  

By using all the remarks above,  we obtain that there exists  vectors $w_1(v), \ldots ,w_l(v)\in\Cs^{p+q}$, and functions $\Psi_{v}, \Psi_{v,1} , \ldots ,  \Psi_{v,l}$ such that $\|\Psi_v\|_{\Cs^{p+q-1}}+\sum_l\|\Psi_{v,l}\|_{\Cs^{p+q}}\leq C_\# \|\psi\|_{\Cs^{p+q}}$ and for which we have
\begin{equation}\label{eq:lie-molli}
L_v \bM_{\alpha,\Psi,\ve}h=\bM_{\alpha,\Psi_{v},\ve}h+\sum_l \bM_{\alpha,\Psi_{v,l},\ve}(L_{w_l} h).
\end{equation}
By iterating the above equality, using the previous bound for $p=0$ and equation \eqref{eq:0q-norm}, we obtain the first inequality of the Lemma.

To prove the second inequality, given definition \eqref{def:manifoldnorm}, we must evaluate
\begin{equation} \begin{split} \label{eq:molli2}
&\langle \omega_{\alpha,i}, \bM_{\alpha,\ve}h \rangle_{\Theta^{-1}_\alpha(z)}=  
\psi_\alpha\circ \Theta^{-1}_\alpha(z)\left[\int_{\bR^{d}}\kappa_\ve(z-y)\langle \omega_{\alpha, \bar i},h\rangle_{\Theta_\alpha^{-1}(y)} dy\right] \\ 
& = \int_{\bR^{d_u+1}}dy_{d_s+1}\dots dy_d\int_{W_{\alpha,G_{(0,y_{d_{s}+1},\ldots,y_d),0}}} \langle g_{\ve, z,y},h\rangle \volform. \\
\end{split}
\end{equation}
where  $\hat g_{\ve, z,y, \bar i}\circ \Theta_{\alpha}^{-1}(\xi)=\psi_\alpha\circ \Theta^{-1}_\alpha(z)\kappa_\ve(z-(\xi,y_{d_s+1},\dots,y_d) )$ and  $g_{\ve, z,y}=\hat g_{\ve, z,y, \bar i} \omega_{\alpha,\bar i}$.
Then, differentiating \eqref{eq:molli2} with respect to $z$ and integrating by parts yields the desired result since $\|g_{\ve, z,y}\|_{\Gamma^{\ell, p+q}_c}\leq C_\# \ve^{-p-q}$.

Finally note that by construction 
\[
\lim_{\ve\to 0}[\bM_{\alpha,\Psi,\ve}h](x) =\Psi(\Theta_\alpha(x),\Theta_\alpha(x)) h(x)\doteq \widehat{\Psi}_{\alpha}(x) h(x). 
\]
Thus, by \eqref{eq:lie-molli}, we have that $\lim_{\ve\to 0}L_v\bM_{\alpha,\Psi,\ve}h$ exists. On the other hand, for all $g\in\Cs^r$, we have (see \eqref{eq:stokes})
\[
\lim_{\ve\to 0} \int_M g L_v\bM_{\alpha,\Psi,\ve}h\;\omega =\lim_{\ve\to 0} (-1)^d \int_M \bM_{\alpha,\Psi,\ve}h d(i_v(g\omega))=\int_M g L_v[\widehat{\Psi}_{\alpha} h].
\]
Thus,
\begin{equation} \label{eq:lie-comm} 
\lim_{\ve\to 0}L_v\bM_{\alpha,\Psi,\ve}h =L_v[\widehat{\Psi}_{\alpha} h].
\end{equation}
Hence, by \eqref{eq:average-rep} and the explicit representation  of \eqref{eq:aver-comp2}, we can write 
\begin{equation} \label{eq:molli3} 
\begin{split}
& \int_{W_{\alpha, G_{x,F}}}\volform \left[\langle g, \widehat{\Psi}_{\alpha} h\rangle-\langle g, \bM_{\alpha,\Psi,\ve} h\rangle \right]  \\
&= \int d\zeta  \kappa_\ve(\zeta)\int_0^1 dt\frac d{dt}\int_{W_{\alpha,G_{x_{t\zeta},F}}} \hskip-1cm\volform 
\langle g_{t\zeta, \Psi}, h\rangle  \\
&= \int d\zeta \kappa_\ve(\zeta)\int_0^1 dt\int_{W_{\alpha,G_{x_{t\zeta},F}}}\hskip-1cm \volform \langle g_{t\zeta,\Psi},L_{v_\zeta} h\rangle-\sum_{\bar i}\langle g_{t\zeta,\Psi}, \hat h_{\alpha,\bar i}L_{v_\zeta}\omega_{\alpha,\bar i}\rangle  \\
\end{split} 
\end{equation}
where $v_\zeta=-(\Theta^{-1}_{\alpha})_*\zeta$.  The last inequality, for $p=0$, follows since $\|v_\zeta\|_{\Cs^{p+q}}\leq C_\#\ve$. To obtain the result for $p>0$, one can use \eqref{eq:lie-molli} and \eqref{eq:lie-comm} repeatedly.
\end{proof}

The above Lemma has a useful corollary.
\begin{cor}\label{lem:molli-app-s} Both $\bM_\ve$ and $\bM_\ve'$ extend to operators in $L(\cB^{p,q,\ell}_+,\cB^{p,q,\ell}_-)$ and $L((\cB^{p,q,\ell}_-)',(\cB^{p,q,\ell}_+)')$. In addition, if $h\in \cB^{p,q,\ell}_+$, then 
\begin{equation}\label{eq:molli-app}
\lim_{\ve\to 0}\|h-\bM_\ve h\|_{-,p,q,\ell}+\|h-\bM_\ve' h\|_{-,p,q,\ell}=0.
\end{equation}
\end{cor}
\begin{proof}
By definition, there exists a sequence $\{h_n\}\subset \Omega_{s}^\ell$ that converges to $h$ in $\cB_+^{p,q,\ell}$. Hence
\[
\begin{split}
\|h-\bM_\ve h\|_{-,p,q,\ell}&\leq 
\|h-h_n\|_{-,p,q,\ell}+\|h_n-\bM_\ve h_n\|_{-,p,q,\ell}+\|\bM_\ve(h_n- h)\|_{-,p,q,\ell}\\
&\leq C_\#\|h-h_n\|_{+,p,q,\ell}+ C_\#\ve \|h_n\|_{+,p+1,q-1,\ell}
\end{split}
\]
implies the result.
\end{proof}

\begin{rmk}\label{rem:loss} The loss related to the $\pm$ norms is due to the need of viewing a manifold of $\Sigma_\alpha$ in a different chart. If we would consider a $\Psi$ supported in $U_\alpha$ and be interested only in the semi-norms arising from $\Sigma_\alpha$, such a loss would disappear.
\end{rmk}
\section{Holonomies}\label{app:holo-est}

Here we recall some known facts concerning the invariant foliations of an Anosov flow and prove some useful estimates. Concerning  the issue of regularity,  more details can be found in \cite{Hasselblatt94, Has1, Hasselblatt97}.

The strong unstable foliation can be locally trivialized by a change of coordinates of the form $\bH(\xi,\eta)=(H(\xi,\eta),\eta)$, $\eta\in B_{d_u}(0,\delta)$, $\xi\in B_{d_s+1}(0,\delta)$, such that for each $\xi$, $\{\bH(\xi,\eta)\}_{\eta\in B_{d_u}(0,\delta)}$ is a strong unstable manifold. In addition, without loss of generality, we can require that $H(0,\eta)=0$ and $H(\xi,0)=\xi$. The results on the regularity of the foliation can be summarized by considering the map $\tilde H:B_{d_s+1}(0,\delta)\to \Cs^0(B_{d_u}(0,\delta), \bR^{d_s+1})$ defined by $[\tilde H(\xi)](\cdot)=H(\xi, \cdot)$. Then $\tilde H$ has range in $\Cs^r(B_{d_u}(0,\delta), \bR^{d_s+1})$ and $\|\tilde H\|_{\Cs^\varpi(B_{d_s+1}(0,\delta), \Cs^r(B_{d_u}(0,\delta), \bR^{d_s+1}))}\leq C_\#$. Moreover for each $\eta$, the function $H(\cdot, \eta)$ is absolutely continuous and has Jacobian $JH(\xi, \eta)$. Again, we can define $J\tilde H:B_{d_s+1}(0,\delta)\to \Cs^0(B_{d_u}(0,\delta), \bR )$ by $[J\tilde H(\xi)](\cdot)=JH(\xi, \cdot)$. 
Then $J\tilde H$ has range in $\Cs^r(B_{d_u}(0,\delta), \bR)$ and $\|J\tilde H\|_{\Cs^{\varpi^*}(B_{d_s+1}(0,\delta), \Cs^r(B_{d_u}(0,\delta), \bR))}\leq C_\#$.\footnote{ See \cite{Hasselblatt97} and referencess therein for the proofs of these statements.}

In this paper the holonomies are often used as follows: given $W_1, W_2\in \Sigma$ and $x,y\in W_1$ estimate $d(H_{\widetilde W_1,\widetilde W_2}(x),H_{\widetilde W_1,\widetilde W_2}(y))$.\footnote{ Remember that $\widetilde W=\cup_{t\in I}\phi_t W$, for some appropriate interval $I$.}  Although it is clear how to proceed,  we give the details for the benefit of the lazy reader. 

Note that, without loss of generality we can assume that $x=0$, $y=(a,0)$ and $H_{\widetilde W_1,\widetilde W_2}(0)=(0,z)$, $\widetilde W_1=\{(\xi,0)\}_{\xi\in\bR^{d_s+1}}$. Then $\widetilde W_2$ will have the form $\{(\xi,z+ G(\xi))$, $G(0)=0$, and $\|D_\xi G\|\leq 2$. Also let $H_{ \widetilde W_1, \widetilde W_2}(a,0)=(\xi,\eta)$,
\begin{lem}\label{lem:holo-dist} We have that
 \begin{equation}\label{eq:holo-dist}
\|(\xi,\eta)-(a,z)\|\leq C_\#\left[  \|a\|^{\varpi'}\|z\|+\|\partial_\xi G\|_{\Cs^0}\|a\|\right].
\end{equation}
\end{lem}
\begin{proof}
The intersection point $(\xi,\eta)$ satisfies $(\xi,\eta)=(H(a,\eta),z+G(\xi))$. Thus 
\[
\begin{split}
\xi-a&=H(a,\eta)-H(a,0)\\
&=H(a, z+G(\xi))-H(a, z+G(a))+H(a, z+G(a))-H(a,0)\\
&=A(\xi-a)+H(a, z+G(a))-H(a,0).
\end{split}
\]
where the last line is an application of the intermediate value theorem and $\|A(\xi)\|\leq \|\partial_\eta H(a,\cdot)\|_{\Cs^0}\,\|\partial_\xi G\|_{\Cs^0}\leq C_\# \|a\|^{\varpi'} \|\partial_\xi G\|_{\Cs^0}$.
Then 
\[
\|\xi-a\|\leq \|(\Id-A)^{-1}(H(a, z+G(a))-H(a,0))\|\leq C_\# \|a\|^{\varpi'}(\|z\|+ \|\partial_\xi G\|_{\Cs^0}\|a\|).
\]
Analogously,
\[
\eta-z=G(H(a,\eta))-G(H(a,z))+G(H(a,z))=A_1(\eta-z)+G(H(a,z)),
\]
where $\|A_1(\eta)\|\leq \|\partial_\xi G\|_{\Cs^0}\|\partial_\eta H(a,\cdot)\|_{\Cs^0}\leq C_\#  \|a\|^{\varpi'}\|\partial_\xi G\|_{\Cs^0}$. Thus
\[
\|\eta-z\|\leq \|(\Id-A_1)^{-1}G(H(a,z))\|\leq C_\#\|\partial_\xi G\|_{\Cs^0}(\|a\|+\|a\|^{\varpi'}\|z\|).
\]
\end{proof}

Next, we need a small improvement of \cite[Lemma B.7]{Liverani04}. We use the same notation established before \eqref{eq:temporal}, a part form the fact that the points $\xi,\zeta$ are fixed to some value $\xi_*,\zeta_*$.
\begin{lem}\label{lem:temporal}
For $\xi_*, \zeta_*\in B_{d_s+1}(0,\delta)$ and $\Delta^*$ as in \eqref{eq:temp-def} we have that 
\[
\begin{split}
\left|\Delta^*(\xi_*+\zeta_*)-\Delta^*(\xi_*)- d\alpha_0(z(\xi_*),\zeta_*)\right|\leq& C_\#(\|z\|^2\|\tilde\zeta_*\|^{\varpi'}+\|z\|^{\varpi'}\|\tilde\zeta_*\|^2)\\
&+C_\#\|z\|^{1+\varpi'}\|\tilde\zeta_*\|^{1+\varpi'}.
\end{split}
\]
\end{lem}
\begin{proof}
We use the same coordinates defined before Lemma \ref{lem:holo-dist} with $W=W_1$, $W'=W_2$; so that $\xi_*=0$, $\zeta_*=(\tilde\zeta, 0)$.\footnote{ This can be done without loss of generality since the general case follows by a trivial translation.} Thus $\widetilde W$ is a flat manifold and $\widetilde W'$ is the graph of $(\xi, G(\tilde \xi))$, where $\tilde\xi=(\xi_1,\dots\xi_{d_s})$, $\xi_{d_s+1}$ being the flow direction. Also $W=\{(\tilde\xi,0)\}$ while $W'$ is the graph of $(\tilde \xi, G(\tilde \xi), L(\tilde \xi))$ with $L(0)=0$, $G(0)=\bar z=z^u(\xi_*)$. Let $H(\xi,\eta)=(H_s(\xi,\eta),H_0(\xi,\eta))$, $H_s:\bR^{d}\to\bR^{(d_s)}, \,H_0:\bR^d\to\bR$, and $\{(I_u(\zeta), I_s(\zeta), I_0(\zeta))\}=\{\bH(\zeta,\eta)_{\eta\in\bR^{d_u}}\}\cap\{(\xi, G(\xi))_{\xi\in\bR^{d_s+1}}\}$. Then $\Delta^*(\xi_*+\zeta_*)-\Delta^*(\xi_*)=I_0(\tilde \zeta)- L(I_s(\tilde \zeta))$.
Define $\Xi:[0,1]^2\to \bR^d$ by
\[
\Xi(t,s)=(sH_s(\tilde \zeta, t I_u(\tilde\zeta)), t G( sI_s(\tilde \zeta)), s[H_0(\tilde \zeta, tI_u(\tilde\zeta))- tI_0(\tilde\zeta)]+tL(sI_s(\tilde \zeta)))
\]
and set $\Xi([0,1]^2)\doteq\Sigma$. From \cite[equations (B.4), (B.5)]{Liverani04} it follows that $I_0(\tilde \zeta)- L(I_s(\tilde \zeta))$ equals the integral of the symplectic form over $\Sigma$, hence
\[
\begin{split}
&\Delta^*(\xi_*+\zeta_*)-\Delta^*(\xi_*)=\int_\Sigma d\alpha_0=\int_{[0,1]^2}\Xi^*d\alpha_0=\int_{[0,1]^2} d\alpha_0(\partial_s\Xi,\partial_t\Xi)\, dt ds\\
&=\int_{[0,1]^2} \langle H_s(\tilde\zeta, tI_u(\tilde \zeta)), G(sI_s(\tilde\zeta))\rangle- ts\langle \partial_{\tilde \xi} G(sI_s(\tilde \zeta)I_s(\tilde \zeta),\partial_\eta H_s(\tilde\zeta, t I_u(\tilde\zeta))I_u(\tilde\zeta)\rangle\\
&=\langle \tilde\zeta,\bar z\rangle+\cO\left( \|\tilde\zeta\|^{\varpi'}\|\bar z\|\|I_u\|+\|\tilde\zeta\|\|\bar z\|^{\varpi'}\|I_s\|+\|\tilde\zeta\|^{\varpi'}\|\bar z\|^{\varpi'}\|I_u\|\|I_s\|\right).
\end{split}
\]
Since, by Lemma \ref{lem:holo-dist}, $\|I_u-\bar z\|+\|I_s-\tilde\zeta\|\leq C_\# (\|\bar z\|^{\varpi'}\|\tilde\zeta\|+\|\bar z\|\|\tilde\zeta\|^{\varpi'})$, the Lemma follows.
\end{proof}

\addcontentsline{toc}{section}{References}
\footnotesize
\bibliographystyle{plain}
\bibliography{bibliografia-final}

\end{document}